\documentclass{article}       %
\RequirePackage{fix-cm}
\usepackage{graphicx}

\usepackage{color}
\usepackage{amssymb}
\usepackage{amsmath}
\usepackage{amsthm}
\usepackage{mathrsfs}
\usepackage[authoryear]{natbib}

\usepackage{tikz}
\usepackage{hyperref}
\usepackage{xcolor}\hypersetup{allcolors=blue,colorlinks=true}

\newcommand{\R}{\mathbb{R}}
\newcommand{\Z}{\mathbb{Z}}
\newcommand{\var}{\mathbf{var\,}}
\newcommand{\E}{\mathbf{E}}
\renewcommand{\P}{\mathbf{P}}
\newcommand{\I}{\mathbb{I}}
\newcommand{\J}{\mathbb{J}}
\newcommand{\D}{\mathcal{D}}

\newcommand{\chck}{{\color{red}[check]}}
\renewcommand{\chck}{}

\newcommand{\revrem}[1]{#1}
\renewcommand{\revrem}[1]{}
\newcommand{\revadd}[1]{{\color{red}{#1}}}
\renewcommand{\revadd}[1]{{\color{black}{#1}}}

\DeclareMathOperator*{\argmin}{arg\,min}
\DeclareMathOperator{\arcsinh}{arcsinh}

\numberwithin{equation}{section}
\theoremstyle{plain}
\newtheorem{definition}{Definition}
\newtheorem{theorem}{Theorem}[section]
\newtheorem{lemma}{Lemma}[section]
\newtheorem{proposition}{Proposition}[section]
\newtheorem{corollary}{Corollary}[section]
\newtheorem{remark}{Remark}

\begin{document}

\title{Sample Path Large Deviations for L\'evy Processes and Random Walks with Regularly Varying Increments
}


\author{Chang-Han Rhee$^{1,3}$
        \and
        Jose Blanchet$^{2,4}$
        \and
        Bert Zwart$^{1,3}$
}
\footnotetext[1]{Stochastics Group, Centrum Wiskunde \& Informatica, Science Park 123, Amsterdam, North Holland, 1098 XG, The Netherlands}
\footnotetext[2]{Management Science \& Engineering, Stanford University, Stanford, CA 94305, USA
}
\footnotetext[3]{
Supported by NWO VICI grant.  
}
\footnotetext[4]{
Supported by NSF grants DMS-0806145/0902075,  CMMI-0846816 and CMMI-1069064.
}



\date{\today}

\maketitle

\begin{abstract}
Let $X$ be a L\'evy process with regularly varying L\'evy measure $\nu$. We obtain sample-path large deviations for scaled processes $\bar X_n(t) \triangleq X(nt)/n$ and obtain a similar result for random walks.
Our results yield detailed asymptotic estimates in scenarios where multiple big jumps in the increment are required to make a rare event happen; we illustrate this through detailed conditional limit theorems. In addition, we investigate connections with the classical large deviations framework. In that setting, we show that a weak large deviation principle (with logarithmic speed) holds, but a full large deviation principle does not hold.


\end{abstract}

\section{Introduction}
In this paper, we develop sample-path large deviations for one-dimensional L\'evy processes and random walks, assuming the jump sizes are heavy-tailed.
Specifically, let $X(t),t\geq 0,$ be a centered L\'evy process with regularly varying L\'evy measure $\nu$. Assume that $\P(X(1)>x)$ is regularly varying of index $-\alpha$, and that $\P(X(1)<-x)$ is regularly varying of index $-\beta$; i.e.
there exist slowly varying functions $L_{+}$ and $L_{-}$ such that
\begin{equation}
\label{intro-eq-twosided}
\P(X(1)>x) = L_{+} (x) x^{-\alpha}, \hspace{1cm} \P(X(1)<-x) = L_{-}(x) x^{-\beta}.
\end{equation}
Throughout the paper, we assume $\alpha, \beta>1$. We also consider spectrally one-sided processes; in that case only $\alpha$ plays a role.
Define $\bar X_n= \{\bar X_n(t), t\in [0,1]\}$, with $\bar X_n(t) = X(nt)/n, t\geq 0$. We are interested in large deviations of $\bar X_n$.

This topic fits well in a branch of limit theory that has a long history, has intimate connections to point processes and extreme value theory, and is still a subject of intense activity.
The investigation of tail estimates of the one-dimensional distributions of $\bar X_n$ (or random walks with heavy-tailed step size distribution) was initiated in \cite{Nagaev69,Nagaev77}. The state of the art of such results
is well summarized in \cite{BorovkovBorovkov, DiekerDenisovShneer, EmbrechtsKluppelbergMikosch97, FKZ}. In particular, \cite{DiekerDenisovShneer} describe in detail how fast $x$ needs to grow with $n$ for the asymptotic relation
\begin{equation}
\label{onebigjump}
\P( X(n) > x)  = n \P(X(1)>x)(1+o(1))
\end{equation}
to hold, as $n\rightarrow\infty$, in settings that go beyond  (\ref{intro-eq-twosided}). If (\ref{onebigjump}) is valid, the so-called \emph{principle of one big jump} is said to hold. A functional version of this insight has been derived in \cite{HLMS}. A significant number of studies investigate the question of if and how the principle of a single big jump is affected by the impact of (various forms of) dependence, and cover stable processes, autoregressive processes, modulated processes, and stochastic differential equations; see \cite{BuraDamekMikosch13, FossModulated, HultLindskog07,  KonstantinidesMikosch05, MikoschWintenberger13, mikosch2016large, MikoschSamorodnitsky00, Samorodnitsky04}.

The problem we investigate in this paper is markedly different from all of these works. Our aim is to develop asymptotic estimates of $\P(\bar X_n \in A)$ for a sufficiently general collection of sets $A$, so that it is possible to study continuous functionals of $\bar X_n$ in a systematic manner. For many such functionals, and many sets $A$, the associated rare event will not be caused by a single big jump, but multiple jumps. The results in this domain (e.g. \cite{blanchetrisk, FK12, ZBM}) are few, each with an ad-hoc approach. As in large deviations theory for light tails, it is desirable to have more general tools available.

Another aspect of heavy-tailed large deviations we aim to clarify in this paper is the connection with the standard large-deviations approach, which has not been touched upon in any of the above-mentioned references. In our setting, the goal would be to obtain a function $I$ such that
\begin{equation}
\label{weakldp}
-\inf_{\xi\in A^\circ} I(\xi)\leq \liminf_{n\rightarrow\infty} \frac{\log \P(\bar X_n \in A)}{\log n} \leq \limsup_{n\rightarrow\infty} \frac{\log \P(\bar X_n \in A)}{\log n} \leq -\inf_{\xi\in \bar A} I(\xi),
\end{equation}
where ${A}^\circ$ and $\bar A$ are the interior and closure of $A$; all our large deviations results are derived in the Skorokhod $J_1$ topology.
Equation (\ref{weakldp}) is a classical large deviations principle (LDP) with sub-linear speed (cf.\ \cite{dembozeitouni}).
Using existing results in the literature (e.g.\ \cite{DiekerDenisovShneer}), it is not difficult to show that  $X(n)/n=\bar X_n(1)$ satisfies an LDP with rate function $I_1=I_1(x)$ which is $0$ at $0$, equal to $(\alpha-1)$ if $x>0$, and $(\beta-1)$ if $x<0$.
This is a lower-semicontinuous function of which the level sets are not compact.
Thus, in large-deviations terminology, $I_1$ is a rate function, but is not a good one.
This implies that techniques such as the projective limit approach cannot be applied.
In fact, in Section \ref{subsec:nonexistence}, we show that there does not exist an LDP of the form (\ref{weakldp}) for general sets $A$, by giving a counterexample.
A version of (\ref{weakldp}) for compact sets is derived in Section~\ref{subsec:weak-ldp}, as a corollary of our main results.
A result similar to (\ref{weakldp}) for random walks with semi-exponential (Weibullian) tails has been derived in \cite{Gantert98} (see also \cite{Gantert00, Gantert14} for related results).
Though an LDP for finite-dimensional distributions can be derived, lack of exponential tightness also persists at the sample-path level. To make the rate function good (i.e., to have compact level sets), a topology  chosen in \cite{Gantert98}  is considerably weaker than any of the Skorokhod topologies (but sufficient for the application that is central in that work).

The approach followed in the present paper is based on the recent developments in the theory of regular variation. In particular, in \cite{LRR},  the classical notion of regular variation is re-defined through a new convergence concept called $\mathbb M$-convergence (this is in itself a refinement of other reformulations of regular variation in function spaces; see \cite{DeHaanLin01, HultLindskog05, HultLindskog06}). In Section~\ref{sec:preliminaries}, we further investigate the $\mathbb M$-convergence framework by deriving a number of general results that facilitate the development of our proofs.

This paves the way towards our main large deviations results, which are presented in Section~\ref{sec:sample-path-ldps}. We actually obtain estimates that are sharper than (\ref{weakldp}), though we impose a condition on $A$. For one-sided L\'evy processes, our result takes the form
\begin{equation}
\label{asymptoticslevy}
 C_{\mathcal J(A)}(A^\circ) \leq \liminf_{n\rightarrow\infty} \frac{\P(\bar X_n \in A)}{(n\nu[n,\infty))^{\mathcal J(A)}} \leq \limsup_{n\rightarrow\infty} \frac{\P(\bar X_n \in A)}{(n\nu[n,\infty))^{\mathcal J(A)}} \leq C_{\mathcal J(A)}(\bar A).
\end{equation}
Precise definitions can be found in Section~\ref{subsec:one-sided-large-deviations}; for now we just mention that $C_j$ is a measure on the Skorokhod space, and $\mathcal J(\cdot)$ is an integer valued set function defined as $\mathcal J(A) = \inf_{\xi\in A\cap \mathbb D_s^\uparrow} \mathcal D_+(\xi)$, where $\mathcal D_+(\xi)$ is the number of discontinuites of $\xi$, and $\mathbb D_s^\uparrow$ is the set of all non-increasing step functions vanishing at the origin. 
Throughout the paper, we adopt the convention that the infimum over an empty set is $\infty$.
Letting $\mathbb D_j$ and $\mathbb D_{< j}$ be the sets of step functions vanishing at the origin with precisely $j$ and at most $j-1$ steps respectively, we note that the measure $C_j$, defined on $\mathbb D \setminus \mathbb D_{<j}$
has its support on $\mathbb D_j$. A crucial assumption for (\ref{asymptoticslevy}) to hold is that the Skorokhod $J_1$ distance between the sets $A$ and $\mathbb D_{< \mathcal J(A)}$ is strictly positive.
For $A$ such that $\mathcal J(A)=1$ this result corresponds to the one shown in \cite{HLMS}.  
(Note that \cite{HLMS} deals with multi-variate regular variation whereas we focus on 1-dimensional regular variation in this paper.)
The interpretation of the ``rate function'' $\mathcal J(A)$ is that it provides the number of jumps in the L\'evy process that are necessary to make the event $A$ happen. This can be seen as an extension of the principle of a single big jump to multiple jumps. A rigorous statement on  when (\ref{asymptoticslevy}) holds can be found in Theorem~\ref{thm:one-sided-main-theorem}, which is the first main result of the paper.


The result that comes closest to (\ref{asymptoticslevy}) is Theorem~5.1 in \cite{LRR} which considers the $\mathbb M$-convergence of $\nu[n,\infty)^{-j} \P(X/n \in A)$.
This result could be used as a starting point to investigate rare events that happen on a time-scale of $O(1)$. However, in the large-deviations scaling we consider,
rare events that happen on a time-scale of $O(n)$. Controlling the L\'evy process on this larger time-scale requires more delicate estimates, eventually
leading to an additional factor $n^j$ in the asymptotic results. We further show that the choice $j=\mathcal J(A)$ is the only choice that leads to a non-trivial limit.
One useful notion that we develop and rely on in our setting is a form of asymptotic equivalence, which can best be compared with exponential equivalence in classical large deviations theory.

In Section~\ref{subsec:two-sided-large-deviations} we present sample-path large deviations for two-sided L\'evy processes. Our main results in this case are Theorems
\ref{thm:two-sided-limit-theorem}--\ref{thm:two-sided-multiple-asymptotics}.
In the two-sided case, \revrem{we need to resolve}%
determining the most likely path requires resolving
significant combinatorial issues which do not appear in the one sided case. 
The polynomial rate of decay for $\P(\bar X_n \in A)$, which was described by the function $\mathcal J(A)$ in the one-sided case, has a more complicated description; the corresponding polynomial rate in the two-sided case is
\begin{equation}
\label{two-sided rate function}
\inf_{\xi,\zeta\in \mathbb D_s^\uparrow;\; \xi-\zeta \in A} (\alpha-1)\mathcal D_+(\xi) + (\beta-1)\mathcal D_+(\zeta).
\end{equation}
Note that this is a result that one could expect from the result for one-sided L\'evy processes and a heuristic application of the contraction principle.
A rigorous treatment of the two-sided case requires a more delicate argument compared to the one-sided case: in the one-sided case, the argument simplifies since if one takes $j$ largest jumps away from $\bar X_n$, then the probability that the residual process is of significant size is $o\big((n\nu[n,\infty))^j\big)$ so that it does not contribute in (\ref{asymptoticslevy}), while in two-sided case, taking $j$ largest upward jumps and $k$ largest downward jumps from $\bar X_n$ doesn't guarantee that the residual process remains small with high enough probability---i.e., the probability that the residual process is of significant size cannot be bounded by $o\big((n\nu[n,\infty))^j(n\nu(-\infty,-n])^k\big)$. In addition, it may be the case that multiple pairs $(j,k)$ of jumps lead to optimal solutions of (\ref{two-sided rate function}).
\revadd{
To overcome such difficulties, we first develop general tools---Lemma~\ref{thm:simple-product-space} and \ref{thm:union-limsup}---that establish a suitable notion of $\mathbb M$-convergence on product spaces. 
Using these results,
we prove 
in  Theorem~\ref{thm:multi-d-limit-theorem}
the suitable $\mathbb M$-convergence for multiple L\'evy processes in the associated product space.
Viewing the two-sided L\'evy process as a superposition of one-sided L\'evy processes, we then apply the continuous mapping principle for $\mathbb M$-convergence to Theorem~\ref{thm:multi-d-limit-theorem} to establish our main results.
Although no further implications are discussed in this paper, 
we believe that Theorem~\ref{thm:multi-d-limit-theorem} itself is of independent interest as well because it can be applied to generate large deviations results for a general class of functionals of multiple L\'evy processes.
}

We derive analogous results for random walks in Section~\ref{subsec:random-walks}. Random walks cannot be decomposed into independent components with small jumps and large jumps as easily as L\'evy processes, making the analysis of random walks more technical if done directly. However, it is possible to follow an indirect approach. Given a random walk $S_k, k\geq 0$, one can study a subordinated version $S_{N(t)}, t\geq 0$ with
$N(t),t\geq 0$ an independent unit rate Poisson process. The Skorokhod $J_1$ distance between rescaled versions of $S_k, k\geq 0$ and $S_{N(t)}, t\geq 0$ can then be bounded in terms of the deviations of $N(t)$ from $t$, which have been studied thoroughly. 

In Section~\ref{subsec:conditional-limit-theorem}, we provide conditional limit theorems which give a precise description of the limit behavior of $\bar X_n$ given that $\bar X_n \in A$ as $n\rightarrow\infty$. An early result of this type is given in \cite{durrett1980conditioned}, which focuses on regularly varying random walks with finite variance conditioned on the event $A= \{\bar X_n(1)>a\}$.
Using the recent results that we have discussed (e.g. \cite{HLMS}) more general conditional limit theorems can be derived for single-jump events.

We prove an LDP of the form (\ref{weakldp}) in Section~\ref{subsec:weak-ldp}, where the upper bound requires a compactness assumption. We construct a counterexample showing that the compactness assumption cannot be totally removed, and thus, a full LDP does not hold. Essentially, if a rare event is caused by $j$ big jumps, then the framework developed in this paper applies if each of these jumps is bounded away from below by a strictly positive constant. Our counterexample in Section~\ref{subsec:nonexistence} indicates that it is not trivial to remove this condition.

As one may expect, it is not possible to apply classical variational methods to derive an expression for the exponent $\mathcal J(A)$, as is often the case in large deviations for light tails. Nevertheless, there seems to be a generic connection with a class of control problems called impulse control problems.
Equation (\ref{two-sided rate function}) is a specific deterministic impulse-control problem, which is related to \cite{Barles}.
We expect that techniques similar to those in \cite{Barles} will be useful to characterize optimality of solutions for
problems like (\ref{two-sided rate function}).
The latter challenge is not taken up in the present study and will be addressed elsewhere. 
Instead, in Section~\ref{sec:applications}, we analyse (\ref{two-sided rate function}) directly in several examples; see also \cite{chen2017efficient}.
In each case, a condition needs to be checked to see whether
our framework is applicable. We provide a general result that essentially
states that we only need to check this condition for step functions in A, which
makes this check rather straightforward.
%

In summary, this paper is organized as follows. After developing some preliminary results in Section~\ref{sec:preliminaries}, we present our main results in Section~\ref{sec:sample-path-ldps}. Applications to random walks and connections with classical large deviations theory are investigated in Section~\ref{sec:implications}. 
Section~\ref{sec:proofs} is devoted to proofs. We collect some useful bounds in Appendix A.

\section{$\mathbb M$-convergence}\label{sec:preliminaries}
This section reviews and develops general concepts and tools that are useful in deriving our large deviations results. 
The proofs of the lemmas and corollaries stated throughout this section are provided in Section~\ref{subsec:proofs-for-M-convergence}.
We start with briefly reviewing the notion of $\mathbb M$-convergence, introduced in \citet{LRR}.

Let $(\mathbb S,d)$ be a complete separable metric space, and $\mathscr S$ be the Borel $\sigma$-algebra on $\mathbb S$.
Given a closed subset $\mathbb C$ of $\mathbb S$, let $\mathbb S\setminus \mathbb C$ be equipped with the relative topology as a subspace of $\mathbb S$,
and consider the associated sub $\sigma$-algebra $\mathscr S_{\mathbb S\setminus \mathbb C} \triangleq \{A: A\subseteq \mathbb S\setminus \mathbb C, A\in \mathscr S\}$ on it.
Define $\mathbb C^r \triangleq \{x\in \mathbb S: d(x, \mathbb C) < r\}$ for $r>0$, and let $\mathbb M(\mathbb S\setminus \mathbb C)$  be the class of measures defined on  $\mathscr S_{\mathbb S\setminus \mathbb C}$ whose restrictions to $\mathbb S\setminus \mathbb C^r$ are finite for all $r>0$.
Topologize $\mathbb M(\mathbb S\setminus \mathbb C)$ with a sub-basis $\big\{\{\nu \in \mathbb M(\mathbb S\setminus \mathbb C): \nu (f) \in G\}$: $f\in \mathcal C_{\mathbb S\setminus \mathbb C}$, $G$ open in $\R_+\big\}$ where $\mathcal C_{\mathbb S\setminus \mathbb C}$ is the set of real-valued, non-negative, bounded, continuous functions whose support is bounded away from $\mathbb C$ (i.e., $f(\mathbb C^r) = \{0\}$ for some $r>0$).
A sequence of measures $\mu_n \in \mathbb M(\mathbb S\setminus \mathbb C)$ converges to  $\mu\in \mathbb M(\mathbb S\setminus \mathbb C)$ if $\mu_n (f)  \to \mu(f)$ for each $f\in \mathcal C_{\mathbb S\setminus \mathbb C}$. 
Note that this notion of convergence in $\mathbb M(\mathbb S\setminus\mathbb C)$ coincides with the classical notion of weak convergence of measures (\citealp{billingsley2013convergence}) if $\mathbb C$ is an empty set.
We say that a set $A\subseteq \mathbb S$ is bounded away from another set  $B\subseteq \mathbb S$  if $\inf_{x\in A, y\in B} d(x,y) > 0$.
An important characterization of $\mathbb M(\mathbb S\setminus \mathbb C)$-convergence is as follows:
\begin{theorem}[Theorem~2.1 of \citealp{LRR}]\label{result:L21}
Let $\mu, \mu_n \in \mathbb M({\mathbb S\setminus \mathbb C})$. Then $\mu_n \to \mu$ in $\mathbb M({\mathbb S\setminus \mathbb C})$ as $n \to \infty$ if and only if
\begin{equation}
\label{result1ub}
\limsup_{n\to\infty } \mu_n(F) \leq \mu(F)
\end{equation}
for all closed $F\in \mathscr S_{\mathbb S\setminus \mathbb C}$ bounded away from $\mathbb C$ and
\begin{equation}
\label{result1lb}
\liminf_{n\to\infty} \mu_n(G) \geq \mu(G)
\end{equation}
for all open $G\in \mathscr S_{\mathbb S\setminus \mathbb C}$ bounded away from $\mathbb C$.
\end{theorem}
We now introduce a new notion of equivalence between two families of random objects, which will prove to be useful in Section~\ref{subsec:one-sided-large-deviations}, and Section~\ref{subsec:random-walks}.
Let $F_\delta \triangleq \{x \in \mathbb S: d(x,F) \leq \delta\}$ and $G^{-\delta} \triangleq ((G^c)_\delta)^c$. (Compare these notations to $\mathbb C^r$; note that we are using the convention that superscript implies open sets and subscript implies closed sets.)
\begin{definition}\label{def:asymptotic-equivalence}
Suppose that $X_n$ and $Y_n$ are random elements taking values in a complete separable metric space $(\mathbb S, d)$, and $\epsilon_n$ is a sequence of positive real numbers.
$Y_n$ is said to be asymptotically equivalent to $X_n$ with respect to $\epsilon_n$ 
if
for each $\delta>0$, 
$$\limsup_{n\to\infty} \epsilon_n^{-1}\P(d(X_n,Y_n) \geq \delta) = 0.$$
\end{definition}
The usefulness of this notion of equivalence comes from the following lemma, which states that
if $Y_n$ is asymptotically equivalent to $X_n$, and $X_n$ satisfies a limit theorem, then $Y_n$ satisfies the same limit theorem.
Moreover, it also allows one to extend the lower and upper bounds to more general sets in case there are asymptotically equivalent distributions that are supported on a subspace $\mathbb S_0$ of $\mathbb S$:

\begin{lemma}\label{lem:extended-bounds}
Suppose that $\epsilon_n^{-1} \P(X_n \in \cdot)\to \mu(\cdot)$ in $\mathbb M(\mathbb S\setminus \mathbb C)$ for some sequence $\epsilon_n$ and a closed set $\mathbb C$. In addition, suppose that  $\mu(\mathbb S \setminus \mathbb S_0) = 0$ and $\P(X_n \in \mathbb S_0) = 1$ for each $n$. If $Y_n$ is asymptotically equivalent to $X_n$ with respect to $\epsilon_n$, 
then
$$
\liminf_{n\to\infty} \epsilon_n^{-1}\P(Y_n \in G) \geq \mu(G)
$$
if $G$ is open and $G\cap \mathbb S_0$ is bounded away from $\mathbb C$;
$$
\limsup_{n\to\infty} \epsilon_n^{-1}\P(Y_n \in F) \leq \mu(F)
$$
if $F$ is closed and there is a $\delta>0$ such that $F_\delta \cap \mathbb S_0$ is bounded away from $\mathbb C$.
\end{lemma}
This lemma is particularly useful 
when we work in Skorokhod space, and $\mathbb S_0$ is the class of step functions.
Taking $\mathbb S_0 = \mathbb S$, a simpler version of Lemma~\ref{lem:extended-bounds} follows immediately:

\begin{corollary}\label{lem:asymptotic-equivalence}
Suppose that $\epsilon_n^{-1} \P(X_n \in \cdot) \to \mu(\cdot)$ in $\mathbb M(\mathbb S\setminus \mathbb C)$ for some sequence $\epsilon_n$. 
If $Y_n$ is asymptotically equivalent to $X_n$ with respect to $\epsilon_n$, 
then the law of $Y_n$ has the same (normalized) limit, i.e.,
$\epsilon_n^{-1}
 \P(Y_n \in \cdot ) \to \mu(\cdot)$
in $\mathbb M(\mathbb S\setminus \mathbb C)$.
\end{corollary}

Next, we discuss the $\mathbb M$-convergence in a product space as a result of the $\mathbb M$-convergences on each space. 
\begin{lemma}\label{thm:simple-product-space}
Suppose that $\mathbb S_1, \ldots, \mathbb S_d$ are separable metric spaces, $\mathbb C_1, \ldots, \allowbreak \mathbb C_d$ are closed subsets of $\mathbb S_1, \ldots, \mathbb S_d$, respectively. If
$\mu_n^{(i)}(\cdot)\to\mu^{(i)}(\cdot)$  in $\mathbb M(\mathbb S_i\setminus \mathbb C_i)$ for each $i=1,\ldots, d$
then,
\begin{equation}\label{crude-m-convergence-in-product-space}
\mu_n^{(1)}\times\cdots\times\mu_n^{(d)}(\cdot) \to \mu^{(1)}\times\cdots\times\mu^{(d)}(\cdot)
\end{equation}
in
$
\mathbb M\Big(\big(\prod_{i=1}^d \mathbb S_i\big) \setminus \bigcup_{i=1}^d \big(\big(\prod_{j=1}^{i-1} \mathbb S_j\big) \times  \mathbb C_i \times \big(\prod_{j=i+1}^d \mathbb S_j\big)
\big)\Big).
$

\end{lemma}

It should be noted that Lemma~\ref{thm:simple-product-space} itself is not exactly ``right'' in the sense that the set we take away is unnecessarily large, and hence, has limited applicability. 
More specifically, the $\mathbb M$-convergence in (\ref{crude-m-convergence-in-product-space}) applies only to the sets that are contained in a ``rectangular'' domain $\prod_{i=1}^d (\mathbb S_i\setminus \mathbb C_i)$.
Our next observation allows one to combine multiple instances of $\mathbb M$-convergences to establish a more refined one so that (\ref{crude-m-convergence-in-product-space}) applies to a class of sets that are not confined to a rectangular domain.
In particular, we will see later in Theorem~\ref{thm:two-sided-limit-theorem} and Theorem~\ref{thm:multi-d-limit-theorem} that in combination with Lemma~\ref{thm:simple-product-space}, the following lemma produces the ``right'' $\mathbb M$-convergence for two-sided L\'evy processes and random walks.

\begin{lemma}\label{thm:union-limsup}
Consider a family of measures $\{\mu^{(i)}\}_{i=0,1,\ldots,m}$ and a family of closed subsets $\{\mathbb C(i)\}_{i=0,1,\ldots,m}$ of $\mathbb S$
such that $\frac{1}{\epsilon_n{(i)}}\P(X_n \in \cdot) \to \mu^{(i)}(\cdot)$ in $\mathbb M(\mathbb S\setminus \mathbb C(i))$ for $i=0,\ldots,m$
where $\big\{\{{\epsilon_n{(i)}}: n\geq 1\}\big\}_{i=0,1,\ldots,m}$ is the family of associated normalizing sequences.
Suppose that
$\mu^{(0)} \in \mathbb M\big(\mathbb S \setminus \bigcap_{i=0}^{m}\mathbb C(i)\big)$;
$
\limsup_{n\to\infty}\frac{{\epsilon_n{(i)}}}{{\epsilon_n{(0)}}} = 0
$
for $i=1,\ldots,m$; and
for each $r>0$, there exist positive numbers $r_0,\ldots,r_m$ such that $\bigcap_{i=0}^m\mathbb C(i)^{r_i}\subseteq \big(\bigcap_{i=0}^{m}\mathbb C(i)\big)^r$.
Then
$$
\frac{1}{\epsilon_n{(0)}}\P(X_n\in \cdot) \to \mu^{(0)}
$$
in $\mathbb M\big(\mathbb S \setminus \bigcap_{i=0}^{m}\mathbb C(i)\big)$.
\end{lemma}

A version of the continuous mapping principle is satisfied by $\mathbb M$-convergence. Let $(\mathbb S',d')$ be a complete separable metric space, and let $\mathbb C'$ be a closed subset of $\mathbb S'$.

\begin{theorem}[Mapping theorem; Theorem~2.3 of \citet{LRR}]\label{result:L23}
Let $h:(\mathbb S\setminus \mathbb C, \mathscr S_{\mathbb S\setminus \mathbb C}) \to (\mathbb S'\setminus \mathbb C', \mathscr S_{\mathbb S'\setminus \mathbb C'})$ be a measurable mapping such that $h^{-1}(A')$ is bounded away from $\mathbb C$ for any $A'\in \mathscr S_{\mathbb S'\setminus \mathbb C'}$ bounded away from $\mathbb C'$.
Then $\hat h: \mathbb M({\mathbb S\setminus \mathbb C}) \to \mathbb M({\mathbb S'\setminus \mathbb C'})$ defined by $\hat h(\nu) = \nu \circ h^{-1}$ is continuous at $\mu$ provided $\mu(D_h) = 0$, where $D_h$ is the set of discontinuity points of $h$.
\end{theorem}
For our purpose, the following slight extension will prove to be useful in developing rigorous arguments.
\begin{lemma}\label{lem:almost-continuous-mapping}
Let $\mathbb S_0$ be a measurable subset of $\mathbb S$, and $h:(\mathbb S_0,  \mathscr S_{\mathbb S_0 })\to (\mathbb S'\setminus \mathbb C', \mathscr S'_{\mathbb S' \setminus \mathbb C'})$ be a measurable mapping such that $h^{-1}(A')$ is bounded away from $\mathbb C$ for any $A'\in \mathscr S_{\mathbb S'\setminus \mathbb C'}$ bounded away from $\mathbb C'$.
Then $\hat h:\mathbb M({\mathbb S\setminus \mathbb C}) \to \mathbb M({\mathbb S'}\setminus \mathbb C')$ defined by $\hat h(\nu) = \nu \circ h^{-1}$ is continuous at $\mu$ provided that $\mu(\partial \mathbb S_0\setminus \mathbb C^r) = 0$ and $\mu(D_h\setminus \mathbb C^r)=0$ for all $r>0$, where $D_h$ is the set of discontinuity points of $h$.
\end{lemma}
When we focus on L\'evy processes, we are specifically interested in the case where $\mathbb S$ is $\mathbb R_+^{\infty\downarrow}\times[0,1]^\infty$,
where $\mathbb R_+^{\infty\downarrow} \triangleq \{x \in \mathbb R_+^\infty: x_1 \geq x_2 \geq \ldots\}$, and $\mathbb S'$ is the Skorokhod space $\mathbb D = \mathbb D([0,1],\R)$ --- the space of real-valued RCLL functions on $[0,1]$.
We use the usual product metrics $d_{\mathbb R_+^{\infty\downarrow}}(x,y) = \sum_{i=1}^\infty \frac{|x_i - y_i| \wedge 1}{2^i}$ and $d_{[0,1]^\infty}(x,y) = \sum_{i=1}^\infty \frac{|x_i - y_i| }{2^i}$ for $\mathbb R_+^{\infty\downarrow}$ and $[0,1]^\infty$, respectively.
For the finite product of metric spaces, we use the maximum metric; i.e., we use $d_{\mathbb S_1\times\cdots\times\mathbb S_d}((x_1,\ldots,x_d), (y_1,\ldots,y_d)) \triangleq \max_{i=1,\ldots,d}d_{\mathbb S_i}(x_i,y_i) $ for the product $\mathbb S_1\times\cdots\times\mathbb S_d$ of metric spaces $(\mathbb S_i,d_{\mathbb S_i})$.
For $\mathbb D$, we use the usual Skorokhod $J_1$ metric $d(x,y) \triangleq \inf_{\lambda \in \Lambda} \|\lambda - e\| \vee \| x\circ \lambda - y\|$, where $\Lambda$ denotes the set of all non-decreasing homeomorphisms from $[0,1]$ onto itself, $e$ denotes the identity, and $\|\cdot\|$ denotes the supremum norm.
Let $$S_j \triangleq \{(x,u)\in \R_+^{\infty\downarrow}\times [0,1]^\infty: 0, 1, u_1,\ldots, u_j \text{ are all distinct}\}.$$
This set will play the role of $\mathbb S_0$ of Lemma~\ref{lem:almost-continuous-mapping}.
Define
$T_j: S_j\to \mathbb D$ to be $T_j(x,u) = \sum_{i=1}^j x_i 1_{[u_i,1]}$.
Let $\mathbb D_j$ be the subspaces of the Skorokhod space consisting of nondecreasing step functions, vanishing at the origin, with exactly $j$ jumps, and $\mathbb D_{\leqslant j}\triangleq \bigcup_{0\leq i\leq j} \mathbb D_i$---i.e., nondecreasing step functions vanishing at the origin with at most $j$ jumps.
Similarly, let $\mathbb D_{< j}\triangleq \bigcup_{0\leq i<\, j} \mathbb D_i$.
Define $\mathbb H_{j} \triangleq \{x \in \mathbb R_+^{\infty\downarrow}: x_j > 0, x_{j+1} = 0\},$ and
$\mathbb H_{< \,j} \triangleq \{x\in \mathbb R_+^{\infty \downarrow}: x_{j} = 0\}$.
The continuous mapping principle applies to $T_j$, as we can see in the following result.

\begin{lemma}[Lemma 5.3 and Lemma 5.4 of \citealp{LRR}]\label{result:L53}
Suppose $A \subset \mathbb D$ is bounded away from $\mathbb D_{< \,j}$. Then, $T_j^{-1}(A)$ is bounded away from $\mathbb H_{<\,j} \times [0,1]^\infty$. Moreover, $T_j:S_j \to \mathbb D$ is continuous.
\end{lemma}

A consequence of Result~\ref{result:L53} and Lemma~\ref{lem:almost-continuous-mapping} along with the observation that $S_j$ is open is that one can derive a limit theorem in a path space from a limit theorem for jump sizes.
\begin{corollary}\label{result:consequence-of-L23-53-54}
If $\mu_n\to \mu$ in $\mathbb M\big((\mathbb R_+^{\infty\downarrow} \times [0,1]^\infty) \setminus (\mathbb H_{<\,j} \times [0,1]^\infty)\big)$, and $\mu\big(S_{j}^c\setminus (\mathbb H_{<\, j}\times [0,1]^\infty)^r\big) = 0$ for all $r>0$, then $\mu_n \circ T_{j}^{-1} \to \mu\circ T_{j}^{-1}$ in $\mathbb M(\mathbb D \setminus \mathbb D_{<\, j})$.
\end{corollary}

To obtain the large deviations for two-sided L\'evy measures, we will  first establish the large deviations for independent spectrally positive L\'evy processes, and then apply Lemma~\ref{lem:almost-continuous-mapping} with $h(\xi,\zeta) = \xi-\zeta$.
The next lemma verifies two important conditions of Lemma~\ref{lem:almost-continuous-mapping} for such $h$.
Let $\mathbb D_{l,m}$ denote the subspace of the Skorokhod space consisting of step functions vanishing at the origin with exactly $l$ upward jumps and $m$ downward jumps.
Given $\alpha, \beta>1$, let $\mathbb D_{< \,j,k} \triangleq \bigcup_{(l,m)\in \I{<\,j,k}} \mathbb D_{l,m}$ and $\mathbb D_{<(j,k)} \triangleq\bigcup_{(l,m)\in \I{<\,j,k}} \mathbb D_l\times \mathbb D_m$,
where $\I_{<\,j,k}  \triangleq \big\{(l,m)\in \Z_+^2\setminus\{(j,k)\}: (\alpha-1) l + (\beta-1) m \leq (\alpha-1) j + (\beta-1) k\big\}$ and $\Z_+$ denotes the set of non-negative integers.
Note that in the definition of $\I_{<\,j,k}$, the inequality is not strict;
however, we choose to use the strict inequality in our notation to emphasize that $(j,k)$ is not included in $\I_{<\,j,k}$. 

\begin{lemma}\label{lem:continuous-mapping-principle-for-subtraction}
Let $h:\mathbb D \times \mathbb D \to \mathbb D$ be defined as
$h(\xi,\zeta) \triangleq \xi-\zeta.$
Then, $h$ is continuous at $(\xi,\zeta)\in \mathbb D\times \mathbb D$ such that $(\xi(t) - \xi(t-))(\zeta(t)-\zeta(t-)) = 0$ for all $t\in(0,1]$. Moreover, $h^{-1}(A)\subseteq \mathbb D\times \mathbb D$ is bounded away from $\mathbb D_{<(j,k)}$ for any $A\subseteq \mathbb D$ bounded away from $\mathbb D_{<\,j,k}$.
\end{lemma}

We next characterize convergence-determining classes for the convergence in $\mathbb M(\mathbb S\setminus \mathbb C)$.

\begin{lemma}\label{lem:convergence-determining-class}
Suppose that 
\emph{(i)}  $\mathcal A_p$ is a $\pi$-system; 
\emph{(ii)}  each open set $G \subseteq \mathbb S$ bounded away from $\mathbb C$ is a countable union of sets in $\mathcal A_p$; and 
\emph{(iii)}  for each closed set $F\subseteq \mathbb S$ bounded away from $\mathbb C$, there is a set $A \in \mathcal A_p$ bounded away from $\mathbb C$ such that  $F\subseteq A^\circ$ and $\mu(A\setminus A^\circ) = 0$.
If, in addition, $\mu\in \mathbb M(\mathbb S\setminus \mathbb C)$ and $\mu_n(A) \to \mu(A)$ for every $A\in \mathcal A_p$ such that $A$ is bounded away from $\mathbb C$, then $\mu_n \to \mu$ in $\mathbb M(\mathbb S\setminus \mathbb C)$.
\end{lemma}
\begin{remark}
Since $\mathbb S$ is a separable metric space, the Lindel\"of property holds.
Therefore,
a sufficient condition for  assumption (ii) of Lemma~\ref{lem:convergence-determining-class} is that
for every $x\in \mathbb S \setminus \mathbb C$ and $\epsilon>0$, there is $A\in \mathcal A_p$ such that $x \in A^\circ \subseteq B(x,\epsilon)$. To see that this implies  assumption (ii), note that for any given open set $G$, one can construct a cover $\{(A_x)^\circ: x \in G\}$ of $G$ by choosing $A_x$ so that $x \in (A_x)^\circ \subseteq G$ and then extract a countable subcover (due to the Lindel\"of property) whose union is equal to $G$.
Note also that if $A$ in assumption (iii) is open, then $\mu(A\setminus A^\circ) = \mu(\emptyset) = 0$ automatically.
\end{remark}

\section{Sample-Path Large Deviations}\label{sec:sample-path-ldps}
In this section, we present large-deviations results for scaled L\'evy processes with heavy-tailed L\'evy measures. Section~\ref{subsec:one-sided-large-deviations} studies a special case, where the L\'evy measure is concentrated on the positive part of the real line, and Section~\ref{subsec:two-sided-large-deviations} extends this result to L\'evy processes with two-sided L\'evy measures. In both cases, let $X_n(t) \triangleq X(nt)$ be a scaled process of $X$, where $X$ is a L\'evy process with a L\'evy measure $\nu$.
Recall that $X_n$ has It\^{o} representation (see, for example, Section 2 of \citealp{kyprianou2014fluctuations}):
\begin{align*}
X_n(s)
&= 
nsa + B(ns) 
\\
&\quad
+ \int_{|x|\leq 1} x[N([0,ns]\times dx) - ns\nu(dx)] 
+ \int_{|x|>1} xN([0,ns]\times dx),
\end{align*}
with $a$ a drift parameter, $B$ a Brownian motion, and $N$ a Poisson random measure with mean measure Leb$\times\nu$ on $[0,n]\times(0,\infty)$; Leb denotes the Lebesgue measure.

\subsection{One-sided Large Deviations}\label{subsec:one-sided-large-deviations}
Let $X$ be a L\'evy process with  L\'evy measure $\nu$.
In this section, we assume that $\nu$ is a regularly varying (at infinity, with index $-\alpha<-1$) L\'evy measure concentrated on $(0,\infty)$.
Consider a centered and scaled process
\begin{equation}\label{math-display-above-definition-nu-alpha-j}
\bar X_n(s) \triangleq \frac{1}{n}X_n(s) - sa - \mu_1^+\nu_1^+s,
\end{equation}
where $\mu_1^+ \triangleq \frac{1}{\nu_1^+}\int_{[1,\infty)} x\nu(dx)$, and $\nu_1^+ \triangleq \nu[1,\infty)$.
For each constant $\gamma>1$, let $\nu_\gamma(x,\infty) \triangleq x^{-\gamma}$, and 
let $\nu_\gamma^j$ denote the restriction (to $\mathbb R_+^{j\downarrow}$) of the $j$-fold product measure of $\nu_\gamma$.
Let $C_0(\cdot)\triangleq \delta_\mathbf 0(\cdot)$ be the Dirac measure concentrated on the zero function. Additionally, for each $j\geq 1$, define a measure $C_j\in \mathbb M(\mathbb D \setminus \mathbb D_{<\, j})$  concentrated on $\mathbb D_{j}$ as $C_j(\cdot) \triangleq \E \Big[\nu_\alpha ^j \{y\in (0,\infty)^j:\sum_{i=1}^j y_i 1_{[U_i,1]}\in \cdot\}\Big]$, where the random variables $U_i, i\geq 1$ are i.i.d.\ uniform on $[0,1]$.

The proof of the main result of this section hinges critically on the following limit theorem.
\begin{theorem}\label{thm:one-sided-limit-theorem}
For each $j\geq 0$,
\begin{equation}\label{eq:main-result}
(n\nu[n,\infty))^{-j}\P(\bar X_n\in \cdot) \to C_j(\cdot),
\end{equation}
in $\mathbb M(\mathbb D \setminus \mathbb D_{<\, j})$, as $n\to \infty$.
Moreover, $\bar X_n$ is asymptotically equivalent to a process that assumes values in $\mathbb D_{\leqslant \mathcal J(A)}$ almost surely.
\end{theorem}
\begin{proof}[Proof Sketch]
The proof of Theorem~\ref{thm:one-sided-limit-theorem} is based on establishing the asymptotic equivalence of $\bar X_n$ and the process obtained by just keeping its $j$ biggest jumps, which we will denote by $\hat J_n^{\leqslant j}$ in Section~\ref{sec:proofs}.
Such an equivalence is established via Proposition~\ref{prop:asymptotic-equivalence-Xbar-Jbar}, and Proposition~\ref{prop:asymptotic-equivalence-Jbar-Jj}.
Then, Proposition~\ref{prop:Jj} identifies the limit of $\hat J_n^{\leqslant j}$, which coincides with the limit in (\ref{eq:main-result}).
The full proof of Theorem~\ref{thm:one-sided-limit-theorem} is provided in Section~\ref{subsec:proofs-for-sample-path-ldps}.
\end{proof}

Recall that $\mathbb D_s^\uparrow$ denotes the subset of $\mathbb D$ consisting of non-decreasing step functions vanishing at the origin, and $\mathcal D_+(\xi)$ denotes the number of upward jumps of an element $\xi$ in $\mathbb D$. Finally, set
\begin{equation}\label{def:JA}
\mathcal J(A) \triangleq \inf_{\xi\in \mathbb D_s^\uparrow \cap A } \mathcal D_+(\xi).
\end{equation}
Now we are ready to present the main result of this section, which is the following large-deviations theorem for $\bar X_n$.

\begin{theorem}\label{thm:one-sided-main-theorem}
Suppose that $A$ is a measurable set. 
If $\mathcal J(A)<\infty$, and if $A_\delta \cap \mathbb D_{\leqslant \mathcal J(A)}$ is bounded away from $\mathbb D_{< \mathcal J(A)}$ for some $\delta>0$,
then
\begin{equation}\label{eq:one-sided-large-deviations}
\begin{aligned}
C_{\mathcal J(A)}(A^\circ)
&
\leq
\liminf_{n\rightarrow\infty} \frac{\P(\bar X_n \in A) }{(n \nu[n,\infty))^{\mathcal J(A)}}
\\
&
\leq
\limsup_{n\rightarrow\infty} \frac{\P(\bar X_n \in A)}{(n \nu[n,\infty))^{\mathcal J(A)}}
\leq
C_{\mathcal J(A)}(\bar A).
\end{aligned}
\end{equation}
\chck If $\mathcal J(A)= \infty$, and $A_\delta \cap \mathbb D_{\leqslant i+1}$ is bounded away from $\mathbb D_{\leqslant i}$ for some $\delta>0$ and $i\geq 0$, then  
\begin{equation}\label{eq:one-sided-large-deviations-null}
\lim_{n\to\infty}\frac{\P(\bar X_n \in A)}{(n\nu[n,\infty))^{i}} = 0.
\end{equation}
In particular, in case $\mathcal J(A)< \infty$, \eqref{eq:one-sided-large-deviations} holds if $A$ is bounded away from $\mathbb D_{<\,\mathcal J(A)}$; in case $\mathcal J(A) = \infty$, \eqref{eq:one-sided-large-deviations-null} holds if $A$ is bounded away from $\mathbb D_{\leqslant i}$.
\end{theorem}
\begin{proof}
We first consider the case $\mathcal J(A) < \infty$.
Note that $\mathcal J(A^\circ) > \mathcal J(A)$ implies that $A^\circ$ doesn't contain any element of $\mathbb D_{\leqslant \mathcal J(A)}$. 
Since $C_{\mathcal J(A)}$ is supported on $\mathbb D_{\leqslant \mathcal J(A)}$, $A^\circ$ is a $C_{\mathcal J(A)}$-null set.
Therefore, the lower bound holds trivially if $\mathcal J(A^\circ) > \mathcal J(A)$.
On the other hand, $\mathcal J(A) = \mathcal J(\bar A)$. 
To see this, suppose not---i.e., $\mathcal J(\bar A) < \mathcal J(A)$. 
Then, there exists $\zeta\in \mathbb D_s^\uparrow \cap \bar A $ such that $\zeta \in \mathbb D_{< \mathcal J(A)}$.
This implies that $\zeta \in A_\delta \cap \mathbb D_{\leqslant J(A) }$ for any $\delta>0$, which is contradictory to the assumption that $A_\delta \cap \mathbb D_{\leqslant J(A) }$ is bounded away from $\mathbb D_{< \mathcal J(A)}$ for some $\delta>0$. 
In view of these observations, we can assume w.l.o.g.\ that $\mathcal J(A^\circ) = \mathcal J(A) = \mathcal J(\bar A)$.
Now, from Theorem~\ref{thm:one-sided-limit-theorem} with $j=\mathcal J(A^\circ)$ along with the lower bound of Lemma~\ref{lem:extended-bounds},
\begin{align*}
C_{\mathcal J(A)} (A^\circ) 
&
= C_{\mathcal J(A^\circ)} (A^\circ) \leq \liminf_{n\to\infty} \frac{\P(\bar X_n \in A^\circ)}{(n\nu[n,\infty))^{\mathcal J(A^\circ)}}
\\
&\leq 
\liminf_{n\to\infty} \frac{\P(\bar X_n \in A)}{(n\nu[n,\infty))^{\mathcal J(A)}}.
\end{align*}
Similarly, from Theorem~\ref{thm:one-sided-limit-theorem} with $j=\mathcal J(\bar A)$ along with the upper bound of Lemma~\ref{lem:extended-bounds},
\begin{align*}
\limsup_{n\to\infty} \frac{\P(\bar X_n \in A)}{(n\nu[n,\infty))^{\mathcal J(A)}}
&
\leq
\limsup_{n\to\infty} \frac{\P(\bar X_n \in \bar A)}{(n\nu[n,\infty))^{\mathcal J(\bar A)}}
\\
&
\leq
C_{\mathcal J(\bar A)} (\bar A) 
=
C_{\mathcal J(A)} (\bar A).
\end{align*}
In case $\mathcal J(A) = \infty$, we reach the conclusion by applying Theorem~\ref{thm:one-sided-limit-theorem} with $j=i$ along with noting that $C_i(\bar A) = 0$.
\end{proof}


Theorem~\ref{thm:one-sided-main-theorem} dictates the ``right'' choice of $j$ in Theorem~\ref{thm:one-sided-limit-theorem} for which (\ref{eq:main-result}) can lead to a limit in $(0,\infty)$.
We conclude this section with an investigation of a sufficient condition for $C_j$-continuity; i.e., we provide a sufficient condition on $A$ which guarantees $C_j(\partial A) = 0$. The latter property implies
\begin{equation}
\label{A-continuity}
C_j(A^\circ) = C_j(A)= C_j(\bar A),
\end{equation}
implying that the liminf and limsup in our asymptotic estimates yield the same result.
Assume that $A$ is a subset of $\mathbb D_{j}$ bounded away from $\mathbb D_{<\, j}$; i.e., $d(A,\mathbb D_{<\,j})>\gamma$ for some $\gamma>0$.
Consider a path $\xi\in A$. Note that every $\xi\in \mathbb D_j$ is determined by the pair of jump sizes and jump times $(x,u) \in (0,\infty)^j\times [0,1]^j$; i.e.,
$\xi(t) = \sum_{i=1}^j x_i 1_{[u_i,1]}(t)$.
Formally, we define a mapping $\hat T_j: \hat S_j \rightarrow \mathbb D_j$ by $\hat T_j(x,u) = \sum_{i=1}^j x_i 1_{[u_i,1]}$, where $\hat S_j \triangleq \{(x,u)\in \R_+^{j\downarrow}\times [0,1]^j: 0, 1, u_1,\ldots, u_j \text{ are all distinct}\}$.
Since $d(A,\mathbb D_{<\, j})>\gamma$, we know that $\hat T_j(x,u)\in A$ implies $x\in (\gamma,\infty)^j$; see Lemma~\ref{lem:Djk} (b).
In view of this, we can see that (\ref{A-continuity}) holds if the Lebesgue measure of $\hat T_{j}^{-1}(\partial A )$ is 0 since
$
C_j(A) = \int_{(x,u) \in \hat T_j^{-1}(A)} du d\nu_\alpha^j(x)
$%
.
One of the typical settings that arises in applications is that the set $A$ can be written as a finite combination of unions and intersections of $\phi_1^{-1}(A_1),\ldots,\phi_m^{-1}(A_m)$, where each $\phi_i:\mathbb D \to \mathbb S_i$ is a continuous function, and all sets $A_i$ are subsets of general topological space $\mathbb S_i$. If we denote this operation  of taking unions and intersections by $\Psi$ (i.e., $A = \Psi(\phi_1^{-1}(A_1),\ldots,\phi_m^{-1}(A_m))$), then
\[
\Psi(\phi_1^{-1}(A_1^\circ),\ldots,\phi_m^{-1}(A_m^\circ))
\subseteq
A^\circ \subseteq A \subseteq \bar A
\subseteq
\Psi(\phi_1^{-1}(\bar A_1),\ldots,\phi_m^{-1}(\bar A_m)).
\]
Therefore,  (\ref{A-continuity}) holds if $\hat T_j^{-1}(\Psi(\phi_1^{-1}(\bar A_1),\ldots,\phi_m^{-1}(\bar A_m)))\setminus \hat T_j^{-1}(\Psi(\phi_1^{-1}(A_1^\circ),\allowbreak\ldots,\allowbreak\phi_m^{-1}(A_m^\circ)))$ has Lebesgue measure zero. A similar principle holds for the limit measures $C_{j,k}$, defined in
the next section where we deal with two-sided L\'evy processes. 

\subsection{Two-sided Large Deviations}\label{subsec:two-sided-large-deviations}
Consider a two-sided L\'evy measure $\nu$ for which $\nu[x,\infty)$ is regularly varying with index $-\alpha$ and $\nu(-\infty, -x]$ is regularly varying with index $-\beta$.
Let
$$
\bar X_n(s) \triangleq \frac{1}{n}X_n(s) - sa - (\mu_1^+\nu_1^+-\mu_1^-\nu_1^-)s,
$$
where
\begin{align*}
\mu_1^+ &\triangleq \frac{1}{\nu_1^+}\int_{[1,\infty)} x\nu(dx),
&
\nu_1^+ &\triangleq \nu[1,\infty),
\\
\mu_1^- &\triangleq \frac{-1}{\nu_1^-}\int_{(-\infty,-1]} x\nu(dx),
&
\nu_1^- &\triangleq \nu(-\infty,-1].
\end{align*}
Recall the definition of $\mathbb D_{j,k}$ given below Corollary~\ref{result:consequence-of-L23-53-54}, 
%
and the definition of $\nu_\alpha^j$ and $\nu_\beta^k$ as given below \eqref{math-display-above-definition-nu-alpha-j}.
Let $C_{0,0}(\cdot) \triangleq \delta_{\mathbf 0}(\cdot)$ be the Dirac measure concentrated on the zero function.
For each $(j,k)\in \Z_+^2\setminus \{(0,0)\}$, define a measure $C_{j,k}\in \mathbb M(\mathbb D\setminus \mathbb D_{<j,k})$ concentrated on $\mathbb D_{j,k}$ as $C_{j,k}(\cdot) \triangleq \E \Big[\nu_\alpha ^j\times\nu_\beta^k \{(x,y)\in (0,\infty)^j\times(0,\infty)^k:\sum_{i=1}^j x_i 1_{[U_i,1]} - \sum_{i=1}^k y_i1_{[V_i,1]}\in \cdot\}\Big]$, where $U_i$'s and $V_i$'s are i.i.d.\ uniform on $[0,1]$.
Recall that $\mathbb D_{< j,k} = \bigcup_{(l,m)\in \I{<j,k}} \mathbb D_{l,m}$ and $\I{<j,k}  = \big\{(l,m)\in \Z_+^2\setminus\{(j,k)\}: (\alpha-1) l + (\beta-1) m \leq (\alpha-1) j + (\beta-1) k\big\}$.

As in the one-sided case, the proof of the main theorem of this section hinges on the following limit theorem.

\begin{theorem}\label{thm:two-sided-limit-theorem}
For each $(j, k)\in \Z_+^2$,
\begin{equation}\label{eq:two-sided-limit-theorem}
(n\nu[n,\infty))^{-j}(n\nu(-\infty,-n])^{-k}\P(\bar X_n\in \cdot) \to C_{j,k}(\cdot)
\end{equation}
in $\mathbb M(\mathbb D \setminus \mathbb D_{< j,k})$ as $n\to \infty$.
\end{theorem}
The proof of Theorem~\ref{thm:two-sided-limit-theorem} builds on 
Theorem~\ref{thm:one-sided-limit-theorem},
using
Lemma~\ref{thm:simple-product-space},
Lemma~\ref{thm:union-limsup}, and
Lemma~\ref{lem:continuous-mapping-principle-for-subtraction}
and Theorem~\ref{thm:multi-d-limit-theorem}.
We provide the full proof in Section~\ref{subsec:proofs-for-sample-path-ldps}. 

Let $\mathcal I(j,k) \triangleq (\alpha-1)j + (\beta-1)k$, and consider a pair of integers $(\mathcal J(A),\mathcal K(A))$ such that
\begin{equation}\label{def:JK}
(\mathcal J(A), \mathcal K(A)) \in \argmin_{\substack{(j,k)\in \mathbb Z_+^2\\ \mathbb D_{j,k} \cap A \neq \emptyset}}\mathcal I(j,k).
\end{equation}
The next theorem is the first main result of this section. 
\begin{theorem}\label{thm:two-sided-main-theorem}
Suppose that $A$ is a measurable set. If the argument minimum in (\ref{def:JK}) is non-empty and $A$ is bounded away from $\mathbb D_{< \mathcal J(A), \mathcal K(A)}$, then the argument minimum is unique and
\begin{equation}\label{eq:two-sided-main-result}
\begin{aligned}
\liminf_{n\rightarrow\infty} \frac{\P(\bar X_n \in A) }{(n \nu[n,\infty))^{\mathcal J(A)}(n \nu(-\infty,-n])^{\mathcal K(A)}}  
&\geq
C_{\mathcal J(A), \mathcal K(A)}(A^\circ),
\\
\limsup_{n\rightarrow\infty} \frac{\P(\bar X_n \in A)}{(n \nu[n,\infty))^{\mathcal J(A)} (n \nu(-\infty,-n])^{\mathcal K(A)} }
&\leq
C_{\mathcal J(A),\mathcal K(A)}(\,\bar A\;).
\end{aligned}
\end{equation}
\chck Moreover, 
if the argument minimum in (\ref{def:JK}) is empty and $A$ is bounded away from $\mathbb D_{<l,m}\cup \mathbb D_{l,m}$ for some $(l,m)\in \Z_+^2\setminus \{(0,0)\}$, then 
\begin{equation}\label{two-sided-limit-tends-to-zero}
\lim_{n\to\infty}\frac{ \P(\bar X_n \in A)}{(n \nu[n,\infty))^{l}\allowbreak (n \nu(-\infty,-n])^{m}} = 0.
\end{equation}
\end{theorem}

The proof of the theorem is provided below as a consequence of the following lemma. 

\begin{lemma}\label{wasteful-lemma}
Suppose that a sequence of $\mathbb D$-valued random elements $Y_n$ satisfies \eqref{eq:two-sided-limit-theorem} (with $\bar X_n$ replaced with $Y_n$) for each $(j,k)\in \mathbb Z_+^2$. Then \eqref{eq:two-sided-main-result} (with $\bar X_n$ replaced with $Y_n$) holds if $A$ is a measurable set for which the argument minimum in \eqref{def:JK} is non-empty, and $A$ is bounded away from $\mathbb D_{<\mathcal J(A), \mathcal K(A)}$.
Moreover, \eqref{two-sided-limit-tends-to-zero} (with $\bar X_n$ replaced with $Y_n$) holds if the argument minimum in (\ref{def:JK}) is empty and $A$ is bounded away from $\mathbb D_{<l,m}\cup \mathbb D_{l,m}$ for some $(l,m)\in \Z_+^2\setminus \{(0,0)\}$.
\end{lemma}
The proof of this lemma is provided in Section~\ref{subsec:proofs-for-sample-path-ldps}.

\begin{proof}[Proof of Theorem~\ref{thm:two-sided-main-theorem}]
The uniqueness of the argument minimum is immediate from the assumption that $A$ is bounded-away from $\mathbb D_{<\mathcal J(A), \mathcal K(A)}$.
Since $\bar X_n$ satisfies \eqref{eq:two-sided-limit-theorem} by Theorem~\ref{thm:two-sided-limit-theorem}, the conclusion of the theorem follows from applying Lemma~\ref{wasteful-lemma} with $Y_n = \bar X_n$. 
\end{proof}

In case one is interested in a set for which the $\argmin$  of $\mathcal I$ in (\ref{def:JK}) is not unique, a natural approach is to partition $A$ into smaller sets and analyze each element separately. In the next theorem, we show that this strategy can be successfully employed with a minimal requirement on $A$. However, due to the presence of two different slowly varying functions $n^\alpha\nu[n,\infty)$ and $n^\beta\nu(-\infty,-n]$, the limit behavior may not be dominated by a single $\mathbb D_{l,m}$. 

To deal with this case, let $\I_{= j,k} \triangleq \{(l,m): (\alpha-1)l+(\beta-1)m = (\alpha-1)j+(\beta-1)k\}$, $\I_{\ll j,k} \triangleq \{(l,m): (\alpha -1)l + (\beta-1)m < (\alpha-1)j+(\beta-1)k\}$, $\mathbb D_{= j,k} \triangleq \bigcup_{(l,m)\in \I_{= j,k}} \mathbb D_{l,m}$, and $\mathbb D_{\ll j,k} \triangleq \bigcup_{(l,m)\in \I_{\ll j,k}} \mathbb D_{l,m}$. Denote the slowly varying functions $n^\alpha\nu[n,\infty)$ and $n^\beta\nu(-\infty,-n]$ with $L_+(n)$ and $L_-(n)$, respectively.

\begin{theorem}\label{thm:two-sided-multiple-asymptotics}
Let $A$ be a measurable set and suppose that the argument minimum in (\ref{def:JK}) is non-empty and contains a pair of integers $(\mathcal J(A),\mathcal K(A))$.
If $A_\delta \cap \mathbb D_{=\mathcal J(A), \mathcal K(A)}$ is bounded away from $\mathbb D_{\ll\mathcal J(A), \mathcal K(A)}$ for some $\delta>0$, then for any given $\epsilon > 0$, there exists $N\in \mathbb N$ such that
\begin{equation}\label{two-sided-large-deviation-combined}
\begin{aligned}
\P(\bar X_n \in A) & \geq \frac{\sum_{(l,m)} \big(C_{l,m}(A^\circ)-\epsilon\big)L_+^l(n)L_-^m(n)}{n^{(\alpha-1)\mathcal J(A)+(\beta-1)\mathcal K(A)}},
\\
\P(\bar X_n \in A) & \leq \frac{\sum_{(l,m)} \big(C_{l,m}(\bar A)+\epsilon\big)L_+^l(n)L_-^m(n)}{n^{(\alpha-1)\mathcal J(A)+(\beta-1)\mathcal K(A)}},
\end{aligned}
\end{equation}
for all $n\geq N$, where the summations are over the pairs $(l,m)\in \I_{=\mathcal J(A), \mathcal K(A)}$.
In particular, \eqref{two-sided-large-deviation-combined} holds if $A$ is bounded away from $\mathbb D_{\ll \mathcal J(A), \mathcal K(A)}$.
\end{theorem}
\begin{proof}
Note first that from Lemma~\ref{lemma-for-two-sided-multiple-optima} (i), there exists a $\delta'>0$ such that $\mathbb D_{\ll \mathcal J(A), \mathcal K(A)}$ is bounded away from  $A \cap (\mathbb D_{l,m})_{\delta'}$  for all $(l,m) \in \I_{=\mathcal J(A), \mathcal K(A)}$. 
Moreover, applying Lemma~\ref{lemma-for-two-sided-multiple-optima} (ii) to each $A \cap (\mathbb D_{l,m})_{\delta'}$, we conclude that there exists $\rho>0$ such that $A\cap (\mathbb D_{l,m})_\rho$ is bounded away from $(\mathbb D_{j,k})_\rho$ for any two distinct pairs $(l,m),(j,k) \in \I_{=\mathcal J(A), \mathcal K(A)}$.
This means that $A\cap (\mathbb D_{l,m})_\rho$'s are all disjoint and bounded away from $\mathbb D_{<l,m}$. 

To derive the lower bound, we apply Theorem~\ref{thm:two-sided-main-theorem} to $A^\circ\cap (\mathbb D_{l,m})^{\rho}$ to obtain 
\begin{align*}
C_{l,m}(A^\circ) 
&= C_{l,m}(A^\circ\cap \mathbb D_{l,m} ) 
= 
C_{l,m}(A^\circ\cap \mathbb D_{l,m} \cap (\mathbb D_{l,m})^{\rho} )
\\
&=
C_{l,m}(A^\circ\cap (\mathbb D_{l,m})^{\rho} )
\leq  
\liminf_{n\to\infty} \frac{\P(\bar X_n \in A^\circ\cap (\mathbb D_{l,m})^\rho)}{(n\nu[n,\infty))^l(n\nu(-\infty,-n])^m}
\\
&
\leq
\liminf_{n\to\infty} \frac{\P(\bar X_n \in A\cap (\mathbb D_{l,m})^\rho)}{(n\nu[n,\infty))^l(n\nu(-\infty,-n])^m},
\end{align*}
for each $(l,m)\in \I_{=\mathcal J(A), \mathcal K(A)}$.
That is, for any given $\epsilon>0$, there exists an $N_{l,m}\in \mathbb N$ such that
\begin{equation}\label{eq:asympt-low-lm}
\begin{aligned}
\frac{\big(C_{l,m}( A^\circ)-\epsilon\big)L_+^l(n)L_-^m(n)}{n^{(\alpha-1)l+(\beta-1)m}}
\leq \P\big(\bar X_n \in A \cap (\mathbb D_{l,m})^\rho\big),
\end{aligned}
\end{equation}
for all $n\geq N_{l,m}$.
Meanwhile, an obvious bound holds for $A\setminus \bigcup_{(l,m)\in \I_{=\mathcal J(A), \mathcal K(A)}} \allowbreak (\mathbb D_{l,m})^\rho$; i.e.,
\begin{equation}\label{eq:asympt-low-0}
0\leq \P\left(\bar X_n \in \textstyle{A\setminus \bigcup_{(l,m)\in \I_{=\mathcal J(A), \mathcal K(A)}} (\mathbb D_{l,m})^\rho}\right).
\end{equation}
Since $(\alpha-1)l + (\beta-1)m = (\alpha-1)\mathcal J(A) + (\beta-1)\mathcal K(A)$ for $(l,m)\in \I_{=\mathcal J(A),\mathcal K(A)}$, summing (\ref{eq:asympt-low-lm}) over $(l,m) \in \I_{=\mathcal J(A), \mathcal K(A)}$ together with (\ref{eq:asympt-low-0}), we arrive at the lower bound of the theorem, with $N = \max_{(l,m)\in \I_{=\mathcal J(A), \mathcal K(A)}} N_{l,m}$. 

Turning to the upper bound, we apply Theorem~\ref{thm:two-sided-main-theorem} to $\bar A\cap (\mathbb D_{l,m})_\rho$ to get
\begin{align*}
\limsup_{n\to\infty} \frac{\P(\bar X_n \in \bar A \cap (\mathbb D_{l,m})_\rho )}{(n\nu[n,\infty))^l(n\nu(-\infty,-n])^m}
\leq C_{l,m}(\bar A\cap (\mathbb D_{l,m})_\rho) = C_{l,m}(\bar A).
\end{align*}
for each $(l,m) \in \I_{=\mathcal J(A), \mathcal K(A)}$.
That is, for any given $\epsilon>0$, there exists $N_{l,m}'\in \mathbb N$ such that
\begin{equation}\label{eq:asympt-up-lm}
\begin{aligned}
\P(\bar X_n \in A \cap (\mathbb D_{l,m})_\rho)
\leq
\frac{\big(C_{l,m}(\bar A\revrem{\cap (\mathbb D_{l,m})_\rho})+\epsilon/2\big)L_+^l(n)L_-^m(n)}{n^{(\alpha-1)l+(\beta-1)m}},
\end{aligned}
\end{equation}
for all $n\geq N_{l,m}'$.
On the other hand, since $\bar A\setminus \bigcup_{(l,m)\in \I_{=\mathcal J(A),\mathcal K(A)}}(\mathbb D_{l,m})^{\rho}
$ is closed and bounded away from $\mathbb D_{<\mathcal J(A), \mathcal K(A) }$,
\begin{equation*}
\begin{aligned}
\limsup_{n\to\infty} \frac{\P\left(\bar X_n \in  A\setminus \bigcup_{(l,m)}(\mathbb D_{l,m})^\rho \right)}{(n\nu[n,\infty))^{\mathcal J(A)}(n\nu(-\infty,-n])^{\mathcal K(A)}}
\leq C_{\mathcal J(A),\mathcal K(A)}\left( \textstyle{ \bar A\setminus \bigcup_{(l,m)}(\mathbb D_{l,m})^\rho} \right),
\end{aligned}
\end{equation*}
where the union is over the pairs $(l,m)\in \I_{=\mathcal J(A),\mathcal K(A)}$.
Therefore, there exists $N'$ such that
\begin{equation}\label{eq:asympt-up-0}
\begin{aligned}
&
\P\left(\bar X_n \in  \textstyle{A\setminus \bigcup_{(l,m)}(\mathbb D_{l,m})^\rho} \right)
\\
&
\leq
\frac{\left(C_{\mathcal J(A),\mathcal K(A)}\left(\textstyle{ \bar A\setminus \bigcup_{(l,m)}(\mathbb D_{l,m})^\rho}\right)+\epsilon/2\right)L_+^{\mathcal J(A)}(n)L_-^{\mathcal K(A)}(n)}{n^{(\alpha-1)\mathcal J(A)+(\beta-1)\mathcal K(A)}}
\\
&
=
\frac{\left(\epsilon/2\right)L_+^{\mathcal J(A)}(n)L_-^{\mathcal K(A)}(n)}{n^{(\alpha-1)\mathcal J(A)+(\beta-1)\mathcal K(A)}},
\end{aligned}
\end{equation}
for $n \geq N'$ since $\textstyle{ \bar A\setminus \bigcup_{(l,m)}(\mathbb D_{l,m})^\rho}$ is disjoint from the support of $C_{\mathcal J(A),  \mathcal K(A)}$.
Summing (\ref{eq:asympt-up-lm}) over $(l,m)\in \I_{=\mathcal J(A), \mathcal K(A)}$ and (\ref{eq:asympt-up-0}),
\begin{equation}
\P(\bar X_n \in A)
\leq
\frac{\sum_{(l,m)} \big(C_{l,m}\big(\bar A\revrem{\cap (\mathbb D_{l,m})_\delta}\big)+\epsilon\big)L_+^l(n)L_-^m(n)}{n^{(\alpha-1)\mathcal J(A)+(\beta-1)\mathcal K(A)}},
\end{equation}
for $n \geq N$, where $N = N'\vee\max_{(l,m)\in \I_{=\mathcal J(A), \mathcal K(A)}} N'_{l,m}$. \revrem{Taking $\delta \to 0$, we obtain the upper bound of the theorem.}

\end{proof}

%


\section{Implications}\label{sec:implications} This section explores the implications of the large-deviations results in Section~\ref{sec:sample-path-ldps}, and is organized as follows. Section~\ref{subsec:random-walks} proves a result similar to Theorem~\ref{thm:two-sided-main-theorem}, now focusing on random walks with regularly varying increments. Section~\ref{subsec:conditional-limit-theorem} illustrates that conditional limit theorems can easily be studied by means of the limit theorems established in Section~\ref{sec:sample-path-ldps}.  Section~\ref{subsec:weak-ldp} develops a weak large deviation priciple (LDP) of the form (\ref{weakldp}) for the scaled L\'evy processes. Finally, Section~\ref{subsec:nonexistence} shows that the weak LDP proved in Section~\ref{subsec:weak-ldp} is the best one can hope for in the presence of regularly varying tails, by showing that a full LDP of the form (\ref{weakldp}) does not exist.

\subsection{Random Walks}\label{subsec:random-walks}
Let $S_k, k\geq 0,$ be a random walk, set $\bar S_n(t) = S_{[nt]}/n, t\geq 0$, and define $\bar S_n = \{\bar S_n(t), t\in [0,1]\}$.
Let $N(t), t\geq 0,$ be an independent unit rate Poisson process. Define  the L\'evy process $X(t) \triangleq S_{N(t)}, t\geq 0$, and set $\bar X_n(t) \triangleq X(nt)/n, t\geq 0$.
The goal is to prove an analogue of Theorem~\ref{thm:two-sided-main-theorem} for the scaled random walk $\bar S_n$.
Let $\mathcal J(\cdot)$, $\mathcal K(\cdot)$, and $C_{j,k}(\cdot)$ be defined as in Section~\ref{subsec:two-sided-large-deviations}.
\begin{theorem}\label{thm:random-walk}
Suppose that $\P(S_1 \geq x)$ is regularly varying with index $-\alpha$ and $\P(S_1 \leq -x)$ is regularly varying with index $-\beta$. Let $A$ be a measurable set bounded away from $\mathbb D_{< \mathcal J(A), \mathcal K(A)}$.
Then
\begin{equation}\label{eq:random-walk}
\begin{aligned}
&
\liminf_{n\rightarrow\infty} \frac{\P(\bar S_n \in A) }{(n \P(S_1\geq n))^{\mathcal J(A)}(n \P(S_1\leq -n))^{\mathcal K(A)}}  
\geq 
C_{\mathcal J(A), \mathcal K(A)}(A^\circ),
\\
&
\limsup_{n\rightarrow\infty} \frac{ \P(\bar S_n \in A)}{(n \P(S_1\geq n))^{\mathcal J(A)} (n \P(S_1\leq -n))^{\mathcal K(A)} }
\leq
C_{\mathcal J(A),\mathcal K(A)}(\bar A).
\end{aligned}
\end{equation}
\end{theorem}
\begin{proof}
The idea is to combine our notion of asymptotic equivalence with Theorem~\ref{thm:two-sided-main-theorem}. First, we need to derive the asymptotic behavior of the L\'evy measure of the constructed L\'evy process.
From Example A3.17 in \cite{EmbrechtsKluppelbergMikosch97}, we obtain $\P(X(1)\geq x) \sim \P(S_1\geq x)$. Moreover, \cite{EmbrechtsVeraverbeke} implies that $\nu(x,\infty) \sim \P(X(1)\geq x)$.
Similarly, it follows that $\nu(-\infty,-x)\sim \P(S_1 \leq - x)$.

Now, from Lemma~\ref{wasteful-lemma}, \eqref{eq:random-walk} is proved if \eqref{eq:two-sided-limit-theorem} holds for $\bar S_n$.
In view of Corollary~\ref{lem:asymptotic-equivalence}, \eqref{eq:two-sided-limit-theorem} holds---and hence, the proof is completed---if we prove the asymptotic equivalence between $\bar X_n$ and $\bar S_n$ (w.r.t.\ a geometrically decaying sequence).
To prove the asymptotic equivalence, we first argue that the Skorokhod distance between $\bar S_n$ and $\bar X_n$ is bounded by $\sup_{t\in [0,1]} |N(tn)/n - t|$.
To see this, define the homeomorphism $\lambda_n(t)$ as the linear interpolation of the jump points of $N(nt)/n$, and observe that $\bar X_n(t) = \bar S_n (\lambda_n(t))$.
Thus, the distance between $\bar S_n$ and $\bar X_n$ is bounded by $\sup_{t\in [0,1]}  |\lambda_n(t)-t|$ which, in itself, is bounded by $\sup_{t\in [0,1]} |N(tn)/n - t|$.
From Lemma~\ref{lem:cont_etemadi},
\begin{equation*}
\P(\sup_{t\in [0,1]} |N(tn)/n - t|)>\delta) \leq 3\sup_{t\in [0,1]}\P( |N(tn)/n - t|)>\delta/3),
\end{equation*}
where $\P( |N(tn)/n - t|)>\delta/3)$ vanishes at a geometric rate w.r.t.\ $n$ uniform in $t\in [0,1]$, from which the asymptotic equivalence follows.
\end{proof}

\subsection{Conditional Limit Theorems}\label{subsec:conditional-limit-theorem}

\label{subsec:conditional-limit-theorem}


As before, $\bar{X}_{n}$ denotes the scaled L\'{e}vy process defined as in
Section~\ref{subsec:one-sided-large-deviations} for the one-sided case and
Section \ref{subsec:two-sided-large-deviations} for the two-sided case, respectively.
In this section, we present conditional limit theorems which give a precise description of the limit law of $\bar X_n$ conditional on $\bar X_n\in A$.

The next result, for the one-sided case, follows immediately from the definition of weak
convergence and Theorem \ref{thm:one-sided-main-theorem}.

\begin{corollary}
Suppose that a subset $B$ of $\mathbb{D}$ satisfies the conditions in
Theorem~\ref{thm:one-sided-main-theorem} and that $C_{\mathcal{J}(B)}%
(B^{\circ})=C_{\mathcal{J}(B)}(B)=C_{\mathcal{J}(B)}(\bar B)>0$. Let $\bar
{X}_{n}^{|B}$ be a process having the conditional law of $\bar{X}_{n}$ given
that $\bar{X}_{n}\in B$, then there exists a process $\bar{X}_{\infty}^{|B}$ such that
\[
\bar{X}_{n}^{|B}\Rightarrow\bar{X}_{\infty}^{|B},
\]
in $\mathbb{D}$. Moreover, if $\P^{|B}\left(  \cdot\right)  $ is the law of
$\bar{X}_{\infty}^{|B}$, then%
\[
\P^{|B}\left(  \bar{X}_{\infty}^{|B}\in
\cdot\right)  :=\frac{C_{\mathcal{J}%
(B)}(\cdot\cap B)}{C_{\mathcal{J}(B)}(B)}.
\]

\end{corollary}

Let us provide a more direct probabilistic description of the process $\bar
{X}_{\infty}^{|B}$. Directly from the definition of $\P^{|B}$ we have that
\[
\bar{X}_{\infty}^{|B}\left(  t\right)  =\sum_{n=1}^{\mathcal{J}(B)}\chi
_{n}1_{[U_{n},1]}\left(  t\right)  ,
\]
where $U_{1},...,U_{\mathcal{J}(B)}$ are i.i.d.\ uniform random variables on $[0,1]$ and
\begin{align*}
&  \P^{|B}\left(  \chi_{1}\in dx_{1},...,\chi_{\mathcal{J}(B)}\in
dx_{\mathcal{J}(B)}\right)  \\
&  =\frac{\Pi_{i=1}^{\mathcal{J}(B)}\left(  \alpha x_{i}{}^{-\alpha-1}%
dx_{i}\right)  \,\I\left(  x_{\mathcal{J}(B)}>...>x_{1}>0\right)  \P\left(
\sum_{n=1}^{\mathcal{J}(B)}x_{n}1_{[U_{n},1]}\left(  \cdot\right)  \in
B\right)  }{C_{\mathcal{J}(B)}(B)}.
\end{align*}

An easier to interpret description of $\P^{|B}$ can be obtained by using the
fact that $\delta_{B}:=d\left(  B,\mathbb{D}_{\leqslant\mathcal{J}(B)-1}\right)
>0$. Define an auxiliary probability measure, $\P_{\#}^{|B}$, under which, not
only $U_{1},...,U_{\mathcal{J}(B)}$ are i.i.d. Uniform$\left(  0,1\right)  $, but
also $\chi_{1},...,\chi_{\mathcal{J}(B)}$ are i.i.d. distributed
Pareto$\left(  \alpha,\delta_{B}\right)  $ and independent of the $U_{i}$'s;
that is,%
\begin{align*}
&\P_{\#}^{|B}\left(  \chi_{1}\in dx_{1},...,\chi_{\mathcal{J}(B)}\in
dx_{\mathcal{J}(B)}\right)  
\\
&
=(\alpha/\delta_{B})^{\mathcal{J}(B)}\Pi
_{i=1}^{\mathcal{J}(B)}(x_{i}/\delta_{B})^{-\alpha-1}dx_{i}\,\I\left(  x_{i}%
\geq\delta_{B}\right)  .
\end{align*}
Then, we have that
\begin{equation}
\P^{|B}\left(  \bar{X}_{\infty}^{|B}\in\cdot\right)  =\P_{\#}^{|B}\left(
\bar{X}_{\infty}^{|B}\in\cdot\text{ }|\text{ }\bar{X}_{\infty}^{|B}\in
B\right)  .\label{Rec_Cond}%
\end{equation}
Moreover, note that%
\begin{equation}
\P_{\#}^{|B}\left(  \bar{X}_{\infty}^{|B}\in B\right)  =\delta_{B}%
^{-\mathcal{J}(B)\left(  \alpha+2\right)  }C_{\mathcal{J}(B)}%
(B)>0.\label{Rec_Cond_2}%
\end{equation}

In view of (\ref{Rec_Cond}) and (\ref{Rec_Cond_2}) one can say, at least
qualitatively, that the most likely way in which the event $\bar{X}_{n}\in B$
is seen to occur is by means of $\mathcal{J}(B)$ i.i.d.\ jumps which are
suitably Pareto distributed and occurring uniformly throughout the time
interval $[0,1]$.

We now are ready to provide the corresponding conditional limit theorem for the two-sided case, building on
Theorem \ref{thm:two-sided-main-theorem}. The proof is
again immediate, using the definition of weak convergence.

\begin{corollary}
\label{cor:two-sided-conditional-probability_special_case} Suppose that a
subset $B$ of $\mathbb{D}$ satisfies the conditions in
Theorem~\ref{thm:two-sided-main-theorem} and that
\[
C_{\mathcal{J}(B),\mathcal{K}(B)}(B^{\circ})=C_{\mathcal{J}(B),\mathcal{K}%
(B)}(B)=C_{\mathcal{J}(B),\mathcal{K}(B)}(\bar B)>0.
\]
Let $\bar{X}_{n}^{|B}$ be a process having the conditional law of $\bar{X}%
_{n}$ given that $\bar{X}_{n}\in B$, then
\[
\bar{X}_{n}^{|B}\Rightarrow\bar{X}_{\infty}^{|B},
\]
in $\mathbb{D}$. Moreover, if $\P^{|B}\left(  \cdot\right)  $ is the law of
$\bar{X}_{\infty}^{|B}$, then%
\[
\P^{|B}\left(  \bar{X}_{\infty}^{|B}\in\cdot\right)  :=\frac{C_{\mathcal{J}%
(B),\mathcal{K}(B)}(\cdot\cap B)}{C_{\mathcal{J}(B),\mathcal{K}(B)}(B)}.
\]

\end{corollary}

A probabilistic description, completely analogous to that given for the
one-sided case, can also be provided in this case. Define $\delta_{B}=d\left(
B,\mathbb{D}_{<\mathcal{J}(B),\mathcal{K}(B)}\right)  >0$ and introduce a
probability measure $\P_{\#}^{|B}$ under which we have the following: First,
$U_{1},...,U_{\mathcal{J}(B)},V_{1},...,V_{\mathcal{K}(B)}$ are i.i.d.
$U\left(  0,1\right)  $; second, $\chi_{1},...,\chi_{\mathcal{J}(B)}$ are
i.i.d. Pareto($\alpha,\delta_{B}$), and, finally $\varrho_{1},...,\varrho
_{\mathcal{K}(B)}$ are i.i.d. Pareto($\beta,\delta_{B}$) random variables (all
of these random variables are mutually independent). Then, write
\[
\bar{X}_{\infty}^{|B}\left(  t\right)  =\sum_{n=1}^{\mathcal{J}(B)}\chi
_{n}1_{[U_{n},1]}\left(  t\right)  -\sum_{n=1}^{\mathcal{K}(B)}\varrho
_{n}1_{[V_{n},1]}\left(  t\right)  .
\]
Applying the same reasoning as in the one sided case we have that
\[
\P^{|B}\left(  \bar{X}_{\infty}^{|B}\in\cdot\right)
=\P_{\#}^{|B}\left(
\bar{X}_{\infty}^{|B}\in\cdot\text{ }|\text{ }\bar{X}_{\infty}^{|B}\in
B\right)
\]
and
\[
\P_{\#}^{|B}\left(  \bar{X}_{\infty}^{|B}\in B\right)  =\delta_{B}%
^{-\mathcal{J}(B)\left(  \alpha+2\right)  -\mathcal{K}(B)\left(
\beta+2\right)  }C_{\mathcal{J}(B),\mathcal{K}(B)}(B)>0.
\]

We note that these results also hold for random walks, and thus is a significant extension of Theorem 3.1 in \cite{durrett1980conditioned}, where it is assumed that $\alpha>2$ and $B=\{\bar{X}_{n}\left(1\right)  \geq a\}$.

\subsection{Large Deviation Principle}\label{subsec:weak-ldp}

In this section, we show that $\bar X_n$ satisfies a weak large deviation principle with speed $\log n$,
and a rate function which is piece-wise linear in the number of discontinuities.
More specifically, define
\begin{equation}\label{eq:rate-function}
I(\xi)\triangleq
\left\{\begin{array}{ll}
	(\alpha-1)\mathcal D_+(\xi) + (\beta-1)\mathcal D_-(\xi), & \text{if $\xi$ is a step function \& $\xi(0) = 0$;}
	\\
	\infty, & \text{otherwise.}
\end{array}
\right.
\end{equation}
where $\mathcal D_-(\xi)$ denotes the number of downward jumps in $\xi$.
\begin{theorem}\label{thm:weak-ldp}
The scaled process $\bar X_n$ satisfies the weak large deviation principle with rate function $I$ and  speed  $\log n$, i.e.,
\begin{equation}\label{eq:ldp-lower-bound}
-\inf_{x \in G} I(x)
\leq
\liminf_{n\to\infty} \frac{\log \P(\bar X_n \in G)}{\log n}
\end{equation}
for every open set $G$, and
\begin{equation}\label{eq:ldp-upper-bound}
\limsup_{n\to\infty} \frac{\log \P(\bar X_n \in K)}{\log n}
\leq
-\inf_{x\in K} I(x)
\end{equation}
for every compact set $K$.
\end{theorem}
The proof of Theorem~\ref{thm:weak-ldp} is provided in Section~\ref{subsec:proofs-for-implications}. It is based on Theorem~\ref{thm:two-sided-main-theorem}, and a reduction of the case of general $A$ to open neighborhoods; reminiscent of arguments made in the proof of Cram\'ers theorem \cite{dembozeitouni}.

\subsection{Nonexistence of Strong Large Deviation Principle}\label{subsec:nonexistence}
We conclude the current section by showing that the weak LDP presented in the previous section is the best one can hope for in our setting, in the sense that for any L\'evy process $X$ with a regularly varying L\'evy measure, $\bar X_n$ cannot satisfy a strong LDP; i.e., (\ref{eq:ldp-upper-bound}) in Theorem~\ref{thm:weak-ldp} cannot be extended to all closed sets.

Consider a mapping $\pi:\mathbb D \to \R_+^2$ that maps paths in $\mathbb D$ to their largest jump sizes, i.e.,
$$
\pi(\xi) \triangleq \Big(\sup_{t\in (0,1]} \big(\xi(t)-\xi(t-)\big), \sup_{t\in(0,1]} \big(\xi(t-) - \xi(t)\big)\Big).
$$
Note that $\pi$ is continuous, since each coordinate is continuous: for example, if the first coordinate (the largest upward jump sizes) of $\pi(\xi)$ and $\pi(\zeta)$ differ by $\epsilon$ then $d(\xi,\zeta) \geq \epsilon/2$, which implies that the first coordinate is continuous.
Now, to derive a contradiction, suppose that $\bar X_n$ satisfies a strong LDP. In particular, suppose (\ref{eq:ldp-upper-bound}) in Theorem~\ref{thm:weak-ldp} is true for all closed sets rather than just compact sets.
Since $\pi$ is continuous w.r.t.\ the $J_1$ metric, $\pi(\bar X_n)$ has to satisfy a strong LDP with rate function
$I'(y) = \inf \{I(\xi): \xi\in \mathbb D, y=\pi (x)\}$ by the contraction principle, in case $I'$ is a rate function. (Since $I$ is not a good rate function, $I'$ is not automatically guaranteed to be a rate function per se; see, for example, Theorem~4.2.1 and the subsequent remarks of \citealp{dembozeitouni}.) From the exact form of $I'$, given by
$$
I'(y_1,y_2)  = (\alpha-1)\I(y_1>0) + (\beta-1)\I(y_2>0),
$$
one can check that $I'$ indeed happens to be a rate function. For the sake of simplicity, suppose that $\alpha = \beta = 2$, and $\nu[x,\infty) = \nu(-\infty,-x] = x^{-2}$.
Let
$
\hat J_n^{\leqslant 1} \triangleq \frac1n  Q_n^\gets (\Gamma_1)1_{[U_1,1]}
$
and
$
\hat K_n^{\leqslant 1} \triangleq \frac1n R_n^\gets (\Delta_1)1_{[V_1,1]}
$
where
$
Q_n^{\gets}(y) \triangleq \inf\{s>0: n\nu[s,\infty)< y\}= \left(n/y\right)^{1/2}
$ and
$R_n^{\gets}(y) \triangleq \inf\{s>0: n\nu(-\infty,-s]< y\}= \left(n/y\right)^{1/2}$.
%
%
%
The random variables $\Gamma_1$ and $\Delta_1$ are standard exponential, and $U_1,V_1$ uniform $[0,1]$ (see also Section~\ref{sec:proofs} for similar and more general notational conventions).
Note that $\bar Y_n\triangleq (\hat J_n^{\leqslant 1}, \hat K_n^{\leqslant 1})$ is exponentially equivalent to $\pi(\bar X_n)$ if we couple $\pi(\bar X_n)$ and $(\hat J_n^{\leqslant 1}, \hat K_n^{\leqslant 1})$, using the representation of $\bar X_n$ as in (\ref{eq:Poisson-Jump-representation}): for any $\delta>0$, $\P\big( |\bar Y_n - \pi(\bar X_n)|>\delta \big) \leq \P\big(\bar Y_n \neq \pi(\bar X_n)\big) = \P\big( Q_n^{\gets} (\Gamma_1)\leq 1 \text{ or }  R_n^{\gets} (\Delta_1)\leq 1\big) $, which decays at an exponential rate. Hence,
$$
\frac{\log \P\big( |\bar Y_n - \pi(\bar X_n)|>\delta \big)}{\log n}\to -\infty,
$$
as $n\to \infty$, where $|\cdot|$ is the Euclidean distance. As a result, $\bar Y_n$ should satisfy the same (strong) LDP as $\pi(\bar X_n)$.
Now, consider the set $A \triangleq \bigcup_{k=2}^\infty [\log k, \infty) \times [k^{-1/2},\infty)$. Then, since $[\log k, \infty) \times [k^{-1/2},\infty) \subseteq A$ for $k\geq 2$,
\begin{align*}
\P(\bar Y_n \in A)
&
\geq \P\big((\hat J_n^{\leqslant 1}, \hat K_n^{\leqslant 1}) \in [\log n, \infty) \times[n^{-1/2},\infty)\big)
\\
&
= \P\big( Q_n^{\gets}(\Gamma_1) > n\log n, R_n^{\gets}(\Delta_1) > n^{1/2}\big)
\\
&
= \P\left( \left(\frac{n}{\Gamma_1}\right)^{1/2} > n\log n, \left(\frac{n}{\Delta_1}\right)^{1/2}  > n^{1/2}\right)
\\
&
= \P\left(\Gamma_1 < \frac{1}{n(\log n)^2}\right)\P( \Delta_1 < 1)
\\
&
= (1-e^{- \frac{1}{n(\log n)^2}})(1-e^{-1}).
\end{align*}
Thus,
\begin{equation}\label{eq:no-full-ldp-1}
\begin{aligned}
\limsup_{n\to\infty} \P(\bar Y_n \in A)
&
\geq
\limsup_{n\to\infty} \frac{\log (1-e^{- \frac{1}{n(\log n)^2}})(1-e^{-1})}{\log n}
\\
&
\geq
\limsup_{n\to\infty} \frac{\log \frac{1}{n(\log n)^2} (1-\frac{1}{2n(\log n)^2})(1-e^{-1})}{\log n}
\\
&
= -1.
\end{aligned}
\end{equation}
On the other hand, since $A \subseteq (0,\infty) \times (0,\infty)$,
\begin{equation}\label{eq:no-full-ldp-2}
-\inf_{(y_1,y_2)\in A} I'(y_1,y_2) = -2.
\end{equation}
Noting that $A$ is a closed (but not compact) set, we arrive at a contradiction to the large deviation upper bound for $\bar Y_n$. This, in turn, proves that $\bar X_n$  cannot satisfy a full LDP.

\section{Proofs}\label{sec:proofs}
Section~\ref{subsec:proofs-for-M-convergence},
Section~\ref{subsec:proofs-for-sample-path-ldps},
and
Section~\ref{subsec:proofs-for-implications}
provide proofs of the results in
Section~\ref{sec:preliminaries},
Section~\ref{sec:sample-path-ldps},
and
Section~\ref{sec:implications},
respectively.

\subsection{Proofs of Section~\ref{sec:preliminaries}}\label{subsec:proofs-for-M-convergence}
Recall that $F_\delta = \{x \in \mathbb S: d(x,F) \leq \delta\}$ and $G^{-\delta} = ((G^c)_\delta)^c$.
\begin{proof}[Proof of Lemma~\ref{lem:extended-bounds}]
Let $G$ be an open set such that $G\cap \mathbb S_0$ is bounded away from $\mathbb C$. 
For a given $\delta>0$, due to the assumed asymptotic equivalence, $\P(X_n \in G^{-\delta}, d(X_n,Y_n) \geq \delta) = o(\epsilon_n)$. 
Therefore,
\begin{equation}\label{eq:asymp-equiv-lower-bound}
\begin{aligned}
&
\liminf_{n\to\infty} \epsilon_n^{-1}\P(Y_{n} \in G)
\\
&
\geq
\liminf_{n\to\infty} \epsilon_n^{-1}\P\left(X_n \in G^{-\delta}, d(X_n,Y_n) < \delta\right)
\\&
=
\liminf_{n\to\infty} \epsilon_n^{-1}\left\{\P\left(X_n \in G^{-\delta}\right)- \P\left(X_n\in G^{-\delta},d(X_n,Y_n) \geq \delta\right) \right\}
\\&
= \liminf_{n\to\infty} \epsilon_n^{-1}\P\left(X_n \in G^{-\delta}\right)
\end{aligned}
\end{equation}
Pick $r>0$ such that $G^{-\delta}\cap \mathbb S_0 \cap \mathbb C_r = 0$ and note that $G^{-\delta}\cap {\mathbb C_r}^\mathsf{c}$ is an open set bounded away from $\mathbb C$.
Then,
\begin{align*}
\liminf_{n\to\infty} \epsilon_n^{-1}\P(X_n \in G^{-\delta})
&
=
\liminf_{n\to\infty} \epsilon_n^{-1}\P(X_n \in G^{-\delta}\cap \mathbb S_0)
\\
&
=
\liminf_{n\to\infty} \epsilon_n^{-1}\P(X_n \in G^{-\delta}\cap \mathbb S_0 \cap {\mathbb C_r}^\mathsf{c})
\\
&
=
\liminf_{n\to\infty} \epsilon_n^{-1}\P(X_n \in G^{-\delta}\cap {\mathbb C_r}^\mathsf{c})
\geq
\mu(G^{-\delta}\cap {\mathbb C_r}^\mathsf{c})
\\
&
=
\mu(G^{-\delta}\cap {\mathbb C_r}^\mathsf{c}\cap \mathbb S_0)
=
\mu(G^{-\delta}\cap \mathbb S_0)
=
\mu(G^{-\delta}).
\end{align*}
Since $G$ is an open set, $G = \bigcup_{\delta>0} G^{-\delta}$. Due to the continuity of measures,
$\lim_{\delta\to 0}\mu(G^{-\delta}) = \mu(G),$
and hence, we arrive at the lower bound
\begin{align*}
&\liminf_{n\to\infty} \epsilon_n^{-1}\P(Y_n \in G) \geq \mu(G)
\end{align*}
by taking $\delta \to 0$. 

Now, turning to the upper bound,
consider a closed set $F$ such that $F_\delta \cap \mathbb S_0$ is bounded away from $\mathbb C$. 
Given a $\delta>0$, by the equivalence assumption, $\P(Y_n \in F, d(X_n,Y_n) \geq \delta) = o(\epsilon_n)$. Therefore,
\begin{equation}\label{eq:asymp-equiv-upper-bound}
\begin{aligned}
&
\limsup_{n\to\infty} \epsilon_n^{-1}\P(Y_{n} \in F)
\\
&
=
\limsup_{n\to\infty} \epsilon_n^{-1}\big\{\P\left(Y_n \in F,\, d(X_n,Y_n) < \delta\right) 
\\
&\qquad\qquad\qquad\qquad
+ \P\left(Y_n \in F,\, d(X_n,Y_n) \geq \delta\right)\big\}
\\&
= \limsup_{n\to\infty} \epsilon_n^{-1}\P\left(X_n \in F_{\delta}\right)
%
=
\limsup_{n\to\infty}\epsilon_n^{-1}\P(X_n \in F_\delta \cap \mathbb S_0)
\\
&
\leq
\limsup_{n\to\infty}\epsilon_n^{-1}\P(X_n \in \overline{F_\delta \cap \mathbb S_0}\,)
\leq
\mu\big(\,\overline{F_\delta \cap \mathbb S_0}\,\big)
=\mu\big(\,\overline{F_\delta \cap \mathbb S_0}\,\cap \mathbb S_0\big)
\\
&
\leq\mu\big(\bar{F}_\delta\cap \mathbb S_0\big)
=\mu(\bar F_\delta) = \mu(F_\delta).
\end{aligned}
\end{equation}
Note that $\{F_\delta\}$ is a decreasing sequence of sets, $F = \bigcap_{\delta>0} F_\delta$ (since $F$ is closed), and $\mu\in \mathbb M(\mathbb S\setminus \mathbb C)$ (and hence $\mu$ is a finite measure on $\mathbb S\setminus \mathbb C^r$ for some  $r>0$ such that $F_\delta \subseteq \mathbb S\setminus \mathbb C^r$ for some $\delta>0$). Due to the continuity (from above) of finite measures,
$
\lim_{\delta \to 0}  \mu(F_{\delta}) = \mu(F).
$
Therefore, we arrive at the upper bound
\begin{align*}
&\limsup_{n\to\infty} \epsilon_n^{-1}\P(X_n \in F) \leq \mu(F)
\end{align*}
by taking $\delta \to 0$.
\end{proof}

\revadd{
For a measure $\mu$ on a measurable space $\mathbb S$, denote the restriction of $\mu$ to a subspace $\mathbb O\subseteq \mathbb S$ with $\mu_{|\mathbb O}$.

\begin{proof}[Proof of Lemma~\ref{thm:simple-product-space}]
We provide a proof for $d=2$ which suffices for the application in this article. The extension to general $d$ is straightforward, and hence, omitted.
In view of the Portmanteau theorem for $\mathbb M$-convergence---in particular item \emph{(v)} of Theorem 2.1 of \cite{LRR}---it is enough to show that for all but countably many $r>0$, $(\mu^{(1)}_n\times\mu^{(2)}_n)_{|(\mathbb S_1\times\mathbb S_2)\setminus ((\mathbb C_1\times\mathbb S_2)\cup (\mathbb S_1\times\mathbb C_2))^r}(\cdot)$ converges to $(\mu^{(1)}\times\mu^{(2)}\allowbreak)_{|(\mathbb S_1\times\mathbb S_2)\setminus ((\mathbb C_1\times\mathbb S_2)\cup (\mathbb S_1\times\mathbb C_2))^r}(\cdot)$ weakly on $(\mathbb S_1\times\mathbb S_2)\setminus \big((\mathbb C_1\times\mathbb S_2)\cup (\mathbb S_1\times\mathbb C_2)\big)^r$, which is equipped with the relative topology as a subspace of $\mathbb S_1\times\mathbb S_2$.
From the assumptions of the lemma and again by Portmanteau theorem for $\mathbb M$-convergence, we note that ${\mu^{(1)}_n}_{|{\mathbb S_1\setminus\mathbb C_1^r}}$ converges to ${\mu^{(1)}}_{|\mathbb S_1\setminus\mathbb C_1^r}$ weakly in $\mathbb S_1 \setminus \mathbb C_1^r$, and ${\mu^{(2)}_n}_{|\mathbb S_2\setminus\mathbb C_2^r}$ converges to ${\mu^{(2)}}_{|\mathbb S_2\setminus\mathbb C_2^r}$ weakly in $\mathbb S_2 \setminus \mathbb C_2^r$ for all but countably many $r>0$.
For such $r$'s, ${\mu^{(1)}_n}_{|\mathbb S_1\setminus\mathbb C_1^r}\times {\mu^{(2)}_n}_{|\mathbb S_2\setminus\mathbb C_2^r}$ converges weakly to ${\mu^{(1)}}_{|\mathbb S_1\setminus\mathbb C_1^r}\times {\mu^{(2)}}_{|\mathbb S_2\setminus\mathbb C_2^r}$ in $\big(\mathbb S_1 \setminus \mathbb C_1^r\big)\times\big(\mathbb S_2 \setminus \mathbb C_2^r\big)$.
Noting that $(\mathbb S_1\times\mathbb S_2)\setminus \big((\mathbb C_1\times\mathbb S_2)\cup (\mathbb S_1\times\mathbb C_2)\big)^r$ coincides with $\big(\mathbb S_1 \setminus \mathbb C_1^r\big)\times\big(\mathbb S_2 \setminus \mathbb C_2^r\big)$, and ${\mu^{(1)}}_{|\mathbb S_1\setminus\mathbb C_1^r} \times{\mu^{(2)}}_{|\mathbb S_2\setminus \mathbb C_2^r}$ and ${\mu_n^{(1)}}_{|\mathbb S_1\setminus \mathbb C_1}\times{\mu_n^{(2)}}_{|\mathbb S_2\setminus \mathbb C_2}$ coincide with $(\mu^{(1)}\times\mu^{(2)})_{|(\mathbb S_1\times \mathbb S_2)\setminus((\mathbb C_1\times \mathbb S_2)\cup(\mathbb S_1 \times \mathbb C_2))^r}$ and $({\mu_n^{(1)}}\times {\mu_n^{(2)}})_{|(\mathbb S_1\times \mathbb S_2)\setminus((\mathbb C_1\times \mathbb S_2)\cup(\mathbb S_1 \times \mathbb C_2))^r}$, respectively, we reach the conclusion.
\end{proof}

\begin{proof}[Proof of Lemma~\ref{thm:union-limsup}]
Starting with the upper bound, suppose that $F$ is a closed set bounded away from $\bigcap_{i=0}^m\mathbb C(i)$. From the assumption, there exist $r_0,\ldots,r_m$ such that $F\subseteq \bigcup_{i=0}^m (\mathbb S \setminus \mathbb C(i)^{r_i})$, and hence,
\begin{align*}
\limsup_{n\to\infty}\frac{\P(X_n\in F)}{{\epsilon_n{(0)}}}
&\leq
\limsup_{n\to\infty}\sum_{i=0}^m\frac{\P\big(X_n\in F \cap (\mathbb S \setminus \mathbb C(i)^{r_i})\big)}{{\epsilon_n{(i)}}}\frac{{\epsilon_n{(i)}}}{{\epsilon_n{(0)}}}
\\&
\leq
\limsup_{n\to\infty}\sum_{i=0}^m\frac{\P(X_n\in F\setminus \mathbb C(i)^{r_i})}{{\epsilon_n{(i)}}}\frac{{\epsilon_n{(i)}}}{{\epsilon_n{(0)}}}
\\
&=
\limsup_{n\to\infty}\frac{\P(X_n\in F\setminus \mathbb C(0)^{r_0})}{{\epsilon_n{(0)}}}
\\
&
\leq
\mu^{(0)}(F\setminus \mathbb C(0)^{r_0})
\leq
\mu^{(0)}(F)
\end{align*}
Turning to the lower bound, if $G$ is an open set bounded away from $\bigcap_{i=0}^m\mathbb C(i)$,
\begin{align*}
\liminf_{n\to\infty}\frac{\P(X_n\in G)}{{\epsilon_n{(0)}}}
&\geq
\liminf_{n\to\infty}\frac{\P(X_n\in G  \setminus \mathbb C(0)_{r})}{{\epsilon_n{(0)}}}
\geq \mu^{(0)}(G  \setminus \mathbb C(0)_{r}).
\end{align*}
Taking $r\to 0$ yields the lower bound.

\end{proof}
}

\begin{proof}[Proof of Lemma~\ref{lem:almost-continuous-mapping}]
Suppose that $\mu_n \to \mu$ in $\mathbb M(\mathbb S\setminus \mathbb C)$, and $\mu(D_h\setminus \mathbb C^r) = 0$ and $\mu(\partial \mathbb S_0 \setminus \mathbb C^r)= 0$ for each $r>0$. 
Note that $\partial h^{-1}(A') \subseteq \mathbb S \setminus \mathbb C^r$ for some $r>0$ due to the assumption, and $\partial h^{-1}(A') \subseteq h^{-1}(\partial A') \cup D_h \cup \partial \mathbb S_0$.
Therefore, 
$\mu( \partial h^{-1}(A')) \leq \mu\circ h^{-1}(\partial A') + \mu(D_h\setminus \mathbb C^{r}) + \mu(\partial \mathbb S_0 \setminus \mathbb C^{r}) = 0$. 
Applying Theorem 2.1 (iv) of \cite{LRR} for $h^{-1}(A')$, 
we conclude that $\mu_n(h^{-1}(A')) \to \mu(h^{-1}(A'))$. 
Again, by Theorem 2.1 (iv) of \cite{LRR}, this means that $\mu_n\circ h^{-1} \to \mu \circ h^{-1}$ in $\mathbb M(\mathbb S'\setminus \mathbb C')$, and hence, $\hat h$ is continuous at $\mu$.

\end{proof}

\revrem{
\begin{proof} [Proof of Lemma~\ref{lemma:two-sided-continuous-mapping}]
The first two claims are straightforward analogies of Result~\ref{result:L53}.
As a consequence of those claims, we can apply Lemma~\ref{lem:almost-continuous-mapping} with $\mathbb S_0 = S_{j,k}$ and $h = T_{j,k}$ to conclude the proof of the last claim.
\end{proof}
}

\revadd{
\begin{proof}[Proof of Lemma~\ref{lem:continuous-mapping-principle-for-subtraction}]
The continuity of $h$ is well known; see, for example, \cite{whitt1980some}.
For the second claim, it is enough to prove that for each $j$ and $k$, $h^{-1}(A)\subseteq \mathbb D\times \mathbb D$ is bounded away from $\mathbb D_{j}\times \mathbb D_{k}$ whenever $A\subseteq \mathbb D$ is bounded away from $\mathbb D_{j,k}$.
Given $j$ and $k$, let $A\subseteq \mathbb D$ be bounded away from $\mathbb D_{j,k}$.
To prove that $h^{-1}(A)$ is bounded away from $\mathbb D_j\times \mathbb D_k$ by contradiction, suppose that it is not.
Then, for any given $\epsilon>0$, one can find $\xi\in \mathbb D$ and $\zeta \in \mathbb D$ such that
$d(\xi,\mathbb D_j) < \epsilon/2$, $d(\zeta,\mathbb D_k)< \epsilon/2$, and $\xi-\zeta \in A$.
Since a time-change of a step function doesn't change the number of jumps and jump-sizes, there exist $\xi'\in \mathbb D_j$ and $\zeta' \in \mathbb D_k$ such that $\|\xi - \xi'\|_\infty < \epsilon/2$ and $\|\zeta - \zeta'\|_\infty < \epsilon/2$.
Therefore, $
d(\xi-\zeta,\xi'-\zeta') \leq \| (\xi - \zeta) - (\xi'-\zeta')\|_\infty \leq \|\xi - \xi'\|_\infty + \|\zeta - \zeta'\|_\infty < \epsilon.
$
From this along with the property $d(\xi'-\zeta', \mathbb D_{j,k}) = 0$, we conclude that
$
d(\xi-\zeta, \mathbb D_{j,k}) < \epsilon
$.
Taking $\epsilon\to 0$, we arrive at $d(A, \mathbb D_j \times \mathbb D_k) = 0$ which is contradictory to the assumption.
\end{proof}
}

\begin{proof}[Proof of Lemma~\ref{lem:convergence-determining-class}]
From (i) and the inclusion-exclusion formula, $\mu_n(\bigcup_{i=1}^m A_i) \to \mu(\bigcup_{i=1}^m A_i)$ as $n\to \infty$ for any finite $m$ if $A_i \in \mathcal A_p$ is bounded away from $\mathbb C$ for $i=1,\ldots,m$.
If $G$ is open and bounded away from $\mathbb C$, there is a sequence of sets $A_i, i\geq 1$ in $\mathcal A_p$ such that $G=\bigcup_{i=1}^\infty A_i$;  note that since $G$ is bounded away from $\mathbb C$, $A_i$'s are also bounded away from $\mathbb C$.
For any $\epsilon >0$, one can find $M_\epsilon$ such that $\mu(\bigcup_{i=1}^{M_\epsilon} A_i) \geq \mu( G)-\epsilon$, and hence,
$$\liminf_{n\to\infty} \mu_n(G) \geq \liminf_{n\to\infty} \mu_n(\bigcup_{i=1}^{M_\epsilon} A_i) = \mu(\bigcup_{i=1}^{M_\epsilon} A_i) \geq \mu(G) -\epsilon.$$
Taking $\epsilon \to 0$, we arrive at the lower bound (\ref{result1lb}). Turning to the upper bound, given a closed set $F$, we pick $A\in \mathcal A_p$ bounded away from $\mathbb C$ such that $F\subseteq A^\circ$. Then,
\begin{align*}
\mu(A) - \limsup_{n\to\infty} \mu_n(F)
&
= \lim_{n\to\infty}\mu_n(A) + \liminf_{n\to\infty} (-\mu_n(F))
\\
&= \liminf_{n\to\infty} (\mu_n(A)-\mu_n(F))
= \liminf_{n\to\infty} \mu_n(A\setminus F)
\\
&
\geq \liminf_{n\to\infty} \mu_n(A^\circ \setminus F)
\geq \mu(A^\circ\setminus F)
\\
&
= \mu(A) - \mu(F).
\end{align*}
Note that $\mu(A)<\infty$ since $A$ is bounded away from $\mathbb C$, which together with the above inequality
establishes the upper bound (\ref{result1lb}).
\end{proof}

\subsection{Proofs of Section~\ref{sec:sample-path-ldps}}\label{subsec:proofs-for-sample-path-ldps}
This section provides the proofs for the limit theorems (Theorem~\ref{thm:one-sided-limit-theorem}, Theorem~\ref{thm:two-sided-limit-theorem}) presented in Section~\ref{sec:sample-path-ldps}. The proof of Theorem~\ref{thm:one-sided-limit-theorem} is based on
\begin{itemize}
\item[1.] The asymptotic equivalence between the target object $\bar X_n$ and the process obtained by keeping its $j$ largest jumps, which will be denoted as $J_n^{\leqslant j}$: Proposition~\ref{prop:asymptotic-equivalence-Xbar-Jbar} and Proposition~\ref{prop:asymptotic-equivalence-Jbar-Jj} prove such asymptotic equivalences.
Two technical lemmas (Lemma~\ref{lem:key-upper-bound-1} and Lemma~\ref{lem:key-upper-bound-2}) play key roles in Proposition~\ref{prop:asymptotic-equivalence-Jbar-Jj}.
\item[2.] $\mathbb M$-convergence  of $J_n^{\leqslant j}$: Lemma~\ref{lem:poisson-jumps} identifies the convergence of jump size sequences, and Proposition~\ref{prop:Jj} deduces the convergence of $J_n^{\leqslant j}$ from the convergence of the jump size sequences via the mapping theorem established in Section~\ref{sec:preliminaries}.
\end{itemize}
For Theorem~\ref{thm:two-sided-limit-theorem}, 
we first establish a general result (Theorem~\ref{thm:multi-d-limit-theorem}) for the $\mathbb M$-convergence of multiple L\'evy processes in the associated product space using Lemma~\ref{thm:simple-product-space} and \ref{thm:union-limsup}. 
We then apply Lemma~\ref{lem:continuous-mapping-principle-for-subtraction} to prove Theorem~\ref{thm:two-sided-limit-theorem}.


Recall that $X_n(t) \triangleq X(nt)$ is a scaled process of $X$,
where $X$ is a L\'evy process with a L\'evy measure $\nu$ supported on $(0,\infty)$.
Also recall that $X_n$ has It\^{o} representation
\begin{align}\label{eq:ito_rep_onesided2}
X_n(s) 
&
= nsa + B(ns) + \int_{|x|\leq 1} x[N([0,ns]\times dx) - ns\nu(dx)] 
\\
&
\hspace{75pt}
+ \int_{|x|>1} xN([0,ns]\times dx),
\nonumber
\end{align}
where $N$ is the Poisson random measure with mean measure Leb$\times\nu$ on $[0,n]\times(0,\infty)$ and Leb denotes the Lebesgue measure.
It is easy to see that
\begin{align*}
J_n(s)
&
\triangleq \sum_{l=1}^{\tilde N_n} Q_n^\gets (\Gamma_l)1_{[U_l,1]}(s)
\stackrel{\mathcal D}{=} \int_{|x|>1} xN([0,ns]\times dx),
\end{align*}
where $\Gamma_l = E_1 + E_2 + ... + E_l$; $E_i$'s are i.i.d.\ and standard exponential random variables; $U_l$'s are i.i.d.\ and uniform variables in $[0,1]$; $\tilde N_n  = N_n\big([0,1]\times [1,\infty)\big)$;
$
N_n = \sum_{l=1}^\infty \delta_{(U_l, Q_n^{\gets}(\Gamma_l))},
$
where $\delta_{(x,y)}$ is the Dirac measure concentrated on $(x,y)$;
$
Q_n(x) \triangleq  n\nu[x,\infty)$,
$
Q_n^{\gets}(y) \triangleq \inf\{s>0: n\nu[s,\infty)< y\}
$.
Note that $\tilde N_n$ is the number of $\Gamma_l$'s  such that $\Gamma_l \leq n\nu_1^+$, where $\nu_1^+ \triangleq \nu[1,\infty)$, and hence, $\tilde N_n \sim \text{Poisson}(n\nu_1^+)$.
Throughout the rest of this section, we use the following representation for the centered and scaled process $\bar X_n \triangleq \frac{1}{n} X_n$:
\begin{align}\label{eq:Poisson-Jump-representation}
\bar X_n(s) 
&\stackrel{\mathcal D}{=} 
\frac1n J_n(s) + \frac1n B(ns) 
\\
&\hspace{15pt}
+ \frac1n\int_{|x|\leq 1} x[N([0,ns]\times dx) - ns\nu(dx)] - (\mu_1^+\nu_1^+)s.
\nonumber
\end{align}
\begin{proof}[Proof of Theorem~\ref{thm:one-sided-limit-theorem}]
We decompose $\bar X_n$ into a centered compound Poisson process $\bar J_n$, a centered L\'evy process with small jumps and continuous increments $\bar Y_n$, and a residual process that arises due to centering $\bar Z_n$. After that, we will show that the compound Poisson process determines the limit. More specifically, consider the following decomposition:
\begin{equation}\label{eq:main-decomposition}
\begin{aligned}
\bar X_n(s)
&\stackrel{\D}{=}
\bar Y_n(s) + \bar J_n(s) + \bar Z_n(s),\\
\bar Y_n(s) &\triangleq \frac1n B(ns) + \frac1n \int_{|x|\leq 1} x[N([0,ns]\times dx) - ns\nu(dx)],\\
\bar J_n(s) &\triangleq \frac1n \sum_{l=1}^{\tilde N_n} (Q_n^{\gets}(\Gamma_l)-\mu_1^+)1_{[U_l,1]}(s),\\
\bar Z_n(s) &\triangleq \frac1n \sum_{l=1}^{\tilde N_n} \mu_1^+ 1_{[U_l, 1]}(s) - \mu_1^+\nu_1^+s,
\end{aligned}
\end{equation}
where $\mu_1^+ \triangleq \frac{1}{\nu_1^+}\int_{[1,\infty)} x\nu(dx)$.
Let $\hat J_n^{\leqslant j} \triangleq \frac1n \sum_{l=1}^j Q_n^\gets(\Gamma_l)1_{[U_l, 1]}$ be, roughly speaking, the process obtained by just keeping the $j$ largest (un-centered) jumps of $\bar J_n$.
 In view of Corollary~\ref{lem:asymptotic-equivalence} and Proposition~\ref{prop:Jj}, it suffices to show that $\bar X_n$ and $\hat J_n^{\leqslant j}$ are asymptotically equivalent. Proposition~\ref{prop:asymptotic-equivalence-Xbar-Jbar} along with Proposition~\ref{prop:asymptotic-equivalence-Jbar-Jj} prove the desired asymptotic equivalence, and hence, conclude the proof of the Theorem~\ref{thm:one-sided-limit-theorem}.
\end{proof}

\begin{proposition}\label{prop:asymptotic-equivalence-Xbar-Jbar}
Let $\bar X_n$ and $\bar J_n$ be as in the proof of Theorem~\ref{thm:one-sided-limit-theorem}. Then,
$\bar X_n$ and $\bar J_n$ are asymptotically equivalent w.r.t.\ $\big(n\nu[n,\infty)\big)^{j}$ for any $j\geq 0$.
\end{proposition}
\begin{proof}
In view of the decomposition (\ref{eq:main-decomposition}), we are done if we show that $\P(\|\bar Y_n\| > \delta)  = o\left((n\nu[n,\infty))^{-j}\right)$ and $\P(\|\bar Z_n\| > \delta)  = o\left((n\nu[n,\infty))^{-j}\right)$.
For the tail probability of $\|\bar Y_n\|$,
\begin{align*}
\P\bigg[\sup_{t\in[0,1]}|\bar Y_{n}(t)| > \delta\bigg]
&\leq \P\bigg[\sup_{t\in[0,n]}\big|B(t)\big|>n\delta/2\bigg]
\\
&
+
\P\bigg[\sup_{t\in[0,n]}\left|\int_{|x|\leq 1} x[N((0,t] \times dx) - t\nu(dx)]\right| > n\delta/2\bigg].
\end{align*}
We have an explicit expression for the first term by the reflection principle, and in particular, it decays at a geometric rate w.r.t.\ $n$. 
For the second term, let $Y'(t)\triangleq \int_{|x|\leq 1} x[N((0,t] \times dx) - t\nu(dx)]$.
Using Etemadi's bound for L\'evy processes (see Lemma~\ref{lem:cont_etemadi}), we obtain
\begin{align*}
&\P\bigg[\sup_{t\in[0,n]}\left|\int_{|x|\leq 1} x[N([0,t] \times dx) - t\nu(dx)]\right| > n\delta/2\bigg]\\
&\leq 3 \sup_{t\in[0,n]}\P\bigg[\left|Y'(t)\right| > n\delta/6\bigg]\\
&\leq 3 \sup_{t\in[0,n]}\bigg\{\P\bigg[|Y'(\lfloor t\rfloor)| > n\delta/12\bigg] + \P\bigg[|Y'(t)-Y'(\lfloor t\rfloor) | > n\delta /12\bigg]\bigg\}\\
&\leq 3 \sup_{t\in[0,n]}\P\bigg[|Y'(\lfloor t\rfloor)| > n\delta/12\bigg] + 3\sup_{t\in[0,n]}\P\bigg[|Y'(t)-Y'(\lfloor t\rfloor) | > n\delta /12\bigg]\\
&= 3 \sup_{1 \leq k \leq n}\P\bigg[|Y'(k)| > n\delta/12\bigg] + 3\sup_{t\in[0,1]}\P\bigg[|Y'(t) | > n\delta /12\bigg]\\
&
\leq
3 \sup_{1 \leq k \leq n}\P\bigg[\bigg|\sum_{i=1}^k\{Y'(i)-Y'(i-1)\}\bigg| > n\delta/12\bigg] 
\\
&\hspace{180pt}
+ 3\P\bigg[\sup_{t\in[0,1]}|Y'(t) |^m > (n\delta /12)^m\bigg].
\end{align*}
Since $Y'(i)-Y'(i-1)$ are i.i.d.\ with $Y'(i)-Y'(i-1)\stackrel{\D}{=}Y'(1)= \int_{|x|\leq 1}\allowbreak x[N((0,1]\times dx) - \nu(dx)]$ and $Y'(1)$ has exponential moments, the first term decreases at a geometric rate w.r.t.\ $n$ due to the Chernoff bound;
on the other hand, since $Y'(t)$ is a martingale, 
the second term is bounded by $3\frac{\E |Y'(1)|^m}{n^m (\delta/12)^m}$ for any $m$ by Doob's submartingale maximal inequality.
Therefore, by choosing $m$ large enough, this term can be made negligible.
For the tail probability of $\|\bar Z_n\|$, note that $\bar Z_n$ is a mean zero L\'evy process with the same distribution as $\mu_1^+(N(ns)/n - \nu_1^+s)$, where $N$ is the Poisson process with rate $\nu_1^+$.
Therefore, again from the continuous-time version of Etemadi's bound, we see that
$\P(\|\bar Z_n\| > \delta)$
decays at a geometric rate w.r.t.\ $n$ for any $\delta>0$.
\end{proof}

\begin{proposition}\label{prop:asymptotic-equivalence-Jbar-Jj}
For each $j\geq 0$, let $\bar J_n$  and $\hat J_n^{\leqslant j}$ be defined as in the proof of Theorem~\ref{thm:one-sided-limit-theorem}. Then,
$\bar J_n$ and $\hat J_n^{\leqslant j}$ are asymptotically equivalent w.r.t.\ $\big(n\nu[n,\infty)\big)^{j}$.
\end{proposition}
\begin{proof}
With the convention that the summation is $0$ in case the superscript is strictly smaller than the subscript, consider the following decomposition of $\bar J_n$:
\begin{align*}
\hat J_n^{\leqslant j} &\triangleq \frac1n \sum_{l=1}^{j} Q_n^\gets(\Gamma_l) 1_{[U_l,1]},
&
\bar J_n^{>j} &\triangleq \frac1n \sum_{l=j+1}^{\tilde N_n} (Q_n^\gets(\Gamma_l) -\mu_1^+)1_{[U_l,1]},
\\
\check J_n^{\leqslant j}  &\triangleq \frac1n \sum_{l=1}^j -\mu_1^+1_{[U_l,1]},
&
\bar R_n &\triangleq \frac1n \I(\tilde N_n < j)\sum_{l=\tilde N_n+1}^j (Q_n^\gets(\Gamma_l) - \mu_1^+)1_{[U_l,1]},
\end{align*}
so that
$$
\bar J_n
= \hat J_n^{\leqslant j} + \check J_n^{\leqslant j} + \bar J_n^{>j} - \bar R_n.
$$
Note that $\P(\|\check J_n^{\leqslant j}\| \geq \delta) =0$ for sufficiently large $n$ since $\|\check J_n^{\leqslant j}\| = j\mu_1/n$.
On the other hand, $\P(\|\bar R_n\|\geq \delta)$ decays at a geometric rate since $\{\|\bar R_n\|\geq \delta\}  \subseteq \{\tilde N_n < j\}$ and $\P(\tilde N_n < j)$ decays at a geometric rate.
Since $\P(\|\bar J_n^{>j}\| \geq \delta) \leq \P(\|\bar J_n^{>j}\| \geq \delta, Q_n^\gets(\Gamma_j)\geq n\gamma ) + \P(\|\bar J_n^{>j}\| \geq \delta, Q_n^\gets(\Gamma_j)\leq n\gamma )$, Lemma~\ref{lem:key-upper-bound-1} and Lemma~\ref{lem:key-upper-bound-2} given below imply $\P(\|\bar J_n^{>j}\| \geq \delta)=o\left( (n\nu[n,\infty))^j\right)$ by choosing $\gamma$ small enough.
Therefore, $\hat J_n^{\leqslant j}$ and $\bar J_n$ are asymptotically equivalent w.r.t.\ $(n\nu[n,\infty))^j$.
\end{proof}

Define a measure $\mu_\alpha^{(j)}$ on $\mathbb R_+^{\infty \downarrow}$ by
\begin{equation*}
\begin{aligned}
\mu_\alpha^{(j)}(dx_1, dx_2, \cdots)
&\triangleq \prod_{i=1}^j \nu_\alpha(dx_i)\I_{[x_1\geq x_2\geq \cdots \geq x_j > 0]} \prod_{i=j+1}^\infty \delta_0(dx_i),
\end{aligned}
\end{equation*}
where $\nu_\alpha(x,\infty) = x^{-\alpha}$, and $\delta_0$ is the Dirac measure concentrated at $0$.

\begin{proposition}\label{prop:Jj}
For each $j\geq 0$,
$$
\big(n\nu[n,\infty)\big)^{-j}\P(\hat J_n^{\leqslant j}\in \cdot)
\to C_j(\cdot)
$$
in $\mathbb M\big(\mathbb D \setminus \mathbb D_{<\, j}\big)$ as $n\to\infty$.
\end{proposition}
\begin{proof}
Noting that $(\mu_\alpha^{(j)}\times \text{Leb})\circ T_j^{-1} = C_j$ and $\P(\hat J_n^{\leqslant j} \in \cdot) = \P\big(\allowbreak\big((Q_n^\gets(\Gamma_l)/n,l\geq 1),(U_l,l\geq 1)\big)\in T_j^{-1}(\cdot)\big)$, Lemma~\ref{lem:poisson-jumps} and  Corollary~\ref{result:consequence-of-L23-53-54} prove the proposition.
\end{proof}

\begin{lemma}\label{lem:key-upper-bound-1}
For any fixed $\gamma>0$, $\delta>0$,and $j\geq 0$,
\begin{equation}\label{eq:hatJ_on_geq}
\P \left\{\| \bar J_n^{>j}\|\geq \delta, Q_n^\gets(\Gamma_j) \geq n\gamma\right\} = o\left((n\nu[n,\infty))^j\right).
\end{equation}
\end{lemma}
\begin{proof} (Throughout the proof of this lemma, we use $\mu_1$ and $\nu_1$ in place of $\mu_1^+$ and $\nu_1^+$ respectively.) We start with the following decomposition of $\bar J_n^{>j}$: for any fixed $\lambda \in \left(0, \frac{\delta}{3\nu_1\mu_1}\right)$,
\begin{align*}
\bar J_n^{>j} &= \frac1n \sum_{l=j+1}^{\tilde N_n} (Q_n^\gets(\Gamma_l) - \mu_1)1_{[U_1,1]}
\\
&
= \tilde J_n^{[j+1,n\nu_1(1+\lambda)]} - \tilde J_n^{[\tilde N_n + 1, n\nu_1(1+\lambda)]}\I(\tilde N_n < n\nu_1(1+\lambda))
\\
&\hspace{80pt}
+ \tilde J_n^{[n\nu_1(1+\lambda)+1,\tilde N_n]}\I(\tilde N_n > n\nu_1(1+\lambda)),
\end{align*}
where
$$
\tilde J_n^{[a,b]}\triangleq \frac1n \sum_{l=\lceil a\rceil}^{\lfloor b\rfloor} (Q_n^\gets(\Gamma_l)-\mu_1)1_{[U_l,1]}.
$$
Therefore,
\begin{align*}
&\P \left\{\| \bar J_n^{>j} \|\geq \delta, Q_n^\gets(\Gamma_j) \geq n\gamma\right\}
\\
&
\leq
\P\left(\left\|\tilde J_n^{[j+1,n\nu_1(1+\lambda)]}\right\|\geq \delta/3, Q_n^\gets(\Gamma_j)\geq n\gamma\right)
\\
&\hspace{10pt}
+\P\left( \left\|\tilde J_n^{[\tilde N_n + 1, n\nu_1(1+\lambda)]}\right\| \geq \delta/3\right)
+\P\left(\tilde N_n > n\nu_1(1+\lambda)\right)
\\
&
= \text{(i)} + \text{(ii)} + \text{(iii)}.
\end{align*}
Noting that
$\left\|\tilde J_n^{[\tilde N_n + 1, n\nu_1(1+\lambda)]}\right\| \leq (\nu_1(1+\lambda)-\tilde N_n/n)\mu_1$ --- recall that $\tilde N_n$ is defined to be the number of $l$'s such that $Q_n^\gets(\Gamma_l) \geq 1$, and hence, $0\leq Q_n^\gets(\Gamma_l)<1$ for $l>\tilde N_n$ --- we see that (ii) is bounded by
$$\P((\nu_1(1+\lambda)-\tilde N_n/n)\mu_1\geq \delta/3)= \P\left(\frac{\tilde N_n}{n\nu_1}\leq 1+\lambda-\frac{\delta}{3\nu_1\mu_1}\right),$$
which decays at a geometric rate w.r.t.\ $n$ since $\tilde N_n$ is Poisson with rate $n\nu_1$.
For the same reason, (iii) decays at a geometric rate w.r.t.\ $n$.
We are done if we prove that (i) is $o\left((n\nu[n,\infty))^j\right)$.
Note that $Q_{n}^{\gets}(\Gamma_j) \geq n\gamma$ implies $Q_n(n\gamma)\geq \Gamma_j$, and hence,
\begin{align*}
&\sum_{l=j+1}^{(1+\lambda)n\nu_1} \big(Q_{n}^{\gets}(\Gamma_l - \Gamma_j + Q_n(n\gamma))-\mu_1\big)1_{[U_l,1]}
\\
&\leq \sum_{l=j+1}^{(1+\lambda)n\nu_1} \big(Q_{n}^{\gets}(\Gamma_l)-\mu_1\big)1_{[U_l,1]} \\
&\leq
\sum_{l=j+1}^{(1+\lambda)n\nu_1} \big(Q_{n}^{\gets}(\Gamma_l - \Gamma_j)-\mu_1\big)1_{[U_l,1]}.
\end{align*}
Therefore, if we define
\begin{align*}
A_n &\triangleq \left\{Q_n^\gets(\Gamma_j) \geq n\gamma\right\},\\
B_n' &\triangleq \left\{
\sup_{t\in[0,1]}\sum_{l=j+1}^{(1+\lambda)n\nu_1} \big(Q_{n}^{\gets}(\Gamma_l - \Gamma_j)-\mu_1\big)1_{[U_l,1]}(t)
\geq
n\delta
\right\},
\\
B_n'' &\triangleq \left\{\inf_{t\in[0,1]}\sum_{l=j+1}^{(1+\lambda)n\nu_1} \big(Q_{n}^{\gets}(\Gamma_l - \Gamma_j + Q_n(n\gamma))-\mu_1\big)1_{[U_l,1]}(t)
\leq
-n\delta\right\},
\end{align*}
then we have that
$$
\text{(i)} \leq \P(A_n \cap (B'_n \cup B''_n)) \leq \P(A_n \cap B'_n) + \P(A_n \cap B''_n) = \P(A_n)(\P(B'_n)+\P(B''_n))
$$
where the last equality is from the independence of $A_n$ and $B_n'$ as well as of $A_n$ and $B_n''$ (which is, in turn, due to the independence of $\Gamma_j$ and $\Gamma_l - \Gamma_j$). 
From Lemma~\ref{lem:Djk} (c) and Proposition~\ref{prop:Jj}, $\P(A_n) = \P(\hat J_n^{\leqslant j} \in (\mathbb D\setminus \mathbb D_{<\, j})^{-\gamma/2}) =  O\left((n\nu[n,\infty))^j\right)$, and hence, it suffices to show that the probabilities of the complements of $B_n'$ and $B_n''$ converge to 1---i.e.,
for any fixed $\gamma>0$,
\begin{equation}\label{eq:lem_cn1}
\P\left\{\sup_{t\in[0,1]}\sum_{l=j+1}^{(1+\lambda)n\nu_1} \big(Q_{n}^{\gets}(\Gamma_l - \Gamma_j)-\mu_1\big)1_{[U_l,1]}(t)
<
n\delta\right\}\to 1,
\end{equation}
and
\begin{equation}\label{eq:lem_cn2}
\P \left\{\inf_{t\in[0,1]}\sum_{l=j+1}^{(1+\lambda)n\nu_1} \big(Q_{n}^{\gets}(\Gamma_l - \Gamma_j + Q_n(n\gamma))-\mu_1\big)1_{[U_l,1]}(t)
>
-n\delta\right\}\to 1.
\end{equation}
Starting with (\ref{eq:lem_cn1})
\begin{align*}
&\P\left\{\sup_{t\in[0,1]}\sum_{l=j+1}^{(1+\lambda)n\nu_1} \big(Q_{n}^{\gets}(\Gamma_l - \Gamma_j)-\mu_1\big)1_{[U_l,1]}(t)
<
n\delta\right\}\\
&=\P\left\{\sup_{t\in[0,1]}\sum_{l=1}^{(1+\lambda)n\nu_1-j} \big(Q_{n}^{\gets}(\Gamma_l)-\mu_1\big)1_{[U_l,1]}(t)
<
n\delta\right\}\\
&\geq\P\left\{\sup_{t\in[0,1]}\sum_{l=1}^{(1+\lambda)n\nu_1-j} \big(Q_{n}^{\gets}(\Gamma_l)-\mu_1\big)1_{[U_l,1]}(t)
<
n\delta, \tilde N_n \leq (1+\lambda)n\nu_1 - j\right\}\\
&\geq\P\left\{\sup_{t\in[0,1]}\sum_{l=1}^{\tilde N_n} \big(Q_{n}^{\gets}(\Gamma_l)-\mu_1\big)1_{[U_l,1]}(t)
<
n\delta, \tilde N_n \leq (1+\lambda)n\nu_1 - j\right\}\\
&\geq\P\left\{\sup_{t\in[0,1]}\sum_{l=1}^{\tilde N_n} \big(Q_{n}^{\gets}(\Gamma_l)-\mu_1\big)1_{[U_l,1]}(t)
<
n\delta)\right\} - \P\left\{\tilde N_n > (1+\lambda)n\nu_1 - j\right\}.
\end{align*}
The second inequality is due to the definition of $Q_{n}^\gets$ and that $\mu_1\geq 1$ (and hence $Q_{n}^\gets(\Gamma_l) - \mu_1 \leq 0$ on $l\geq \tilde N_n$), while the last inequality comes from the generic inequality $\P(A\cap B) \geq \P(A) - \P(B^c)$.
The second probability converges to 0 since $\tilde N$ is Poisson with rate $n\nu_1$.
Moving on to the first probability in the last expression, observe that $\sum_{l=1}^{\tilde N_n} \big(Q_{n}^{\gets}(\Gamma_l)-\mu_1\big)1_{[U_l,1]}(\cdot)$ has the same distribution as the compound Poisson process $\sum_{i=1}^{J(n\cdot)} (D_i-\mu_1)$, where $J$ is a Poisson process with rate $\nu_1$ and $D_i$'s are i.i.d.\ random variables with the distribution $\nu$ conditioned (and normalized) on $[1,\infty)$, i.e., $\P\{D_i \geq s\} = 1\wedge\big(\nu[s,\infty)/\nu[1,\infty)\big)$. Using this, we obtain
\begin{align}
&\P\left\{\sup_{t\in[0,1]}\sum_{l=1}^{\tilde N_n} \big(Q_{n}^{\gets}(\Gamma_l)-\mu_1\big)1_{[U_l,1]}(t)
<
n\delta\right\} \nonumber
\\
&= \P\left\{\sup_{1\leq m \leq J(n)}\sum_{l=1}^{m} (D_l -\mu_1)
<
n\delta\right\} \label{eq:same_form_cp}
\\
&\geq
\P\left\{\sup_{1\leq m\leq 2n\nu_1} \sum_{l=1}^m (D_l-\mu_1) < n\delta , J(n) \leq 2n\nu_1\right\} \nonumber
\\
&\geq
\P\left\{\sup_{1\leq m\leq 2n\nu_1} \sum_{l=1}^m (D_l-\mu_1) < n\delta \right\} - \P\big\{J(n) > 2n\nu_1\big\} \nonumber
\end{align}
The second probability vanishes at a geometric rate w.r.t.\ $n$ because $J(n)$ is Poisson with rate $n\nu_1$. The first term can be investigated by the generalized Kolmogorov inequality, cf.\ \cite{Shneer2009} (given as Result~\ref{result:gen_kol} in Appendix A):
\begin{align*}
\P\left(\max_{1\leq m\leq 2n\nu_1} \sum_{l=1}^m(D_l - \mu_1) \geq n\delta/2\right)
&\leq C\frac{2n\nu_1V(n\delta/2)}{(n\delta/2)^2},
\end{align*}
where
$V(x) = \E [(D_l-\mu_1)^2; \mu_1 - x \leq D_l \leq \mu_1 + x]\leq \mu_1^2 + \E[ D_l^2; D_l \leq \mu_1 + x]$.
Note that
\begin{align*}
\E [D_l^2; D_l \leq \mu_1 + x]
&= \int_0^1 2s ds+\int_1^{\mu_1 + x} 2s \frac{\nu[s,\infty)}{\nu[1,\infty)}ds
\\&
=1 + \frac{2}{\nu_1} (\mu_1+x)^{2-\alpha}L(\mu_1+x),
\end{align*}
for some slowly varying $L$. Hence,
$$\P\left(\max_{1\leq m\leq 2n\nu_1} \sum_{l=1}^m(D_l - \mu_1) < n\delta\right) \geq 1-\P\left(\max_{1\leq m\leq 2n\nu_1} \sum_{l=1}^m(D_l - \mu_1) \geq n\delta/2\right)\to 1,$$
as $n\to\infty$.

Now, turning to (\ref{eq:lem_cn2}), let $\gamma_n \triangleq Q_n(n\gamma)$.
\begin{align*}
&\P \left\{\inf_{t\in[0,1]}\sum_{l=j+1}^{(1+\lambda)n\nu_1} \big(Q_{n}^{\gets}(\Gamma_l - \Gamma_j + Q_n(n\gamma))-\mu_1\big)1_{[U_l,1]}(t)
>
-n\delta\right\}
\\
&= \P \left\{\inf_{t\in[0,1]}\sum_{l=1}^{(1+\lambda)n\nu_1-j} \big(Q_{n}^{\gets}(\Gamma_l +\gamma_n)-\mu_1\big)1_{[U_l,1]}(t)
>
-n\delta\right\}
\\
&\geq \P \left\{\inf_{t\in[0,1]}\sum_{l=1}^{(1+\lambda)n\nu_1-j} \big(Q_{n}^{\gets}(\Gamma_l +\gamma_n)-\mu_1\big)1_{[U_l,1]}(t)
>
-n\delta, E_0 \geq \gamma_n\right\}
\\
&\geq \P \left\{\inf_{t\in[0,1]}\sum_{l=1}^{(1+\lambda)n\nu_1-j} \big(Q_{n}^{\gets}(\Gamma_l +E_0)-\mu_1\big)1_{[U_l,1]}(t)
>
-n\delta, E_0 \geq \gamma_n\right\}
\\
&= \P \left\{\inf_{t\in[0,1]}\sum_{l=2}^{(1+\lambda)n\nu_1-j+1} \big(Q_{n}^{\gets}(\Gamma_l)-\mu_1\big)1_{[U_l,1]}(t)
>
-n\delta, \Gamma_1 \geq \gamma_n\right\}
\\
&\geq \P \left\{\inf_{t\in[0,1]}\sum_{l=2}^{(1+\lambda)n\nu_1-j+1} \big(Q_{n}^{\gets}(\Gamma_l)-\mu_1\big)1_{[U_l,1]}(t)
>
-n\delta\right\} - \P\left\{\Gamma_1 < \gamma_n\right\}\\
&= (A)-(B),
\end{align*}
where $E_0$ is a standard exponential random variable. (Recall that $\Gamma_l \triangleq E_1 + E_2 + \cdots + E_l$, and hence $(\Gamma_l + E_0, U_l) \stackrel{\D}{=} (\Gamma_{l+1}, U_{l})\stackrel{\D}{=} (\Gamma_{l+1}, U_{l+1})$.)
Since $(B) = \P\left\{\Gamma_1 < \gamma_n\right\} \to 0$ (recall that $\gamma_n = n\nu[n\gamma,\infty)$ and $\nu$ is regularly varying with index $-\alpha<-1$), we focus on proving that the first term (A) converges to 1:
\begin{align*}
(A) &=\P \Bigg\{\inf_{t\in[0,1]}\sum_{l=2}^{(1+\lambda)n\nu_1-j+1} \big(Q_{n}^{\gets}(\Gamma_l)-\mu_1\big)1_{[U_l,1]}(t)
>
-n\delta\Bigg\}
\\
&\geq \P \Bigg\{\inf_{t\in[0,1]}\sum_{l=2}^{(1+\lambda)n\nu_1-j+1} \big(Q_{n}^{\gets}(\Gamma_l)-\mu_1\big)1_{[U_l,1]}(t) > -n\delta,
\\
&\hspace{180pt}
\tilde N_n \leq (1+\lambda)n\nu_1 -j+1
\Bigg\}
\\
&\geq \P \Bigg\{\inf_{t\in[0,1]}\sum_{l=1}^{\tilde N_n} \big(Q_{n}^{\gets}(\Gamma_l)-\mu_1\big)1_{[U_l,1]}(t) \geq -n\delta/3,
\\
&\hspace{80pt}
\inf_{t\in[0,1]}-\big(Q_{n}^{\gets}(\Gamma_1)-\mu_1\big)1_{[U_1,1]}(t)> -n\delta/3,  
\\
&\hspace{80pt}
\inf_{t\in[0,1]}\sum_{l=\tilde N_n+1}^{(1+\lambda)n\nu_1-j+1} \big(Q_{n}^{\gets}(\Gamma_l)-\mu_1\big)1_{[U_l,1]}(t) \geq -n\delta/3,
\\
&\hspace{80pt}
\tilde N_n \leq (1+\lambda)n\nu_1 -j +1
\Bigg\}
\\
&\geq \P \Bigg\{\inf_{t\in[0,1]}\sum_{l=1}^{\tilde N_n} \big(Q_{n}^{\gets}(\Gamma_l)-\mu_1\big)1_{[U_l,1]}(t) \geq -n\delta/3, \Bigg\} 
\\
&\hspace{50pt}
+\P\Big\{Q_{n}^{\gets}(\Gamma_1)-\mu_1< n\delta/3\Big\}
\\
&\hspace{50pt}
+\P\Bigg\{\inf_{t\in[0,1]}\sum_{l=\tilde N_n+1}^{(1+\lambda)n\nu_1-j+1} \big(Q_{n}^{\gets}(\Gamma_l )-\mu_1\big)1_{[U_l,1]}(t) \geq -n\delta/3 \Bigg\}
\\
&\hspace{50pt}
+\P\left\{\tilde N_n \leq (1+\lambda)n\nu_1 -j +1\right\}-3
\\
&=\text{(AI) + (AII) + (AIII) + (AIV)}-3.
\end{align*}
The third inequality comes from applying the generic inequality $\P(A\cap B) \geq \P(A) + \P(B) -1$ three times.
Since $\tilde N_n$ is Poisson with rate $n\nu_1$, 
\begin{align*}
\text{(AIV)} = \P\left\{\tilde N_n \leq (1+\lambda)n\nu_1 -j +1\right\}
&=
\P\left\{\frac{\tilde N_n}{n\nu_1 } \leq 1+\lambda-\frac{j -1}{n\nu_1 }\right\}
\to 1.
\end{align*}
For the first term (AI),
\begin{align*}
\text{(AI)} &=\P \Bigg\{\inf_{t\in[0,1]}\sum_{l=1}^{\tilde N_n} \big(Q_{n}^{\gets}(\Gamma_l )-\mu_1\big)1_{[U_l,1]}(t) \geq -n\delta/3 \Bigg\}
\\
&=
\P \Bigg\{\sup_{t\in[0,1]}\sum_{l=1}^{\tilde N_n} \big(\mu_1 - Q_{n}^{\gets}(\Gamma_l )\big)1_{[U_l,1]}(t) \leq n\delta/3 \Bigg\}
\\
&=
\P \Bigg\{\sup_{1\leq m\leq J(n)}\sum_{l=1}^{m} (\mu_1-D_l) \leq n\delta/3 \Bigg\},
\end{align*}
where $D_i$ is defined as before. Note that this is of exactly same form as (\ref{eq:same_form_cp}) except for the sign of $D_l$, and hence, we can proceed exactly the same way using the generalized Kolmogorov inequality to prove that this quantity converges to $1$ --- recall that the formula only involves the square of the increments, and hence, the change of the sign has no effect.
For the second term (AII),
\begin{align*}
\text{(AII)} \geq \P\big\{Q_n^\gets(\Gamma_1) \leq n\delta/3\big\}\geq \P\big\{\Gamma_1 > Q_{n}(n\delta/3)\big\} \to 1,
\end{align*}
since $Q_{n}(n\delta/3)\to 0$.
For the third term (AIII),
\begin{align*}
\text{(AIII)} &= \P\Bigg\{\inf_{t\in[0,1]}\sum_{l=\tilde N_n+1}^{(1+\lambda)n\nu_1-j+1} \big(Q_{n}^{\gets}(\Gamma_l)-\mu_1\big)1_{[U_l,1]}(t) \geq -n\delta/3 \Bigg\}
\\
&\geq \P\Bigg\{\inf_{t\in[0,1]}\sum_{l=\tilde N_n+1}^{(1+\lambda)n\nu_1-j+1} (1-\mu_1) 1_{[U_l,1]}(t) \geq -n\delta/3 \Bigg\}
\\
&\geq \P\Bigg\{\sum_{l=\tilde N_n+1}^{(1+\lambda)n\nu_1-j+1} (\mu_1-1) \leq n\delta/3 \Bigg\}
\\
&\geq
\P\Bigg\{(\mu_1-1)\big((1+\lambda)n\nu_1-j - \tilde N_n+1\big) \leq n\delta/3\Bigg\}
\\
&\geq
\P\Bigg\{1+\lambda-\frac{\delta}{3\nu_1(\mu_1-1)}\leq \frac{\tilde N_n}{n\nu_1}  +\frac{j-1}{n\nu_1}\Bigg\}
\\
&\to 1,
\end{align*}
since $\lambda < \frac{\delta}{3\nu_1(\mu_1-1)}$. This concludes the proof of the lemma.
\end{proof}

\begin{lemma} \label{lem:key-upper-bound-2} For any $j\geq0$, $\delta>0$, and $m<\infty$, there is $\gamma_0>0$ such that
\begin{align*}
\P\left\{\left\|\bar J_n^{>j}\right\| > \delta,
Q_{n}^\gets(\Gamma_j)\leq n\gamma_0
\right\}=o(n^{-m}).
\end{align*}
\end{lemma}
\begin{proof}
(Throughout the proof of this lemma, we use $\mu_1$ and $\nu_1$ in place of $\mu_1^+$ and $\nu_1^+$ respectively, for the sake of notational simplicity.) Note first that $Q_n^\gets(\Gamma_j)=\infty$ if $j=0$ and hence the claim of the lemma is trivial. Therefore, we assume $j\geq1$ throughout the rest of the proof.
Since for any $\lambda > 0$
\begin{align}
&\P\left\{\left\|\bar J_n^{>j} \right\| > \delta,
Q_{n}^\gets(\Gamma_j)\leq n\gamma
\right\}
\nonumber
\\
&
\leq
\P\Bigg\{\Bigg\|\sum_{l=j+1}^{\tilde N_n} (Q_{n}^\gets(\Gamma_l)-\mu_1) 1_{[U_l,1]} \Bigg\| > n\delta,
Q_{n}^\gets(\Gamma_j)\leq n\gamma, 
\label{eq:key_bound}
\\
\nonumber
&\hspace{190pt}
\frac{\tilde N_n}{n\nu_1} \in \left[\frac{j}{n\nu_1}, 1+\lambda\right]
\Bigg\}
\\
&
\hspace{10pt}+
\P\left\{\frac{\tilde N_n}{n\nu_1} \notin \left[\frac{j}{n\nu_1}, 1+\lambda\right]
\right\},
\nonumber
\end{align}
and
$\P\left\{\frac{\tilde N_n}{n\nu_1} \notin \left[\frac{j}{n\nu_1}, 1+\lambda\right]
\right\}$ decays at a geometric rate w.r.t.\ $n$, it suffices to show that (\ref{eq:key_bound})  is $o(n^{-m})$ for small enough $\gamma>0$.
First, recall that by the definition of $Q_n^{\gets}(\cdot)$,
$$ Q_n^{\gets}(x) \geq s \qquad \Longleftrightarrow \qquad x \leq Q_n(s),$$
and
$$  n\nu(Q_n^\gets(x), \infty) \leq x \leq n\nu[Q_n^\gets(x),\infty).$$
Let $L$ be a random variable conditionally (on $\tilde N_n$) independent of everything else and uniformly sampled on $\{j+1, j+2,\ldots, \tilde N_n\}$.
Recall that given $\tilde N_n$ and $\Gamma_j$, the distribution of $\{\Gamma_{j+1}, \Gamma_{j+2}, \ldots, \Gamma_{\tilde N_n}\}$ is same as that of the order statistics of $\tilde N_n - j$ uniform random variables on $[\Gamma_j, n\nu[1,\infty)]$.
Let $D_l, l\geq 1$, be i.i.d.\ random variables whose conditional distribution is the same as the conditional distribution of $Q_{n}^\gets  (\Gamma_L)$ given $\tilde N_n$ and $\Gamma_j$.
Then the conditional distribution of $\sum_{l=j+1}^{\tilde N_n} (Q_{n}(\Gamma_l)-\mu_1) 1_{[U_l,1]}$ is the same as that of $ \sum_{l=1}^{\tilde N_n-j} (D_l-\mu_1)1_{[U_l,1]} $.
Therefore, the conditional distribution of $\left\|\sum_{l=j+1}^{\tilde N_n} (Q_{n}(\Gamma_l)-\mu_1) 1_{[U_l,1]}\right\|_\infty$  is the same as the corresponding conditional distribution of
$\sup_{1\leq m \leq \tilde N_n-j}\allowbreak\Big|\sum_{l=1}^{m} (D_l-\mu_1)\Big|$.
To make use of this in the analysis what follows, we make a few observations on the conditional distribution of $Q_n^\gets(\Gamma_L)$ given $\Gamma_j$ and $\tilde N_n$.
\begin{itemize}
\item[(a)] \emph{The conditional distribution of $Q_n^{\gets}(\Gamma_L)$:}\\
Let $q \triangleq Q_n^\gets(\Gamma_j)$. Since $\Gamma_L$ is uniformly distributed on $[\Gamma_j, Q_n(1)] = [\Gamma_j, n\nu[1,\infty)]$, the tail probability is
\begin{align*}
\P\{Q_n^\gets(\Gamma_L) \geq s | \Gamma_j, \tilde N_n \}
&= \P\{ \Gamma_L \leq Q_n(s) | \Gamma_j, \tilde N_n \}
\\
&
= \P\{\Gamma_L \leq n\nu[s,\infty)|\Gamma_j, \tilde N_n \}\\
&= \P\left\{\left. \frac{\Gamma_L - \Gamma_j}{n\nu[1,\infty) - \Gamma_j} \leq\frac{n\nu[s,\infty)-\Gamma_j}{n\nu[1,\infty)- \Gamma_j} \right| \Gamma_j, \tilde N_n\right\}
\\
&= \frac{n\nu[s,\infty)-\Gamma_j}{n\nu[1,\infty)- \Gamma_j}
\end{align*}
for $s\in [1,q]$; since this is non-increasing w.r.t.\ $\Gamma_j$ and $n\nu(q,\infty) \leq \Gamma_j \leq n\nu[q,\infty)$, we have that
\begin{align*}
\frac{\nu[s,q)}{\nu[1,q)}
\leq \P\{Q_n^\gets(\Gamma_L) \geq s | \Gamma_j, \tilde N_n \}
\leq \frac{\nu[s,q]}{\nu[1,q]}.
\end{align*}
\item[(b)]\emph{Difference in mean between conditional and unconditional distribution:}\\
From (a), we obtain
\begin{align*}
\tilde\mu_n \triangleq \E[ Q_n^\gets(\Gamma_L)|\Gamma_j, \tilde N_n] \in \left[1+\int_1^q \frac{\nu[s,q)}{\nu[1,q)}ds, 1+\int_1^q \frac{\nu[s,q]}{\nu[1,q]}ds\right],
\end{align*}
and hence,
\begin{align*}
|\mu_1 - \tilde\mu_n|
&\leq
  \left|  \frac{\nu[1,q)\int_1^\infty \nu[s,\infty) ds - \nu[1,\infty)\int_1^q \nu[s,q)ds}{\nu[1,\infty)\nu[1,q)}\right|
  \\
  &\hspace{15pt}
  \vee
  \left|  \frac{\nu[1,q]\int_1^\infty \nu[s,\infty) ds - \nu[1,\infty)\int_1^q \nu[s,q]ds}{\nu[1,\infty)\nu[1,q]}\right|.
\end{align*}
Since
\begin{align*}
&\frac{\nu[1,q)\int_1^\infty \nu[s,\infty) ds - \nu[1,\infty)\int_1^q \nu[s,q)ds}{\nu[1,\infty)\nu[1,q)}
\\
&
= \frac{\nu[q,\infty)}{\nu[1,q)}(q-1)
+ \frac{\int_q^\infty \nu[s,\infty) ds}{\nu[1,\infty)}
- \frac{\nu[q,\infty)\int_1^q \nu[s,\infty)ds}{\nu[1,\infty)\nu[1,q)},
\end{align*}
and
\begin{align*}
&\frac{\nu[1,q)\int_1^\infty \nu[s,\infty) ds - \nu[1,\infty)\int_1^q \nu[s,q)ds}{\nu[1,\infty)\nu[1,q)}
\\
&\hspace{100pt}-\frac{\nu[1,q]\int_1^\infty \nu[s,\infty) ds - \nu[1,\infty)\int_1^q \nu[s,q]ds}{\nu[1,\infty)\nu[1,q]}\\
&=\frac{\nu\{q\}\left((q-1)\nu[1,\infty) + \int_q^\infty \nu[s,\infty) ds + \int_1^q \nu[s,\infty)ds\right)}{\nu[1,\infty)(\nu[1,q)+\nu\{q\})},
\end{align*}
we see that $|\mu_1-\tilde\mu_n|$ is bounded by a regularly varying function with index $1-\alpha$ (w.r.t.\ $q$) from Karamata's theorem.
\item[(c)] \emph{Variance of $ Q_n^{\gets}(\Gamma_L)$:}
Turning to the variance, we observe  that, if $\alpha \leq 2$,
\begin{equation}
\label{eq:conditional_variance}
\begin{aligned}
&
\E [Q_n^{\gets} (\Gamma_L) ^2 |\Gamma_j, \tilde N_n]
\\
&
\leq \int_0^1 2s ds + 2\int_1^q s\frac{\nu[s,q]}{\nu[1,q]}ds 
\\
&
\leq 1+\frac{2}{\nu[1,q]}\int_1^q s\nu[s,\infty)ds = 1+q^{2-\alpha} L(q)
\end{aligned}
\end{equation}
for some slowly varying function $L(\cdot)$. If $\alpha > 2$, the variance is bounded w.r.t.\ $q$.
\end{itemize}
Now, with (b) and (c) in  hand, we can proceed with an explicit bound since all the randomness is contained in $q$. Namely, we infer
\begin{align*}
&\P\Bigg(\Bigg\|\sum_{l=j+1}^{\tilde N_n} (Q_{n}^\gets(\Gamma_l)-\mu_1) 1_{[U_l,1]} \Bigg\|_\infty > n\delta,
Q_n^\gets(\Gamma_j) \leq n\gamma, \frac{\tilde N_n}{n\nu_1} \in \left[\frac{j}{n\nu_1}, 1+\lambda\right]
\Bigg)\\
&=\P\Bigg(\Bigg\|\sum_{l=j+1}^{\tilde N_n} (Q_{n}^\gets(\Gamma_l)-\mu_1) 1_{[U_l,1]} \Bigg\|_\infty > n\delta,
\Gamma_j \geq Q_{n}(n\gamma), \frac{\tilde N_n}{n\nu_1} \in \left[\frac{j}{n\nu_1}, 1+\lambda\right]
\Bigg)\\
&=\E\left[ \P\Bigg(\Bigg\|\sum_{l=j+1}^{\tilde N_n} (Q_{n}^\gets(\Gamma_l)-\mu_1) 1_{[U_l,1]} \Bigg\|_\infty > n\delta \right|\Gamma_j, \tilde N_n\Bigg)
; \Gamma_j \geq Q_{n}(n\gamma), 
\\
&\hspace{260pt}
\frac{\tilde N_n}{n\nu_1} \in \left[\frac{j}{n\nu_1}, 1+\lambda\right]
 \Bigg]
\\&
=\E\Bigg[ \P\Bigg(\left. \max_{1\leq m \leq \tilde N_n-j}\left|\sum_{l=1}^{m} (D_l-\mu_1)\right| > n\delta \right|\Gamma_j, \tilde N_n\Bigg)
; \Gamma_j \geq Q_{n}(n\gamma),
\\
&\hspace{260pt}
\frac{\tilde N_n}{n\nu_1} \in \left[\frac{j}{n\nu_1}, 1+\lambda\right]\Bigg].
\end{align*}
By Etemadi's bound (Result~\ref{result:etimedi} in Appendix),
\begin{equation}\label{eq:etemadi-in-action}
\begin{aligned}
&
\P\left(\left.\max_{1\leq m \leq \tilde N_n-j}\left|\sum_{l=1}^m(D_l - \mu_1)\right|\geq n\delta\right|\Gamma_j, \tilde N_n\right)
\\&
\leq 3\max_{1\leq m \leq \tilde N_n}\P\left(\left.\left|\sum_{l=1}^m(D_l - \mu_1)\right|\geq n\delta\right|\Gamma_j, \tilde N_n\right)
\\&
\leq 3\max_{1\leq m \leq \tilde N_n}\Bigg\{\P\left(\left.\sum_{l=1}^m(D_l - \mu_1)\geq n\delta\right|\Gamma_j,\tilde N_n\right)
\\
&\hspace{150pt}
+\P\left(\left.\sum_{l=1}^m(\mu_1 - D_l)\geq n\delta\right|\Gamma_j,\tilde N_n\right)\Bigg\}
\end{aligned}
\end{equation}
and as $|D_l-\tilde\mu_n|$ is bounded by $q$, we can apply Prokhorov's bound (Result~\ref{result:prokhorov} in Appendix) to get
\begin{align*}
&\P\left(\left.\sum_{l=1}^m(\mu_1-D_l)\geq n\delta\right|\Gamma_j,\tilde N_n\right)
\\
&
= \P\left(\left.\sum_{l=1}^m(\tilde\mu_n-D_l)\geq n\delta-m(\mu_1-\tilde\mu_n)\right|\Gamma_j,\tilde N_n\right)
\\&
\leq \P\left(\left.\sum_{l=1}^m(\tilde \mu_n-D_l)\geq n\delta-n\nu_1(1+\lambda)(\mu_1-\tilde\mu_n)\right|\Gamma_j,\tilde N_n\right)
\\&
\leq \left(\frac{qn(\delta - \nu_1(1+\lambda)(\mu_1-\tilde\mu_n))}{m\var(Q_n^\gets(\Gamma_L))}\right)^{-\frac{n(\delta - \nu_1(1+\lambda)(\mu_1-\tilde\mu_n))}{2q}}
\\&
\leq \left(\frac{n\nu_1(1+\lambda)\var(Q_n^\gets(\Gamma_L))}{qn(\delta - \nu_1(1+\lambda)(\mu_1-\tilde\mu_n))}\right)^{\frac{n(\delta - \nu_1(1+\lambda)(\mu_1-\tilde\mu_n))}{2q}}
\end{align*}
\begin{align*}
=\left\{
\begin{array}{ll}
\left(\frac{\nu_1(1+\lambda)(1+q^{2-\alpha}L_1(q))}{q(\delta - \nu_1(1+\lambda)q^{1-\alpha}L_2(q))}\right)^{\frac{n(\delta - \nu_1(1+\lambda)q^{1-\alpha}L_2(q))}{2q}} & \text{if } \alpha \leq 2,
\\
\left(\frac{\nu_1(1+\lambda)C}{q(\delta - \nu_1(1+\lambda)q^{1-\alpha}L_2(q))}\right)^{\frac{n(\delta - \nu_1(1+\lambda)q^{1-\alpha}L_2(q))}{2q}} & \text{otherwise,}
\end{array}
\right.
\end{align*}
for some $C>0$ if $m\leq (1+\lambda)n\nu_1$. Therefore, there exist constants $M$ and $c$ such that $q \geq M$ (i.e., $\Gamma_j \leq Q_n(M)$) implies
\begin{align*}
\P\left(\left.\sum_{l=1}^m(\mu_1 - D_l) \geq n\delta\right|\Gamma_j\right)
\leq c(q^{1-\alpha\wedge 2})^{\frac{n\delta}{8q}},
\end{align*}
and since we are conditioning on $q = Q_{n}^\gets(\Gamma_j)\leq n\gamma$,
\begin{align*}
c(q^{1-\alpha\wedge 2})^{\frac{n\delta}{8q}}
\leq c(q^{1-\alpha\wedge 2})^{\frac{\delta}{8\gamma}}.
\end{align*}
Hence,
\begin{align*}
\P\left(\left.\sum_{l=1}^m(\mu_1 - D_l) \geq n\delta\right|\Gamma_j\right)
\leq c\left(Q_n^\gets(\Gamma_j)^{1-\alpha\wedge 2}\right)^{\frac{\delta}{8\gamma}}.
\end{align*}
With the same argument, we also get
\begin{align*}
\P\left(\left.\sum_{l=1}^m( D_l-\mu_1) \geq n\delta\right|\Gamma_j\right)
\leq c\left(Q_n^\gets(\Gamma_j)^{1-\alpha\wedge 2}\right)^{\frac{\delta}{8\gamma}}.
\end{align*}
Combining (\ref{eq:etemadi-in-action}) with the two previous estimates, we obtain
\begin{align*}
\P\left(\left.\max_{1\leq m \leq \tilde N_n-j}\left|\sum_{l=1}^m(D_l - \mu_1)\right|\geq n\delta\right|\Gamma_j, \tilde N_n\right)
\leq 6 c\left(Q_n^\gets(\Gamma_j)^{1-\alpha\wedge 2}\right)^{\frac{\delta}{8\gamma}},
\end{align*}
on $\Gamma_j\geq Q_n(n\gamma)$, $\tilde N_n - j \leq n\nu_1(1+\lambda)$, and $\Gamma_j \leq Q_n(M)$.
Now,
\begin{align*}
&\E\bigg[ \P\Bigg(\left. \max_{1\leq m \leq \tilde N_n-j}\left|\sum_{l=1}^{m} (D_l-\mu_1)\right| > n\delta \right|\Gamma_j, \tilde N_n\Bigg)
; \Gamma_j \geq Q_{n}(n\gamma)\ 
\\
&\hspace{220pt}
\&\ \frac{\tilde N_n}{n\nu_1} \in \left[\frac{j}{n\nu_1}, 1+\lambda\right]\bigg]
\\&
\leq \E\bigg[ \P\Bigg(\left. \max_{1\leq m \leq \tilde N_n-j}\left|\sum_{l=1}^{m} (D_l-\mu_1)\right| > n\delta \right|\Gamma_j, \tilde N_n\Bigg)
; \Gamma_j \geq Q_{n}(n\gamma);\\
&
\hspace{180pt}  
\frac{\tilde N_n}{n\nu_1} \in \left[\frac{j}{n\nu_1}, 1+\lambda\right]; \Gamma_j\leq Q_n(M)\bigg]
\\&
\hspace{20pt} +  \P(\Gamma_j > Q_n(M))
\\&
\leq \E \left[6 c\left(Q_n^\gets(\Gamma_j)^{1-\alpha\wedge 2}\right)^{\frac{\delta}{8\gamma}}\right] + \P(\Gamma_j > Q_n(M))
\\&
\leq \E \left[6 c\left(Q_n^\gets(\Gamma_j)^{1-\alpha\wedge 2}\right)^{\frac{\delta}{8\gamma}};Q_n^\gets(\Gamma_j) \geq n^\beta \right] +
\P \left( Q_n^\gets(\Gamma_j) < n^\beta  \right)
\\
&\hspace{250pt}
+  \P\big(\Gamma_j > Q_n(M)\big)
\\&
\leq 6 c\left(n^{\beta(1-\alpha\wedge 2)}\right)^{\frac{\delta}{8\gamma}} +
\P \left( \Gamma_j > Q_n(n^\beta)  \right)
 + \P\big(\Gamma_j > Q_n(M)\big)
\\&
\leq 6 c\left(n^{\beta(1-\alpha\wedge 2)}\right)^{\frac{\delta}{8\gamma}} +
\P \left( \Gamma_j > (n^{1-\alpha\beta}L(n))  \right)
 + \P\big(\Gamma_j > Q_n(M)\big),
\end{align*}
for any $\beta>0$. If one chooses $\beta$ so that $1-\alpha\beta>0$  (for example, $\beta = \frac1{2\alpha}$), the second and third terms vanish at a geometric rate w.r.t.\ $n$. On the other hand, we can pick $\gamma$ small enough compared to $\delta$, so that the first term is decreasing at an arbitrarily fast polynomial rate. This concludes the proof of the lemma.
\end{proof}

Recall that we denote the Lebesgue measure on $[0,1]^\infty$ with $\text{Leb}$ and
defined measures $\mu_\alpha^{(j)}$ and $\mu_\beta^{(j)}$ on $\mathbb R_+^{\infty \downarrow}$ as
\begin{equation*}
\mu_\alpha^{(j)}(dx_1, dx_2, \ldots)
\triangleq \prod_{i=1}^j \nu_\alpha(dx_i)\I_{[x_1\geq x_2\geq \cdots \geq x_j > 0]} \prod_{i=j+1}^\infty \delta_0(dx_i),
\end{equation*}
and $\nu_\alpha(x,\infty) = x^{-\alpha}$, where $\delta_0$ is the Dirac measure concentrated at $0$.

\begin{lemma}\label{lem:poisson-jumps}
For each $j\geq 0$,
$$\big(n\nu[n,\infty))^{-j}\P[((Q_{n}^{\gets}(\Gamma_l)/n, l \geq 1), (U_l, l\geq 1)) \in \cdot] \to (\mu_\alpha^{(j)}\times \text{Leb})(\cdot)$$
in $\mathbb M\big((\mathbb R_+^{\infty\downarrow}\times[0,1]^\infty) \setminus (\mathbb H_{<\, j}\times [0,1]^\infty)\big)$ as $n\to\infty$.
\end{lemma}
\begin{proof} We first prove that
\begin{equation}\label{eq:claim-poisson-jumps}
\big(n\nu[n,\infty))^{-j}\P[(Q_{n}^{\gets}(\Gamma_l)/n, l \geq 1) \in \cdot] \to \mu_\alpha^{(j)}(\cdot)
\end{equation}
in $\mathbb M(\mathbb R_+^{\infty\downarrow}\setminus \mathbb H_{<\, j})$ as $n\to\infty$. 
To show this, we only need to check that
\begin{equation}\label{eq:what-to-check}
\big(n\nu[n,\infty))^{-j} \allowbreak \P[(Q_{n}^{\gets}(\Gamma_l)/n, l \geq 1) \in A] \to \mu_\alpha^{(j)}(A)
\end{equation}
for $A$'s that belong to the convergence-determining class $\mathcal A_{j} \triangleq \big\{\{z\in \R_+^{\infty\downarrow}: x_1 \leq z_1 , \ldots, x_l \leq z_l\}: l\geq j, x_1\geq\ldots\geq x_l>0\big\}$.
To see that $\mathcal A_{j}$ is a convergence-determining class for $\mathbb M(\mathbb R_+^{\infty\downarrow}\setminus \mathbb H_{<\, j})$-convergence, note that $\mathcal A'_{j} \triangleq \big\{\{z\in \R_+^{\infty\downarrow}: x_1 \leq z_1 < y_1, \ldots, x_l \leq z_l <  y_l\}: l\geq j,\ x_1,\ldots,x_l\in (0,\infty),\ y_1,\ldots,y_l \in (0,\infty] \big\}$  satisfies conditions (i), (ii), and (iii) of Lemma~\ref{lem:convergence-determining-class}, and hence, is a convergence-determining class.
Now define $\mathcal A_j(i)$'s recursively as $\mathcal A_{j}(i+1) \triangleq \{B\setminus A: A,B\in \mathcal A_j(i), A\subseteq B\}$ for $i\geq 0$, and $\mathcal A_{j}(0) = \mathcal A_j'' \triangleq \big\{\{z\in \R_+^{\infty\downarrow}: x_1 \leq z_1 , \ldots, x_l \leq z_l\}: l\geq j, x_1,\ldots, x_l>0\big\}$.
Since we restrict the set-difference operation between nested sets, the limit associated with the sets in $\mathcal A_{j}(i+1)$ is determined by the sets in $\mathcal A_j(i)$, and eventually, $\mathcal A_j''$. Noting that $\mathcal A'_j \subseteq \bigcup_{i=0}^\infty \mathcal A_j(i)$, we see that $\mathcal A_j''$ is a convergence-determining class.
Now, since both $\P[(Q_n^\gets(\Gamma_l)/n, l\geq 1)\in \cdot]$ and $\mu_\alpha^{(j)}(\cdot)$ are supported on $\R^{\infty\downarrow}_+$, one can further reduce the convergence determining class from $\mathcal A_j''$ to $\mathcal A_j$.

To check the desired convergence for the sets in $\mathcal A_j$, we first characterize the limit measure. 
Let $l \geq j$ and $x_1 \geq \cdots \geq x_l > 0$.
By the change of variables  $v_i = x_i^\alpha y_i^{-\alpha}$ for $i=1,\ldots, j$,
\begin{align*}
&
\mu_\alpha^{(j)} (\{z \in \mathbb R_+^{\infty\downarrow}: x_1 \leq z_1, \ldots, x_l \leq z_l\})
\nonumber\\
&
=\I(j=l)\cdot
\int_{x_j}^\infty \cdots\int_{x_1}^\infty \I(y_1\geq \cdots \geq y_j) d\nu_\alpha(y_1)\cdots d\nu_\alpha(y_j)
\nonumber
\\
&
=
\I(j=l)\cdot
\left(\prod_{i=1}^j x_i\right)^{-\alpha}\cdot\int_0^1\cdots \int_0^1 \I(x_1^{-\alpha}v_1 \leq \cdots \leq x_j^{-\alpha} v_j) dv_1\cdots dv_j.
\end{align*}
Next, we find a similar representation for the distribution of $\Gamma_1,\ldots, \Gamma_l$. 
Let $U_{(1)},\ldots, U_{(l-1)}$ be the order statistics of $l-1$ iid uniform random variables on $[0,1]$.
Recall first that the conditional distribution of $(\Gamma_1/\Gamma_l,\ldots, \Gamma_{l-1}/\Gamma_l)$ given $\Gamma_j$ does not depend on $\Gamma_j$ and coincides with the distribution of $(U_{(1)},\ldots, U_{(l-1)})$; see, for example, \cite{pyke1965spacings}.
Suppose that $l\geq j$ and $0\leq y_1\leq \cdots \leq y_l$. 
By the change of variables $u_i = \gamma^{-1} y_i v_i$ for $i=1,\ldots,l-1$, and $\gamma = y_{l}v_l$,
\begin{align*}
&
\P\big(\Gamma_1 \leq y_1,\ldots, \Gamma_l \leq y_l\big)
\\
&
= 
\E \Big[\P\big({\Gamma_1}/{\Gamma_l} \leq {y_1}/{\Gamma_l},\,\ldots,\, {\Gamma_{l-1}}/{\Gamma_l} \leq {y_{l-1}}/{\Gamma_l}\big| \Gamma_l\big)\cdot\I\big(\Gamma_l \leq y_l\big)\Big]
\\
&
=
\int_0^{y_l}\P\big(U_{(1)} \leq y_1/\gamma,\ldots,\, U_{(l-1)} \leq y_{l-1}/\gamma\big) \frac{e^{-\gamma}\gamma^{l-1}}{(l-1)!}d\gamma
\\
&
=
\int_0^{y_l} e^{-\gamma}\gamma^{l-1}\int_0^{y_{l-1}/\gamma}\cdots\int_0^{y_1/\gamma}\I(u_1\leq\cdots\leq u_{l-1} \leq 1) du_1\cdots du_{l-1}d\gamma
\\
&
=
\left(\prod_{i=1}^{l-1} y_i\right)\int_0^{y_l} e^{-\gamma}\int_0^1\cdots\int_0^1\I(y_1v_1\leq\cdots\leq y_{l-1} v_{l-1} \leq \gamma) dv_1\cdots dv_{l-1}d\gamma
\\
&
=
\left(\prod_{i=1}^{l} y_i\right)\cdot
\int_0^1\cdots\int_0^1 e^{-y_lv_l}\I(y_1v_1\leq \cdots \leq y_{l}v_{l}) dv_1\cdots dv_l.
\end{align*}
Since $0\leq Q_n(nx_1)\leq \ldots \leq Q_n(nx_l)$ for $x_1\geq \cdots \geq x_l > 0$,
\begin{align*}
&
(n\nu[n,\infty))^{-j}\P[Q_n^\gets(\Gamma_1)/n \geq x_1, \ldots, Q_n^\gets(\Gamma_l) \geq x_l]
\\
&
=
(n\nu[n,\infty))^{-j}\P[\Gamma_1\leq Q_n(n x_1), \ldots, \Gamma_l\leq Q_n(n x_l)]
\\
&
=
(n\nu[n,\infty))^{-j}\cdot\left(\prod_{i=1}^{l} Q_n(nx_i)\right)
\\
&\hspace{50pt}
\cdot\int_0^1\cdots\int_0^1 e^{-Q_n(nx_l)v_l}\I(Q_n(nx_1)v_1\leq \cdots \leq Q_n(nx_l)v_{l}) dv_1\cdots dv_l
\\
&
=
\Bigg(\prod_{i=1}^{j} \frac{Q_n(nx_i)}{n\nu[n,\infty)}\Bigg)\cdot \Bigg(\prod_{i=j+1}^{l} Q_n(nx_i)\Bigg)
\\
&\hspace{30pt}
\cdot\int_0^1\cdots\int_0^1 e^{-Q_n(nx_l)v_l}\I\bigg(\frac{Q_n(nx_i)}{n\nu[n,\infty)}v_1\leq \cdots \leq \frac{Q_n(nx_i)}{n\nu[n,\infty)}v_{l}\bigg) dv_1\cdots dv_l.
\end{align*}
Note that $Q_n(nx_i)\to 0$ and  $\frac{Q_n(nx_i)}{n\nu[n,\infty)}\to x_i^{-\alpha}$ as $n\to\infty$ for each $i=1,\ldots, l$. 
Therefore, by bounded convergence,
\begin{align*}
&
(n\nu[n,\infty))^{-j}\P[Q_n^\gets(\Gamma_1)/n \geq x_1, \ldots, Q_n^\gets(\Gamma_l) \geq x_l]
\\
&\to 
\I(j=l)\Bigg(\prod_{i=1}^{j}x_i\Bigg)^{-\alpha}
\cdot\int_0^1\cdots\int_0^1 \I(x_1^{-\alpha}v_1\leq \cdots \leq x_j^{-\alpha}v_{j}) dv_1\cdots dv_j
\\
&
= 
\mu_\alpha^{(j)} (\{z \in \mathbb R_+^{\infty\downarrow}: x_1 \leq z_1, \ldots, x_l \leq z_l\}),
\end{align*}
which concludes the proof of \eqref{eq:claim-poisson-jumps}.
The conclusion of the lemma follows from the independence of $(Q_n^\gets(\Gamma_l)/n, l\geq 1)$ and $(U_l,l\geq 1)$ and Lemma~\ref{thm:simple-product-space}.

\end{proof}

\begin{lemma}\label{lem:Djk}
Suppose that $x_1 \geq \cdots \geq x_j \geq 0$; $u_i\in(0,1)$ for $i=1,\ldots,j$;  $y_1 \geq \cdots \geq y_k \geq 0$; $v_i \in (0,1)$ for $i=1, \ldots,k$; $u_1,\ldots,u_j, v_1, \ldots, v_k$ are all distinct.
\begin{itemize}
	\item[(a)] For any $\epsilon>0$,
	\begin{align*}
	&
	\{x \in G: d(x,y)< (1+\epsilon)\delta \text{ implies } y\in G\}
	\\
	&
	\subseteq G^{-\delta}
	\\
	&
	\subseteq\{x \in G: d(x,y)< \delta \text{ implies } y\in G\}.
	\end{align*}
	Also, $(A\cap B)_\delta \subseteq A_\delta \cap B_\delta$ and $A^{-\delta} \cup B^{-\delta} \subseteq (A\cup B)^{-\delta}$ for any $A$ and $B$.
	
	\item[(b)] $\sum_{i=1}^j x_i 1_{[u_i,1]} \in (\mathbb D \setminus \mathbb D_{<\,j})^{-\delta}$ implies $x_j \geq \delta$.
	
	\item[(c)] $\sum_{i=1}^j x_i 1_{[u_i,1]} \notin (\mathbb D \setminus \mathbb D_{<\, j})^{-\delta}$ implies $x_j \leq 2\delta$.
	
	\item[(d)] $\sum_{i=1}^j x_i 1_{[u_i,1]} - \sum_{i=1}^k y_i 1_{[v_i,1]} \in (\mathbb D \setminus \mathbb D_{< j,k})^{-\delta}$ implies $x_j \geq \delta$ and $y_k \geq \delta$.

	\item[(e)] Suppose that $\xi \in \mathbb D_{j,k}$. If $l<j$ or $m<k$, then $\xi$ is bounded away from $\mathbb D_{l,m}$. 

	\item[(f)] If $I(\xi) > (\alpha -1)j + (\beta-1)k$, then $\xi$ is bounded away from $\mathbb D_{<j,k}\cup \mathbb D_{j,k}$.
\end{itemize}
\end{lemma}
\begin{proof} 
(a) Immediate consequences of the definition.

(b) From (a), we see that $\sum_{i=1}^j x_i 1_{[u_i,1]} \in (\mathbb D \setminus \mathbb D_{<\, j})^{-\delta}$ and $\sum_{i=1}^{j-1} x_i 1_{[u_i,1]}\in \mathbb D_{<\, j}$ implies $d\Big(\sum_{i=1}^j x_i 1_{[u_i,1]},\allowbreak\sum_{i=1}^{j-1} x_i 1_{[u_i,1]}\Big) \geq \delta$, which is not possible if  $x_j < \delta$.

(c) We prove that for any $\epsilon>0$, $\sum_{i=1}^j x_i 1_{[u_i,1]} \notin (\mathbb D \setminus \mathbb D_{<\,j})^{-\delta}$ implies $x_j \leq (2+\epsilon)\delta$. 
To show this, in turn, we work with the contrapositive. 
Suppose that $x_j> (2+\epsilon)\delta$. 
If $d(\sum_{i=1}^j x_i 1_{[u_i,1]}, \zeta)< (1+\epsilon/2)\delta$, by the definition of the Skorokhod metric, there exists a non-decreasing homeomorphism $\phi$ of $[0,1]$ onto itself such that $\|\sum_{i=1}^j x_i 1_{[u_i,1]}-\zeta\circ \phi\|_\infty < (1+\epsilon/2)\delta$.
Note that at each discontinuity point of $\sum_{i=1}^{j}x_i1_{[y_i,1]}$, $\zeta\circ\phi$ should also be discontinuous. 
Otherwise, the supremum distance between $\sum_{i=1}^j x_i 1_{[u_i,1]}$ and $\zeta\circ \phi$ has to be greater than $(1+\epsilon/2)\delta$, since the smallest jump size of $\sum_{i=1}^j x_i 1_{[u_i,1]}$ is greater than $(2+\epsilon)\delta$.  Hence, there has to be at least $j$ discontinuities in the path of $\zeta$; i.e., $\zeta \in \mathbb D \setminus \mathbb D_{<\,j}$. We have shown that $d(\sum_{i=1}^j x_i 1_{[u_i,1]}, \zeta)< (1+\epsilon/2)\delta$ implies $\zeta\in \mathbb D \setminus \mathbb D_{<\,j}$, which in turn, along with (a), shows that $\sum_{i=1}^j x_i 1_{[u_i,1]} \in (\mathbb D \setminus \mathbb D_{<\, j})^{-\delta}$.

(d) Suppose that $\sum_{i=1}^j x_i 1_{[u_i,1]} - \sum_{i=1}^k y_i 1_{[v_i,1]}\in (\mathbb D\setminus \mathbb D_{<j,k})^{-\delta}$. Since $\sum_{i=1}^{j-1} x_i 1_{[u_i,1]} - \sum_{i=1}^k y_i 1_{[v_i,1]}\notin \mathbb D \setminus \mathbb D_{<j,k}$,
$$x_j \geq d\left(\sum_{i=1}^j x_i 1_{[u_i,1]} - \sum_{i=1}^k y_i 1_{[v_i,1]}, \ \sum_{i=1}^{j-1} x_i 1_{[u_i,1]} - \sum_{i=1}^k y_i 1_{[v_i,1]}\right)
\geq \delta.$$
Similarly, we get $y_k \geq \delta$.

(e) Let $\xi = \sum_{i=1}^j x_i 1_{[u_i,1]} - \sum_{i=1}^k y_i 1_{[v_i,1]}$. 
First, we claim that $d(\zeta, \xi) \geq x_j/2$ for any $\zeta\in \mathbb D_{l,m}$ with $l < j$.  
Suppose not, i.e., $d(\zeta, \xi) < x_j / 2$.
Then there exists a  non-decreasing homeomorphism $\phi$ of $[0,1]$ onto itself such that $\|\sum_{i=1}^j x_i 1_{[u_i,1]}-\zeta\circ \phi\|_\infty < x_j/2$.
Note that this implies that at each discontinuity point $s$ of $\xi$, $\zeta\circ\phi$ should also be discontinuous. 
Otherwise, $|\zeta\circ\phi(s)- \xi(s)|+|\zeta\circ\phi(s-) - \xi(s-)|\geq|\xi(s)-\xi(s-)| \geq x_j$, and hence it is contradictory to the bound on the supremum distance between $\xi$ and $\zeta\circ \phi$.
However, this implies that $\zeta$ has $j$ upward jumps and hence, contradictory to the assumption $\zeta \in \mathbb D_{l,m}$, proving the claim.
Likewise, $d(\zeta, \xi) \geq y_k/2$ for any $\xi \in \mathbb D_{l,m}$ with $m < k$.

%

(f) Note that in case $I(\xi)$ is finite, $\mathcal D_+(\xi) >  j$ or $\mathcal D_-(\xi) > k$.
In this case, the conclusion is immediate from (e).
In case $I(\xi)=\infty$, either $\mathcal D_+(\zeta) =\infty$, $\mathcal D_-(\zeta)=\infty$, $\xi(0) \neq 0$, or $\xi$  contains a continuous non-constant piece. By containing a continuous non-constant piece, we refer to the case that there exist $t_1$ and $t_2$ such that $t_1 < t_2$, $\xi(t_1)\neq \xi(t_2-)$ and $\xi$ is continuous on $(t_1, t_2)$.
For the first two cases where the number of jumps is infinite, the conclusion is an immediate consequence of (e). The case $\xi(0)\neq 0$ is also obvious. Now we are left with dealing with the last case, where $\xi$ has a continuous non-constant piece.
To discuss this case, assume w.l.o.g.\ that $\xi(t_1) < \xi(t_2-)$.
We claim that $d(\xi,\mathbb D_{j,k}) \geq \frac{\xi(t_2-)-\xi(t_1)}{2(j+1)}$.
Note that for any step function  $\zeta$,
\begin{align*}
\|\xi-\zeta\|
&
\geq |\xi(t_2-)-\zeta(t_2-)|\vee |\xi(t_1)-\zeta(t_1)|
\\
&
\geq (\xi(t_2-)-\zeta(t_2-)) \vee (\zeta(t_1)-\xi(t_1))
\\
&
\geq \frac{1}{2}\Big\{ (\xi(t_2-) - \xi(t_1))-(\zeta(t_2-)-\zeta(t_1))\Big\}
\\
&
\geq \frac12 \Big\{(\xi(t_2-) - \xi(t_1)) - \sum_{t\in(t_1,t_2)} \big(\zeta(t)-\zeta(t-)\big) \Big\}
\\
&
\geq \frac12 \Big\{(\xi(t_2-)-\xi(t_1)) - 2\mathcal D_+(\zeta) \|\xi-\zeta\|\Big\},
\end{align*}
where the fourth inequality is due to the fact that $\|\xi-\zeta\| \geq \frac{\zeta(t)-\zeta(t-)}{2}$ for all $t\in(t_1,t_2)$. From this, we get
$$
\|\xi-\zeta\| \geq \frac{\xi(t_2-)-\xi(t_1)}{2(\mathcal D_+(\zeta) + 1)} \geq \frac{\xi(t_2-)-\xi(t_1)}{2(j + 1)},
$$
for $\zeta \in \mathbb D_{j,k}$. Now, suppose that $\zeta \in \mathbb D_{j,k}$. Since $\zeta\circ\phi$ is again in $\mathbb D_{j,k}$ for any non-decreasing homeomorphism $\phi$ of $[0,1]$ onto itself,
$$
d(\xi, \zeta) \geq \frac{\xi(t_2-)-\xi(t_1)}{2(j + 1)},
$$
which proves the claim.
\end{proof}


Now we move on to the proof of Theorem~\ref{thm:two-sided-limit-theorem}.
We first establish Theorem~\ref{thm:multi-d-limit-theorem}, which plays a key role in the proof. 
Recall that
$\mathbb D_{<j} = \bigcup_{0\leq l<j} \mathbb D_l$ and let
$\mathbb D_{<(j_1,\ldots,j_d)}
\triangleq \bigcup_{(l_1,\ldots,l_d) \in \I_{<(j_1,\ldots,j_d)}}\prod_{i=1}^d\mathbb D_{l_i}$
where
$
\I_{<(j_1,\ldots,j_d)}
\triangleq \big\{(l_1,\ldots,l_d)\allowbreak\in \mathbb Z_+^d\setminus\{(j_1,\ldots,j_d)\}: (\alpha_1-1)l_1+\cdots+(\alpha_d-1)l_d \leq (\alpha_1-1)j_1+\cdots+(\alpha_d-1)j_d\big\}$.
For each $l\in \mathbb Z_+$ and $i=1,\ldots,d$, let 
$
C_l^{(i)}(\cdot) \triangleq \E \Big[\nu_{\alpha_i}^l \{x\in (0,\infty)^l:\sum_{j=1}^l x_j 1_{[U_j,1]}\in \cdot\}\Big]
$
where $U_1,\ldots,U_l$ are iid uniform on $[0,1]$,
and 
$
\nu_{\alpha_i}^l
$ 
is as defined right below \eqref{math-display-above-definition-nu-alpha-j}.

\begin{theorem}\label{thm:multi-d-limit-theorem}
Consider independent 1-dimensional L\'evy processes $X^{(1)},\allowbreak\ldots, X^{(d)}$ with spectrally positive L\'evy measures $\nu_1(\cdot), \ldots,\nu_d(\cdot)$, respectively.
Suppose that each $\nu_i$ is regularly varying (at infinity) with index $-\alpha_i<-1$, and 
let $\bar X^{(i)}_n$ be centered and scaled scaled version of $X^{(i)}$ for each $i=1,\ldots,d$.
Then, for each $(j_1,\ldots,j_d)\in \mathbb Z_+^d$,
$$
\frac{\P((\bar X_n^{(1)},\ldots,\bar X_n^{(d)})\in \cdot)}{\prod_{i=1}^d\big(n\nu_i[n,\infty)\big)^{j_i}}\to C_{j_1}^{(1)}\times\cdots\times C_{j_d}^{(d)}(\cdot)
$$
in $\mathbb M\Big(\prod_{i=1}^d\mathbb D\setminus \mathbb D_{<(j_1,\ldots,j_d)}\Big)$.
\end{theorem}
\begin{proof}
From Theorem~\ref{thm:one-sided-limit-theorem}, we know that
$
(n\nu_i[n,\infty))^{-j}\P(\bar X^{(i)}_n\in \cdot)\to C_{j}(\cdot)
$
in $\mathbb M(\mathbb D\setminus \mathbb D_{<j})$ for $i=1,\ldots,d$ and any $j\geq 0$. 
This along with Lemma~\ref{thm:simple-product-space}, for each $(l_1,\ldots,l_d)\in \mathbb Z_+^d$ we obtain 
$$
\prod_{i=1}^d\big(n\nu_i[n,\infty)\big)^{-l_i}\P((\bar X^{(1)}_n,\ldots,\bar X^{(d)}_n)\in \cdot)\to C_{l_1}^{(1)}\times \cdots \times C_{l_d}^{(d)}(\cdot)
$$
in 
$
\mathbb M\Big( \prod_{i=1}^d\mathbb D \setminus \mathbb C_{(l_1,\ldots,l_d)}\Big)
$
where
$
\mathbb C_{(l_1,\ldots,l_d)}
\triangleq \bigcup_{i=1}^d  (\mathbb D^{i-1} \times \mathbb D_{<l_i} \times \mathbb D^{d-i})
$. 
Since $
\mathbb D_{<(j_1,\ldots,j_d)} 
= \bigcap_{(l_1,\ldots,l_d)\notin \I_{<(j_1,\ldots,j_d)}}\allowbreak\mathbb C_{(l_1,\ldots,l_d)},
$
our strategy is to proceed with Lemma~\ref{thm:union-limsup} to obtain the desired $\mathbb M\Big(\prod_{i=1}^d\mathbb D\setminus \mathbb D_{<(j_1,\ldots,j_d)}\Big)$-convergence by combining the $\mathbb M\Big( \prod_{i=1}^d\mathbb D \setminus \mathbb C_{(l_1,\ldots,l_d)}\Big)$-convergences for $(l_1,\ldots,l_d)\notin \I_{<(j_1,\ldots,j_d)}$. 
%
%
We first rewrite the infinite intersection over $\mathbb Z_+^d \setminus \I_{<(j_1,\ldots,j_d)}$ as a finite one to facilitate the application of the lemma.
Consider a partial order $\prec$ on $\mathbb Z_+^{d}$ such that $(l_1,\ldots,l_d) \prec (m_1,\ldots,m_d)$ if and only if $\mathbb C_{(l_1,\ldots,l_d)} \subsetneq \mathbb C_{(m_1,\ldots,m_d)}$. 
Note that this is equivalent to $l_i\leq m_i$ for $i=1,\ldots,d$ and $l_i < m_i$ for at least one $i=1,\ldots,d$. Let $\J_{j_1,\ldots,j_d}$ be the subset of $\mathbb Z_+^{d}$ consisting of the minimal elements of $\mathbb Z_+^{d}\setminus \I_{<(j_1,\ldots,j_d)}$, i.e., 
$
\J_{j_1,\ldots,j_d} \triangleq \{(l_1,\ldots,l_d)\in \mathbb Z_+^{d}\setminus \I_{<(j_1,\ldots,j_d)}: (m_1,\ldots,m_d) \prec (l_1,\ldots,l_d) \text{ implies }  (m_1,\ldots,m_d) \in \I_{<(j_1,\ldots,j_d)}\}
$.
Figure~\ref{fig1} illustrates how the sets $\I_{<(j_1,\ldots,j_d)}$ and $\J_{j_1,\ldots,j_d}$ look when $d=2$, $j_1=2$, $j_2=2$, $\alpha_1 = 2$, $\alpha_2 = 3$.
\begin{figure}\label{fig1}
\centering
\begin{tikzpicture}

\tikzstyle{axes}=[]
\tikzstyle{important line}=[very thick]
        
\draw[style=help lines, step=0.5cm] (-0.1,-0.1) grid (4.1,2.6);
        
\begin{scope}[style=axes]
\draw[->] (-0.4,0) -- (4.4,0) node[right]  {$l_1$} coordinate(x axis);
\draw[->] (0.01,-0.4) -- (0.01,2.9)  node[above]  {$l_2$} coordinate(y axis);
\end{scope}

\draw (-0.25,-0.25) node{$0$};
\draw (1,-0.25) node{$2$};
\draw (2,-0.25) node{$4$};
\draw (3,-0.25) node{$6$};
\draw (4,-0.25) node{$8$};
\draw (-0.25,1) node{$2$};
\draw (-0.25,2) node{$4$};

\node[label=270:{$(j_1,j_2)=(2,2)$}] (B) at (6,2.1) {};
\draw [->](1.1,1.1) to [out=45,in=180] (2.5,1.8) to [out=0,in=160] (3.5,0.7) to [out=340,in=230] (5,1.5);
\fill[red] (1,1) circle (3pt);
\fill[red] (0,2) circle (3pt);
\fill[red] (0.5,1.5) circle (3pt);
\fill[red] (3.5,0) circle (3pt);
\fill[red] (2.5,0.5) circle (3pt);
\fill[red] (0.5,1.5) circle (3pt);
\fill[red] (3.5,0) circle (3pt);
\fill[red] (2.5,0.5) circle (3pt);
\fill[blue] (0,0) circle (3pt);
\fill[blue] (0,0.5) circle (3pt);
\fill[blue] (0,1) circle (3pt);
\fill[blue] (0,1.5) circle (3pt);
\fill[blue] (0.5,0) circle (3pt);
\fill[blue] (0.5,0.5) circle (3pt);
\fill[blue] (0.5,1) circle (3pt);
\fill[blue] (1,0) circle (3pt);
\fill[blue] (1,0.5) circle (3pt);
\fill[blue] (1.5,0) circle (3pt);
\fill[blue] (1.5,0.5) circle (3pt);
\fill[blue] (2,0) circle (3pt);
\fill[blue] (2,0.5) circle (3pt);
\fill[blue] (2.5,0) circle (3pt);
\fill[blue] (3,0) circle (3pt);

\draw[red,dashed,very thick] (-0.4,1.7) -- (3.8,-0.4);

\end{tikzpicture}
\caption{An example of $\I_{<(j_1,\ldots,j_d)}$ and $\J_{j_1,\ldots,j_d}$ where $d=2$, $j_1=2$, $j_2 = 2$, $\alpha_1 = 2$, and $\alpha_2 = 3$. The blue dots represent the elements of $\I_{<(j_1,j_2)}$, and the red dots represent the elements of $\J_{j_1,j_2}$. The dashed red line represents $(l_1, l_2)$ such that $(\alpha_1-1)l_1 +(\alpha_2-1)l_2 = (\alpha_1-1)j_1 +(\alpha_2-1)j_2$.}
\label{fig1}
\end{figure}
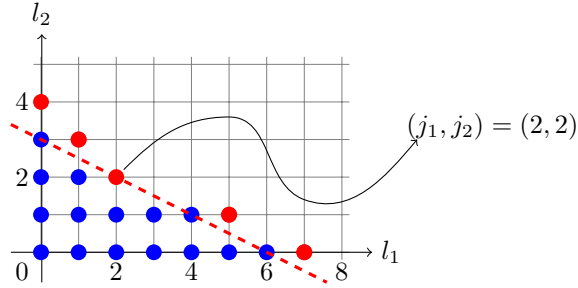
It is straightforward to show that $|\J_{j_1,\ldots,j_d}|<\infty$, and that $(m_1,\ldots,m_d)\notin \I_{<(j_1,\ldots,j_d)}$ implies $\mathbb C_{(l_1,\ldots,l_d)} \subseteq \mathbb C_{(m_1,\ldots,m_d)}$ for some $(l_1,\ldots,l_d) \in \J_{j_1,\ldots,j_d}$; 
therefore, $\mathbb D_{<(j_1,\ldots,j_d)} = \bigcap_{(l_1,\ldots,l_d)\in \J_{j_1,\ldots,j_d}} \mathbb C_{(l_1,\ldots,l_d)}$. 
In view of this and the fact that $\limsup\frac{\prod_{i=1}^d \big(n\nu_i[n,\infty)\big)^{-l_i}}{\prod_{i=1}^d \big(n\nu_i[n,\infty)\big)^{-j_i}}\to 0$ for $(l_1,\ldots,l_d)\in \J_{j_1,\ldots,j_d}\setminus \{(j_1,\ldots,j_d)\}$, the conclusion of the theorem follows from Lemma~\ref{thm:union-limsup} if we show that for each $r>0$,
$
\xi \triangleq (\xi_1,\ldots,\xi_d)\notin \big(\bigcup_{(l_1,\ldots,l_d)\in \I_{<j_1,\ldots,j_d}}\prod_{i=1}^d\mathbb D_{l_i}\big)^r
$
implies
$
\xi\notin (\mathbb C_{(l_1,\ldots,l_d)})^r
\text{ for some }(l_1,\ldots,l_d)\in \J_{j_1,\ldots,j_d}
$.
To see that this is the case, suppose that $\xi$ is bounded away from $\bigcup_{(l_1,\ldots,l_d)\in \I_{<j_1,\ldots,j_d}}\prod_{i=1}^d\mathbb D_{l_i}$ by $r>0$. Let $m_i\triangleq \inf\{k\geq0: \xi_i \in (\mathbb D_{\leqslant k})^r\}$.
In case $m_i = \infty$ for some $i$, one can pick a large enough $M\in \mathbb Z_+$ such that $M\mathbf e_i \notin \I_{<(j_1,\ldots,j_d)}$ where $\mathbf e_i$ is the unit vector with 0 entries except for the $i$-th coordinate.
Letting $(l_1,\ldots,l_d)\in \J_{j_1,\ldots,j_d}$ be an index such that $\mathbb C_{(l_1,\ldots,l_d)} \subseteq \mathbb C_{M\mathbf e_i}$, 
we find that $\xi \notin (\mathbb C_{(l_1,\ldots,l_j)})^r \subseteq (\mathbb C_{M\mathbf e_i})^r$ verifying the premise.
If $\max_{i=1,\ldots,d} m_i < \infty$,
$\xi \in (\prod_{i=1}^d \mathbb D_{m_i})^r$ and hence, $(m_1,\ldots,m_d)\notin \I_{<(j_1,\ldots,j_d)}$, which, in turn, implies that there exists $(l_1,\ldots,l_d)\in \J_{j_1,\ldots,j_d}$ such that $\mathbb C_{(l_1,\ldots,l_d)}\subseteq \mathbb C_{(m_1,\ldots,m_d)}$. 
However, due to the construction of $m_i$'s, each $\xi_i$ is bounded away from $\mathbb D_{<m_i}$ by $r$, and hence, $\xi$ is bounded away from $\mathbb D^{i-1}\times \mathbb D_{<m_i}\times \mathbb D^{d-i}$ by $r$ for each $i$. Therefore, $\xi \notin (\mathbb C_{(l_1,\ldots,l_j)})^r \subseteq (\mathbb C_{(m_1,\ldots,m_j)})^r$, and hence, the premise is verified. Now we can apply Lemma~\ref{thm:union-limsup} to reach the conclusion of the theorem.
\end{proof}

\begin{proof}[Proof of Theorem~\ref{thm:two-sided-limit-theorem}]


Let $X^{(+)}$ and $X^{(-)}$ be  L\`evy processes with spectrally positive L\'evy measures $\nu_+$ and $\nu_-$ respectively, where $\nu_+[x,\infty) = \nu[x,\infty)$ and $\nu_-[x,\infty) = \nu(-\infty,-x]$ for each $x>0$, and denote the corresponding scaled processes as 
$\bar X_n^{(+)}(\cdot) \triangleq X^{(+)}(n\cdot)/n$ and $\bar X_n^{(-)}(\cdot)\triangleq X^{(-)}(n\cdot)/n$. More specifically, let
 \begin{align*}\label{eq:plus_minus}
\bar X^{(+)}_n(s) &= sa + B(ns)/n + \frac{1}{n}\int_{|x|\leq 1} x[N([0,ns]\times dx) - ns\nu(dx)] 
\\
&
\hspace{160pt}+  \frac{1}{n}\int_{x>1} xN([0,ns]\times dx),
\\
\bar X^{(-)}_n(s) &= \frac{1}{n}\int_{x<-1} xN([0,ns]\times dx).
\end{align*}
From Theorem~\ref{thm:multi-d-limit-theorem}, we know that
$(n\nu[n,\infty))^{-j}(n\nu(-\infty,-n])^{-k}\P\big((\bar X_n^{(+)},\allowbreak \bar X_n^{(-)}) \in \cdot\big) \to C_{j}^+\times C_{k}^-(\cdot)$
in $\mathbb M\big((\mathbb D\times \mathbb D) \setminus \mathbb D_{<(j,k)}\big)$
where
$C_j^+(\cdot) \triangleq \E \Big[\nu_\alpha ^j \{x\in (0,\infty)^j:\sum_{i=1}^j x_i 1_{[U_i,1]}\in \cdot\}\Big]$
and
$C_k^-(\cdot) \triangleq \E \Big[\nu_\beta ^k \{y\in (0,\infty)^k:\sum_{i=1}^k y_i 1_{[U_i,1]}\in \cdot\}\Big]$.
In view of Lemma~\ref{lem:continuous-mapping-principle-for-subtraction} and that $C_j^+\times C_k^-\big\{(\xi,\zeta)\in \mathbb D\times\mathbb D: (\xi(t) - \xi(t-))(\zeta(t)-\zeta(t-)) \neq 0 \text{ for some } t\in(0,1]\big\} = 0$,
we can apply Lemma~\ref{lem:almost-continuous-mapping} for $h(\xi,\zeta)= \xi-\zeta$.
Noting that $C_{j,k}(\cdot) = \big(C_j^+\times C_k^-\big)\circ h^{-1}(\cdot)$, we conclude that
$(n\nu[n,\infty))^{-j}\allowbreak(n\nu(-\infty,-n])^{-k}\P\big(\bar X_n^{(+)}-\bar X_n^{(-)} \in \cdot\big) \to C_{j,k}(\cdot)$
in $\mathbb M(\mathbb D \setminus \mathbb D_{<j,k})$.
Since $\bar X_n$ has the same distribution as $\bar X_n^{(+)} - \bar X_n^{(-)}$, the desired $\mathbb M(\mathbb D \setminus \mathbb D_{<j,k})$-convergence for $\bar X_n$ follows.
\end{proof}

\begin{proof}[Proof of Lemma~\ref{wasteful-lemma}]
In general,
$$
\min_{\substack{(j,k)\in \mathbb Z_+^2\\ \mathbb D_{j,k} \cap \bar A \neq \emptyset}}\mathcal I(j,k)
\leq
\mathcal I(\mathcal J(A), \mathcal K(A))
\leq
\min_{\substack{(j,k)\in \mathbb Z_+^2\\ \mathbb D_{j,k} \cap A^\circ \neq \emptyset}}\mathcal I(j,k),
$$
and the left inequality cannot be strict since $A$ is bounded away from $\mathbb D_{<\mathcal J(A), \mathcal K(A)}$. 
On the other hand, in case the right inequality is strict, then $\mathbb D_{\mathcal J(A), \mathcal  K(A)} \cap A^\circ =\emptyset$, which in turn implies $C_{\mathcal J(A), \mathcal K(A)} (A^\circ) = 0$ since $C_{\mathcal J(A), \mathcal K(A)}$ is supported on $\mathbb D_{\mathcal J(A), \mathcal K(A)}$. Therefore, the lower bound is trivial if the right inequality is strict. In view of these observations, we can assume w.l.o.g.\ that $(\mathcal J(A), \mathcal K(A))$ is also in both
$\argmin_{\substack{(j,k)\in \mathbb Z_+^2\\ \mathbb D_{j,k} \cap A^\circ \neq \emptyset}}\mathcal I(j,k)$ and $\argmin_{\substack{(j,k)\in \mathbb Z_+^2\\ \mathbb D_{j,k} \cap \bar A \neq \emptyset}}\mathcal I(j,k)$. Since $A^\circ$ and $\bar A$ are also bounded-away from $\mathbb D_{<\mathcal J(A), \mathcal K(A)}$, the upper bound of (\ref{eq:two-sided-main-result}) is obtained from \eqref{result1ub} and Theorem~\ref{thm:two-sided-limit-theorem} for $\bar A$, $j= \mathcal J(\bar A)=\mathcal J(A)$, and $k = \mathcal K(\bar A) = \mathcal K(A)$; the lower bound of (\ref{eq:two-sided-main-result}) is obtained from \eqref{result1lb} and Theorem~\ref{thm:two-sided-limit-theorem} for $A^\circ$, $j= \mathcal J(A^\circ)=\mathcal J(A)$, and $k = \mathcal K(A^\circ) =\mathcal K(A)$.
Finally, we obtain \eqref{two-sided-limit-tends-to-zero} from Theorem~\ref{thm:two-sided-limit-theorem} and \eqref{result1ub} with $j=l$, $k = m$, $F = \bar A$ along with the fact that $C_{l,m}(\bar A) = 0$ since $A$ is bounded away from $\mathbb D_{l,m}$.
\end{proof}

\begin{lemma}\label{lemma-for-two-sided-multiple-optima}
Let $A$ be a measurable set and suppose that the argument minimum in (\ref{def:JK}) is non-empty and contains a pair of integers $(\mathcal J(A),\mathcal K(A))$.
Let $(l,m) \in \I_{=\mathcal J(A), \mathcal K(A)}$.

\begin{itemize} 
\item[(i)]
If $A_\delta \cap \mathbb D_{l,m}$ is bounded away from $\mathbb D_{\ll\mathcal J(A), \mathcal K(A)}$ for some $\delta>0$, then  $A \cap (\mathbb D_{l,m})_\gamma$ is bounded away from $\mathbb D_{\ll\mathcal J(A), \mathcal K(A)}$ for some $\gamma>0$.

\item[(ii)]
If $A$ is bounded away from $\mathbb D_{\ll \mathcal J(A),\mathcal K(A)}$, then there exists $\delta>0$ such that $A\cap (\mathbb D_{l,m})_\delta$ is bounded away from $\mathbb D_{j,k}$ for any $(j,k) \in \I_{=\mathcal J(A),\mathcal K(A)}\setminus \{(l,m)\}$. 
\end{itemize}
\end{lemma}
\begin{proof}
For (i), we prove that if
$d(A_{2\delta} \cap \mathbb D_{l,m},\, \mathbb D_{\ll \mathcal J(A), \mathcal K(A)})> 3\delta$ then
$d(A \cap (\mathbb D_{l,m})_\delta,\, \mathbb D_{\ll \mathcal J(A), \mathcal K(A)})\geq \delta$.
Suppose that $d( A \cap (\mathbb D_{l,m})_\delta,\, \mathbb D_{\ll\mathcal J(A), \mathcal K(A)}) < \delta$.
Then, there exists $\xi \in A\cap (\mathbb D_{l,m})_\delta$ and $\zeta \in \mathbb D_{\ll \mathcal J(A),\mathcal K(A)}$ such that $d(\xi, \zeta) < \delta$.
Note that we can find $\xi'\in \mathbb D_{l,m}$ such that $d(\xi, \xi') \leq 2\delta$, which means that $\xi' \in A_{2\delta} \cap \mathbb D_{l,m}$. 
Therefore, $d(A_{2\delta}\cap \mathbb D_{l,m}, \mathbb D_{\ll \mathcal J(A), \mathcal K(A)} )\leq d(\xi', \zeta) \leq d(\xi',\xi) + d(\xi, \zeta) \leq 2\delta + \delta \leq 3\delta$. 

For (ii), suppose  that $d\big(A, \mathbb D_{\ll \mathcal J(A), \mathcal K(A)}\big) > \gamma$ for some $\gamma > 0$ and  $(l,m)$ and $(j,k)$ are two distinct pairs that belong to $\I_{=\mathcal J(A), \mathcal K(A)}$. 
Assume w.l.o.g.\ that $j < l$.
(If $j > l$, it should be the case that $k < m$, and hence one can proceed in the same way by switching the roles of upward jumps and downward jumps in the following argument.)
Let $c$ be a positive number such that $c > 8(l-j)+2$ and set $\delta = \gamma / c$. 
We will show that $A \cap (\mathbb D_{l,m})_\delta$ and $(\mathbb D_{j,k})_\delta$ are bounded away from each other.
Let $\xi$ be an arbitrary element of  $A\cap (\mathbb D_{l,m})_\delta$.
Then, there exists a $\zeta \in \mathbb D_{l,m}$ such that $d(\zeta, \xi) \leq 2\delta$.
Note that $d\big(\zeta, \mathbb D_{\ll \mathcal J(A), \mathcal K(A)}\big) \geq (c-2)\delta$; in particular, $d\big(\zeta, \mathbb D_{j,m}) \geq (c-2) \delta$.
If we write $\zeta \triangleq \sum_{i=1}^l x_i 1_{[u_i,1]} - \sum_{i=1}^m y_i 1_{[v_i,1]}$, this implies that $x_{j+1} \geq \frac{(c-2)\delta}{l-j}$.
Otherwise, $(c-2)\delta > \sum_{i=j+1}^lx_i = \|\zeta - \zeta'\| \geq d(\zeta, \zeta')$, where $\zeta' \triangleq \zeta - \sum_{i=j+1}^lx_i1_{[u_i,1]} \in \mathbb D_{j,m}$.
Therefore, $d(\zeta, \mathbb D_{j,k}) \geq \frac{(c-2)\delta}{2(l-j)}$, which in turn implies $d(\xi, \mathbb D_{j,k}) \geq \frac{(c-2)\delta}{2(l-j)}-2\delta > 2\delta$.
Since $\xi$ was arbitrary, we conclude that $A\cap (\mathbb D_{l,m})_\delta$ bounded away from $(\mathbb D_{j,k})_\delta$.

\end{proof}

\subsection{Proofs for Section~\ref{sec:implications}}\label{subsec:proofs-for-implications}
Recall that
\begin{equation*}
I(\xi)\triangleq
\left\{\begin{array}{ll}
	(\alpha-1)\mathcal D_+(\xi) + (\beta-1)\mathcal D_-(\xi) & \text{if $\xi$ is a step function with $\xi(0) = 0$}
	\\
	\infty & \text{otherwise}
\end{array}
\right.
.
\end{equation*}
\begin{proof}[Proof of Theorem~\ref{thm:weak-ldp}]
Observe first that $I(\cdot)$ is a rate function. The level sets $\{\xi: I(\xi) \leq x\}$ equal $\bigcup_{\substack{(l,m)\in \mathbb Z_+^2\\(\alpha-1)l + (\beta-1)m \leq \lfloor x\rfloor}}\mathbb D_{l,m}$ and are therefore closed---note the level sets are not compact so $I(\cdot)$ is not a good rate function (see, for example, \cite{dembozeitouni} for the definition and properties of good rate functions).

Starting with the lower bound, suppose that $G$ is an open set.
We assume w.l.o.g.\ that $ \inf_{\xi\in G} I(\xi) < \infty$, since the inequality is trivial otherwise. Due to the discrete nature of $I(\cdot)$, there exists a $\xi^*\in G$ such that $I(\xi^*) = \inf_{\xi\in G} I(\xi)$.
Set $j\triangleq \mathcal D_+(\xi^*)$ and $k \triangleq \mathcal D_-(\xi^*)$.
Let $u_1^+,\ldots, u_j^+$ be the sorted (from the earliest to the latest) upward jump times of $\xi^*$; $x_1^+,\ldots, x_j^+$ be the sorted (from the largest to the smallest) upward jump sizes of $\xi^*$; $u_1^-,\ldots,u_k^-$ be the sorted downward jump times of $\xi^*$; $x_1^-,\ldots,x_k^-$ be the sorted downward jump sizes of $\xi^*$. Also, let $x_{j+1}^+=x_{k+1}^-=0$, $u_0^+=u_0^- = 0$, and $u_{j+1}^+ = u_{k+1}^-= 1$.
Note that if $\zeta \in \mathbb D_{l,m}$ for $l<j$, then $d(\xi^*,\zeta) \geq x_j^+/2$ since at least one of the $j$ upward jumps of $\xi^*$ cannot be matched by $\xi$. Likewise, if $\zeta \in \mathbb D_{l,m}$ for $m < k$, then $d(\xi^*, \zeta) \geq x_k^-/2$. Therefore, $d(\mathbb D_{<j,k}, \xi^*) \geq  (x_j^+\wedge x_k^-)/2$.
On the other hand, since $G$ is an open set, we can pick $\delta_0>0$ so that the open ball $B_{\xi^*,\delta_0}\triangleq \{\zeta\in \mathbb D: d(\zeta,\xi) < \delta_0\}$  centered at $\xi^*$ with radius $\delta_0$  is a subset of $G$---i.e., $B_{\xi^*,\delta_0} \subset G$. Let $\delta = (\delta_0 \wedge x_j^+ \wedge x_k^-)/4$.
If $j=k=0$, then $\xi^* \equiv0$, and hence, $\{\bar X_n \in G\}$ contains $\{\|\bar X_n\|\leq \delta\}$ which is a subset of $B_{\xi^*,\delta}$.
One can apply Lemma~\ref{lem:cont_etemadi} to show that $\P(X_n\in G)$ converges to 1, which, in turn, proves the inequality.
Now, suppose that either $j\geq 1$ or $k\geq 1$.
Then, $d(B_{\xi^*,\delta},\mathbb D_{<j,k})\geq \delta$. As $d(B_{\xi^*,\delta},\mathbb D_{<j,k})>0$ and $B_{\xi^*,\delta}$ is open, we see from our sharp asymptotics (Theorem~\ref{thm:one-sided-limit-theorem}) that
\begin{equation*}
 C_{j,k}(B_{\xi^*,\delta})\leq \liminf_{n\rightarrow\infty} (n\nu[n,\infty))^{-j} (n\nu(-\infty,-n])^{-k}P( \bar X_n \in B_{\xi^*,\delta}).
\end{equation*}
From the definition of $C_{j,k}$, it follows that $C_j(B_{\xi^*,\delta})>0$.
To see this, note first that we can assume w.l.o.g.\ that $x_i^\pm$'s are all distinct since $G$ is open (because, if some of the jump sizes are identical, we can pick $\epsilon$ such that $B_{\xi^*,\epsilon}\subseteq G$, and then perturb those jump sizes by $\epsilon$ to get a new $\xi^*$ which still belongs to $G$ while whose jump sizes are all distinct.)
Suppose that $\xi^* = \sum_{l=1}^j x^+_{i^+_l}1_{[u^+_i,1]}-\sum_{l=1}^k x^-_{i^-_l}1_{[u^-_i,1]}$, where $\{i^\pm_1,\ldots,i^\pm_j\}$ are permutations of $\{1,\ldots,j\}$.
Let $2\delta' \triangleq \delta \wedge \underline \Delta^+_u\wedge \underline\Delta^+_x\wedge \underline \Delta^-_u\wedge \underline\Delta^-_x$, where $\underline \Delta^+_u = \min_{i=1,\ldots,j+1} (u^+_i-u^+_{i-1})$, $\underline \Delta^+_x = \min_{i=1,\ldots,j} (x^+_{i-1}-x^+_{i})$, $\underline \Delta^-_u = \min_{i=1,\ldots,k+1} (u^-_i-u^-_{i-1})$, and $\underline \Delta^-_x = \min_{i=1,\ldots,k} (x^-_{i-1}-x^-_{i})$.
Consider a subset $B'$ of $B_{\xi^*,\delta}$:
\begin{align*}
B' \triangleq &\bigg\{\sum_{l=1}^jy^+_{i^+_l}1_{[v^+_l,1]}-\sum_{l=1}^k y^-_{i^-_l}1_{[v^-_l,1]}:
\\
&\hspace{50pt}
v^+_i \in (u^+_i-\delta', u^+_i+\delta'), y^+_i \in (x^+_i-\delta', x^+_i+\delta'), i=1,\ldots,j;
\\
&\hspace{50pt}
v^-_i \in (u^-_i-\delta', u^-_i+\delta'), y^-_i \in (x^-_i-\delta', x^-_i+\delta'), i=1,\ldots,k \bigg\}.
\end{align*}
Then,
\begin{align*}
&C_{j,k}(B_{\xi^*,\delta})
\\
&\geq C_{j,k}(B')
\\
&= \int_{(u^+_1-\delta',u^+_1+\delta')\times\cdots\times(u^+_j-\delta',u^+_j+\delta')} dLeb \cdot \int_{(x^+_1-\delta',x^+_1+\delta')\times\cdots\times(x^+_j-\delta',x^+_j+\delta')}d\nu_\alpha
\\
&\hspace{20pt}
\cdot\int_{(u^-_1-\delta',u^-_1+\delta')\times\cdots\times(u^-_k-\delta',u^-_k+\delta')} dLeb \cdot \int_{(x^-_1-\delta',x^-_1+\delta')\times\cdots\times(x^-_k-\delta',x^-_k+\delta')}d\nu_\beta
\\
&\geq {(2\delta')^{j}}{(2\delta'(x_1^+)^\alpha)^{j}}{(2\delta')^{k}}{(2\delta'(x_1^-)^{\beta})^{k}}>0.
\end{align*}
We conclude that
\begin{equation}\label{eq:getting-lower-bound}
\begin{aligned}
&\liminf_{n\rightarrow\infty}\frac{\log \P(\bar X_n \in G)}{\log n}
\geq
\liminf_{n\rightarrow\infty}\frac{\log \P(\bar X_n \in B_{\xi^*,\delta})}{\log n}
\\
&\geq
\liminf_{n\rightarrow\infty}\frac{\log (C_{j,k}(B_{\xi^*,\delta}) (n \nu[n,\infty))^j(n \nu(-\infty,-n])^k (1+o(1)))  }{\log n}
\\
&= -\big((\alpha-1)j+(\beta-1)k\big),
\end{aligned}
\end{equation}
which is the lower bound. Turning to the upper bound, suppose that $K$ is a compact set.
We first consider the case where $\inf_{\xi \in K} I(\xi) < \infty$.
Pick $\xi^*$, $j$ and $k$ as in the lower bound, i.e., $I(\xi^*) \triangleq \inf_{\xi\in K} I(\xi)$,  $j\triangleq \mathcal D_+(\xi^*)$, and $k \triangleq \mathcal D_-(\xi^*)$.
Here we can assume w.l.o.g.\ either $j \geq 1$ or $k\geq 1$ since the inequality is trivial in case $j=k=0$.
For each $\zeta \in K$, either $I(\zeta) > I(\xi^*)$, or $I(\zeta) = I(\xi^*)$. We construct an open cover of $K$  by considering these two cases separately:
\begin{itemize}
\item
If $I(\zeta) > I(\xi^*)$, $\zeta$ is bounded away from $\mathbb D_{<j,k}\cup \mathbb D_{j,k}$ (Lemma~\ref{lem:Djk} (f)).
For each such $\zeta$'s, pick a $\delta_\zeta>0$ in such a way that $d(\zeta, \mathbb D_{<j,k}\cup \mathbb D_{j,k})>\delta_\zeta$.
Set $j_\zeta \triangleq j$ and $k_\zeta \triangleq k$.
Note that in this case $C_{j_\zeta,k_\zeta}(\bar B_{\zeta,\delta_\zeta}) = 0$.
\item
If $I(\zeta) = I(\xi^*)$, set $j_\zeta \triangleq \mathcal D_+(\zeta)$ and $k_\zeta \triangleq \mathcal D_-(\zeta)$.
Since they are bounded away from $\mathbb D_{<j_\zeta,k_\zeta}$ (Lemma~\ref{lem:Djk} (e)), we can choose $\delta_\zeta>0$ such that $d(\zeta, \mathbb D_{<j_\zeta,k_\zeta}) > \delta_\zeta$ and $C_{j_\zeta,k_\zeta}(\bar B_{\zeta,\delta_\zeta}) < \infty$.
\end{itemize}
Consider an open cover $\{B_{\zeta;\delta_\zeta}: \zeta\in K\}$ of $K$ and its finite subcover
$\{B_{\zeta_i;\delta_{\zeta_i}} \}_{i=1,\ldots,m}$. For each $\zeta_i$, we apply the sharp asymptotics (Theorem~\ref{thm:two-sided-limit-theorem}) to $\bar B_{\zeta_i; \delta_{\zeta_i}}$ 
to get
\begin{equation}\label{ineq:some-upper-bound-in-sec-5}
\limsup_{n\rightarrow\infty} \frac{\log \P(\bar X_n \in \bar B_{\zeta_i; \delta_{\zeta_i}})}{\log n}
\leq
(\alpha-1)j_{\zeta_i} + (\beta-1)k_{\zeta_i}
=
-I(\xi^*).
\end{equation}
Therefore,
\begin{align}
\limsup_{n\rightarrow\infty} \frac{\log \P(\bar X_n \in \bar F)}{\log n}
&
\leq
\limsup_{n\rightarrow\infty} \frac{\log \sum_{i=1}^m \P(\bar X_n \in \bar B_{\zeta_i; \delta_{\zeta_i}})}{\log n}
\nonumber
\\
&
=
\max_{i=1,\ldots,m}\limsup_{n\rightarrow\infty} \frac{\log \P(\bar X_n \in \bar B_{\zeta_i; \delta_{\zeta_i}})}{\log n}
\nonumber
\\
&\leq
-I(\xi^*)  = -\inf_{\xi\in K} I(\xi),\label{ineq:another-upper-bound-in-sec-5}
\end{align}
completing the proof of the upper bound in case the right-hand side is finite.

Now, turning to the case $\inf_{\xi \in K} I(\xi) = \infty$, fix an arbitrary positive integer $l$. 
Since $\mathbb D_{<l,l}$ is closed and disjoint with a compact set $K$, it is also bounded away from each $\zeta\in K$. 
Now picking $\delta_\zeta>0$ so that $\bar B_{\zeta,\delta_\zeta}$ is disjoint with $K$ for each $\zeta$, one can construct an open cover $\{B_{\zeta;\delta_\zeta}:\zeta \in K\}$ of $K$. 
Let $\{B_{\zeta_i;\delta_{\zeta_i}}\}_{i=1,\ldots,m}$ its finite subcover, then from the same calculation as \eqref{ineq:some-upper-bound-in-sec-5} and \eqref{ineq:another-upper-bound-in-sec-5},
\begin{equation*}
\limsup_{n\rightarrow\infty} \frac{\log \P(\bar X_n \in K)}{\log n} \leq -(\alpha+\beta-2)m.
\end{equation*}
Taking $m\to\infty$, we arrive at the desired upper bound.
\end{proof}


\section{Applications}\label{sec:applications}
In this section, we illustrate the use of our main results, established in Section~\ref{sec:sample-path-ldps}, in several problem contexts that arise in control, insurance, and finance.
In all examples, we assume that
$\bar{X}_{n}\left(  t\right)  =X\left(  nt\right)  /n$, where $X\left(\cdot\right)$ is a centered L\'{e}vy process satisfying (1.1).

\subsection{Crossing High Levels with Moderate Jumps}\label{subsec:moderate-jumps}

We are interested in level crossing probabilities of L\'evy processes where the jumps are conditioned to be moderate. More precisely, we are interested in probabilities of the form $\P\big(\sup_{t\in [0,1]} [\bar X_n(t)-ct] \geq a ; \;\sup_{t\in [0,1]} [\bar X_n(t)-\bar X_n(t-)]\leq  b\big)$. We make a technical assumption that $a$ is not a multiple of $b$ and focus on the case where the L\'evy process $\bar X_n$ is spectrally positive.


The setting of this example is relevant in, for example, insurance, where huge claims may be reinsured and therefore do not play a role in the ruin of an insurance company.
 \cite{AsmussenPihlsgaard} focus on obtaining various estimates  of infinite-time ruin probabilities  using analytic methods.
Here, we provide complementary sharp asymptotics for the finite-time ruin probability, using probabilistic techniques.

Set $A \triangleq \{\xi\in \mathbb D: \sup_{t\in [0,1]} [\xi(t)-ct]\geq a ; \sup_{t\in [0,1]} [\xi(t)-\xi(t-)] \leq  b\}$ and define
$
j \triangleq \lceil a/b \rceil.
$
Intuitively, $j$ should be the key parameter, as it takes at least $j$ jumps of size $b$ to cross level $a$.
Our goal is to make this intuition rigorous by applying Theorem~\ref{thm:one-sided-main-theorem} and by showing that the upper and lower bounds are tight.

We first check that $A_\delta \cap \mathbb D_j$ is bounded away from the closed set $\mathbb D_{\leqslant j-1}$ for some $\delta>0$. To see this, it suffices to show that
\begin{itemize}
\item[1)] $\sup_{t\in[0,1]}[\xi(t)-\xi(t-)] \leq b$ and $\sup_{t\in[0,1]}[\zeta(t)-\zeta(t-)] > b'$ imply $d(\xi,\zeta) > \frac{b'-b}{3}$; and
\item[2)]  $\sup_{t\in [0,1]}[\xi(t) - ct]< a'$ and $\sup_{t\in[0,1]}[\zeta(t)-ct] \geq a$ imply $d(\xi,\zeta) \geq \frac{a-a'}{c+1}$.
\end{itemize}
It is straightforward to check 1).
To see 2), note that for any $\epsilon>0$, one can find $t^*$ such that $\zeta(t^*) - ct^* \geq a-\epsilon$. Of course, $\xi(\lambda(t^*)) - c\lambda(t^*) < a'$ for any homeomorphism $\lambda (\cdot )$. Subtracting the latter inequality from the former inequality, we obtain
\begin{equation}\label{eq:mod-jump-diff}
\zeta(t^*)-\xi(\lambda(t^*)) \geq  a - a' -\epsilon + c(t^* - \lambda(t^*)).
\end{equation}
One can choose $\lambda$ so that $d(\xi,\zeta) + \epsilon \geq \| \lambda - e\| \geq \lambda(t^*)- t^*$ and $d(\zeta, \xi) + \epsilon \geq \|\zeta - \xi \circ \lambda\|\geq \zeta(t^*) - \xi(\lambda(t^*))$, which together with (\ref{eq:mod-jump-diff}) yields
$$
d(\xi, \zeta) >  a-a'-(c+1)\epsilon-cd(\xi,\zeta).
$$
This leads to $d(\xi,\zeta) \geq \frac{a-a'}{c+1}$ by taking $\epsilon \to 0$.
With 1) and 2) in hand, it follows that $\phi_1(\xi) \triangleq \sup_{t\in[0,1]}[\xi(t)-\xi(t-)]$ and $\phi_2(\xi) \triangleq \sup_{t\in [0,1]}[\xi(t) - ct]$ are continuous functionals and $A_{\delta}\subseteq A(\delta)$, where $A(\delta) \triangleq \{\xi\in \mathbb D: \sup_{t\in [0,1]} [\xi(t)-ct]\geq a-(c+1)\delta ; \sup_{t\in [0,1]} [\xi(t)-\xi(t-)] \leq  b+3\delta\}$.
Since $\xi \in A(\delta)\cap \mathbb D_j$ implies that the jump size of $\xi$ is bounded from below by $(b+3\delta)j-(a-(c+1)\delta)$, one can choose $\delta>0$ so that $A(\delta)\cap \mathbb D_j$ is bounded away from $\mathbb D_{\leqslant j-1}$.
This implies that $A_\delta\cap \mathbb D_j$ is also bounded away from $\mathbb D_{\leqslant j-1}$ for sufficiently small $\delta>0$. Hence, Theorem~\ref{thm:one-sided-main-theorem} applies with $\mathcal J(A) = j$.

Next, to identify the limit, recall the discussion at the end of Section~\ref{subsec:one-sided-large-deviations}.
Note that $A = \phi_1^{-1}[a,\infty) \cap \phi_2^{-1}(-\infty,b]$ and
\begin{equation}\label{eq:constraints-for-B}
\begin{aligned}
&\hat T_{j}^{-1}(\phi_1^{-1}[a,\infty) \cap \phi_2^{-1}(-\infty,b])
\\
&= \textstyle{\left\{(x,u)\in \hat S_j: \sum_{i=1}^j x_i \geq a+ c\max_{i=1,\ldots,j} u_i,\ \max_{i=1,\ldots,j} x_i \leq b\right\}},\\
&\hat T_{j}^{-1}(\phi_1^{-1}(a,\infty) \cap \phi_2^{-1}(-\infty,b))
\\
&= \textstyle{\left\{(x,u)\in \hat S_j: \sum_{i=1}^j x_i > a+ c\max_{i=1,\ldots,j} u_i,\ \max_{i=1,\ldots,j} x_i < b\right\}}.
\end{aligned}
\end{equation}
We see that $\hat T_{j}^{-1}(\phi_1^{-1}[a,\infty) \cap \phi_2^{-1}(-\infty,b]) \setminus \hat T_{j}^{-1}(\phi_1^{-1}(a,\infty) \cap \phi_2^{-1}(-\infty,b))$ has Lebesgue measure 0, and hence, $A$ is $C_j$-continuous. Thus, (\ref{A-continuity}) holds with
\begin{equation*}
C_{j}(A)
= \E \left[\nu_\alpha^j\{(0,\infty)^j: \sum_{i=1}^jx_i1_{[U_i,1]} \in A\}\right]
= \int_{(x,u)\in \hat T_{j}^{-1}(A) } \prod_{i=1}^j [\alpha x_i^{-\alpha-1} dx_i du_i]>0.
\end{equation*}
Therefore, we conclude that
\begin{equation}
\P\left(\sup_{t\in [0,1]}[ \bar X_n(t)-ct] \geq a ; \sup_{t\in [0,1]} [\bar X_n(t)-\bar X_n(t-)] \leq  b\right) \sim C_{j}\big(A\big) (n\nu[n,\infty))^{j}.
\end{equation}
In particular, the probability of interest is regularly varying with index $-(\alpha-1)\lceil a/b \rceil$.

\subsection{A Two-sided Barrier Crossing Problem}\label{subsec:two-sided-barrier}
 We consider a L\'{e}vy-driven Ornstein-Uhlenbeck process of the form%
\[
d\bar{Y}_{n}\left(  t\right)     =-\kappa d\bar{Y}_{n}\left(  t\right)
+d\bar{X}_{n}\left(  t\right), \hspace{1cm}
\bar{Y}_{n}\left(  0\right)     =0.
\]
We apply our results to provide sharp large-deviations estimates for
\[
b\left(  n\right)  =\P\left(  \inf\{\bar{Y}_{n}\left(  t\right)
:0\leq t\leq1\}\leq-a_{-},\bar{Y}_{n}\left(  1\right)  \geq a_{+}\right)
\]
as $n\rightarrow\infty$, where $a_{-},a_{+}>0$. This probability can be
interpreted as the price of a barrier digital option (see \citealp{cont2004financial}, Section 11.3).
In order to apply our results it is useful to represent $\bar{Y}_{n}$ as an
explicit function of $\bar{X}_{n}$. In particular, we have that
\begin{align}
\bar{Y}_{n}\left(  t\right)   &  =\exp\left(  -\kappa t\right)  \left(
\bar{Y}_{n}\left(  0\right)  +\int_{0}^{t}\exp\left(  \kappa s\right)
d\bar{X}_{n}\left(  s\right)  \right) \label{REP_1}\\
&  =\bar{X}_{n}\left(  t\right)  -\kappa\exp\left(  -\kappa t\right)  \int%
_{0}^{t}\exp\left(  \kappa s\right)  \bar{X}_{n}\left(  s\right)  ds.
\label{REP_2}%
\end{align}
Hence, if $\phi:\mathbb{D}\left(  [0,1],\mathbb{R}\right)  \rightarrow
\mathbb{D}\left(  [0,1],\mathbb{R}\right)  $ is defined via
\[
\phi\left(  \xi\right)  \left(  t\right)  =\xi\left(  t\right)  -\kappa
\exp\left(  -\kappa t\right)  \int_{0}^{t}\exp\left(  \kappa s\right)
\xi\left(  s\right)  ds,
\]
then $\bar{Y}_{n}=\phi\left(  \bar{X}_{n}\right)  $. Moreover, if we let%
\[
A=\left\{  \xi\in\mathbb{D}:\inf_{0\leq t\leq1}\phi\left(  \xi\right)  \left(
t\right)  \leq-a_{-},\phi\left(  \xi\right)  \left(  1\right)  \geq
a_{+}\right\}  ,
\]
then we obtain
\[
b\left(  n\right)  ={\P}\left(  \bar{X}_{n}\in A\right)  .
\]
In order to easily verify topological properties of $A$, let us define
$
m,\pi_{1}:\mathbb{D}(  [0,1],\allowbreak\mathbb{R})  \rightarrow\mathbb{R}%
$
by
$
m\left(  \xi\right)  =\inf_{0\leq t\leq1}\xi\left(  t\right)  ,\text{ and }%
\pi_{1}\left(  \xi\right)  =\xi\left(  1\right).
$
Note that $\pi_{1}$ is continuous (see \citealp{billingsley2013convergence},
Theorem~12.5), that $m$ is continuous as well, and so is $\phi$.
Thus, $m\circ\phi$ and $\pi_{1}\circ\phi$ are continuous. We
 can therefore write%
\[
A=\left(  m\circ\phi\right)  ^{-1}(-\infty,-a_{-}]\cap\left(  \pi_{1}\circ
\phi\right)  ^{-1}[a_{+},\infty),
\]
concluding that $A$ is a closed set.
We now apply Theorem~\ref{thm:two-sided-main-theorem}. To show that $\mathbb{D}_{i,0}$ is bounded away
from $\left(  m\circ\phi\right)  ^{-1}(-\infty,-a_{-}]$, select $\theta$ such
that $d\left(  \theta,\mathbb{D}_{i,0}\right)  <r$ with $r<a_{-}/\left(
1+\kappa\exp\left(  \kappa\right)  \right)  $. There exists a
$\xi\in\mathbb{D}_{i,0}$ such that $d\left(  \theta,\xi\right)  <r$ and $\xi$ satisfies
$
\xi\left(  t\right)  =\sum_{j=1}^{i}x_{j}I_{[u_{j},1]}\left(  t\right)  ,
$
with $i\geq1$. There also exists a homeomorphism $\lambda:[0,1]\rightarrow
\lbrack0,1]$ such that
\begin{equation}
\sup_{t\in\lbrack0,1]}\left\vert \lambda\left(  t\right)  -t\right\vert
\vee\left\vert \left(  \xi\circ\lambda\right)  \left(  t\right)
-\theta\left(  t\right)  \right\vert <r.\label{BND_AUX_1}%
\end{equation}
Now, define $\psi=\theta-\left(  \xi\circ\lambda\right)$. Due to
the linearity of $\phi$, and representations (\ref{REP_1}) and
(\ref{REP_2}), we obtain that
\begin{align*}
\phi\left(  \theta\right)  \left(  t\right)   &  =\phi\left(  \left(  \xi
\circ\lambda\right)  \right)  \left(  t\right)  +\phi\left(  \psi\right)
\left(  t\right)  \\
&  =\exp\left(  -\kappa t\right)  \sum_{j=1}^{i}\exp\left(  \kappa\lambda
^{-1}\left(  u_{j}\right)  \right)  x_{j}I_{[\lambda^{-1}\left(  u_{j}\right)
,1]}\left(  t\right)    +\psi\left(  t\right) \\
&\hspace{50pt} -\kappa\exp\left(  -\kappa t\right)  \int_{0}%
^{t}\exp\left(  \kappa s\right)  \psi\left(  s\right)  ds.
\end{align*}
Since $x_{j}\geq0$, applying the triangle inequality and inequality
(\ref{BND_AUX_1}) we conclude (by our choice of $r$), that
\[
\inf_{0\leq t\leq1}\phi\left(  \theta\right)  \left(  t\right)  \geq-r\left(
1+\kappa\exp\left(  \kappa\right)  \right)  >-a_{-}.
\]
A similar argument allows us to conclude that $\mathbb{D}_{0,i}$ is bounded
away from $\left(  \pi_{1}\circ\phi\right)  ^{-1}[a_{+},\infty)$. Hence, in
addition to being closed, $A$ is bounded away from $\mathbb{D}_{0,i}%
\cup\mathbb{D}_{i,0}$ for any $i\geq1$.
Moreover, let $\xi\in A\cap\mathbb{D}_{1,1},$
with
\begin{equation}
\xi\left(  t\right)  ={x}I_{[{u},1]}(t)-{y}I_{[{v},1]}(t), \label{fREP}%
\end{equation}
where ${x}>0$ and ${y}>0$. Using (\ref{REP_1}), we obtain that $\xi\in
A\cap\mathbb{D}_{1,1},$ is equivalent to
\[
{y}\geq a_{-}\text{, \ }{u}>{v}\text{, and }{x}\geq a_{+}\exp\left(  \kappa\left(
1-{u}\right)  \right)  +{y}\exp\left(  -\kappa\left(  {u}-{v}\right)  \right)
.
\]
Now, we claim that
\begin{align}
A^{\circ}  &  =\left\{  \xi\in\mathbb{D}:\inf_{0\leq t\leq1}\phi\left(
\xi\right)  \left(  t\right)  <-a_{-},\phi\left(  \xi\right)  \left(
1\right)  >a_{+}\right\} \label{A_INT}\\
&  =\left(  m\circ\phi\right)  ^{-1}(-\infty,-a_{-})\cap\left(  \pi_{1}%
\circ\phi\right)  ^{-1}(a_{+},\infty).\nonumber
\end{align}
It is clear that $A^{\circ}$ contains the open set in the right hand side. We
now argue that such a set is actually maximal, so that equality
holds. Suppose that $\phi\left(  \xi\right)  \left(  1\right)  =a_{+}$, while
$\min_{0\leq t\leq1}\phi\left(  \xi\right)  \left(  t\right)  <-a_{-}$. We
then consider $\psi=-\delta I_{\{1\}}\left(  t\right)$ with $\delta>0$,
and note that $d\left(  \xi,\xi+\psi\right)  \leq\delta$, and
\[
\phi\left(  \xi+\psi\right)  \left(  t\right)  =\phi\left(  \xi\right)
\left(  t\right)  I_{[0,1)}\left(  t\right)  +\left(  a_{+}-\delta\right)
I_{\{1\}}\left(  t\right)  ,
\]
so that $\xi+\psi\notin A$. Similarly, we can see that the other
inequality (involving $a_{-}$) must also be strict, hence concluding that
(\ref{A_INT}) holds.

We deduce that, if $\xi\in A^{\circ}\cap\mathbb{D}_{1,1}$ with
$\xi$ satisfying (\ref{fREP}), then%
\[
{y}>a_{-}\text{, \ }{u}>{v}\text{, }{x}>a_{+}\exp\left(  \kappa\left(
1-{u}\right)  \right)  +{y}\exp\left(  -\kappa\left(  {u}-{v}\right)  \right)
.
\]
Thus, we can see that $A$ is $C_{1,1}\left(  \cdot\right)  $-continuous,
either directly or by invoking our discussion in
Section~\ref{subsec:one-sided-large-deviations} regarding continuity of sets.
Therefore, applying Theorem~\ref{thm:two-sided-main-theorem}, we conclude that
\[
b\left(  n\right)  \sim n\nu\lbrack n,\infty)n\nu(-\infty,-n]C_{1,1}\left(
A\right)
\]
as $n\rightarrow\infty$, where
\[
C_{1,1}\left(  A\right)  =\int_{0}^{1}\int_{a_{-}}^{\infty}\int_{{v}}^{1}%
\int_{a_{+}\exp\left(  \kappa\left(  1-{u}\right)  \right)  +{y}\exp\left(
-\kappa\left(  {u}-{v}\right)  \right)  }^{\infty}\nu_{\alpha}(dx)\,du\,\nu
_{\beta}(dy)dv.
\]
In particular, the probability of  interest is regularly varying with index $2-\alpha-\beta$.

\subsection{Identifying the Optimal Number of Jumps for Sets of the Form $A=\{\xi:  l \leq \xi \leq u\}$}\label{subsec:sausage}
The sets that appeared in the examples in Section~\ref{subsec:moderate-jumps} and Section~\ref{subsec:two-sided-barrier} lend themselves to a direct characterization of the optimal numbers of jumps $(\mathcal J(A), \mathcal K(A))$. However, in more complicated problems, deciding what kind of paths the most probable limit behaviors consist of may not be as obvious. In this section, we show that for sets of a certain form, we can identify an optimal path.
Consider continuous real-valued functions $l$ and $u$, which satisfy $l(t) < u(t)$ for every $t\in [0,1]$, and suppose that $l(0)<0<u(0)$. Define $A=\{\xi: l(t) \leq \xi(t) \leq u(t)\}$.
We assume that both $\alpha,\beta<\infty$, which is the most interesting case.

The goal of this section is to construct an algorithm which yields an expression for $\mathcal J(A)$ and $\mathcal K(A)$.
In fact, we can completely identify a function $h$ that solves the optimization problem defining $(\mathcal J(A),\mathcal K(A))$. This function will be a step function with both positive and negative steps.
We first construct such a function, and then verify its optimality. The first step is to identify the times at which this function jumps.
Define the sets
\begin{align*}
A_t &\triangleq \{x: l(t) \leq x \leq u(t) \},
&
A_{s,t}^* &\triangleq \cap_{s\leq r\leq t} A_r,
\end{align*}
and the times $(t_n, n\geq 1)$ by
\begin{align*}
t_{n+1} &\triangleq 1 \wedge \inf \{ t>t_n : A_{\tau_n,t} = \emptyset\} \text{ \ for \ }n\geq 2,
&t_1 &\triangleq 1 \wedge \inf \{t>0: 0\notin A_t\}.
\end{align*}
Let $n^*= \inf \{n\geq 1: t_n=1\}$. Assume that $n^*>1$, since the zero function is the obvious optimal path in case $n^*=1$. Due to the construction of the times $t_n, n\geq 1,$ we have the following properties:
\begin{itemize}
\item Either $l(t_1)=0$ or $u(t_1) = 0$.
\item For every $n = 1,\ldots,n^*-2$, $\sup_{t\in [t_{n},t_{n+1}]}l(t) = \inf_{t\in [t_{n},t_{n+1}]}u(t)$.
\item $H_{fin} \triangleq [\sup_{t\in [t_{n^*-1},t_{n^*}]}l(t), \inf_{t\in [t_{n^*-1},t_{n^*}]}u(t)]$ is nonempty.
\end{itemize}
Set $h_n \triangleq \sup_{t\in [t_{n},t_{n+1}]}l(t)$ for $n=1,\ldots,n^*-1$, and set $h_{n^*-1} \triangleq h_{fin}$ for any $h_{fin} \in  H_{fin}$.
Define now $h(t)$ as $0$ on $t\in [0,t_1)$, $h(t) = h_n$ on $t\in [t_{n},t_{n+1})$ for $n=1,\ldots,n^*-2$, and $h(t) = h_{n^*-1}$ on $t\in [t_{n^*-1}, 1]$.
We claim now that $(\mathcal J(A), \mathcal K(A)) = (\mathcal J(\{h\}), \mathcal K(\{h\}))$.
In fact, we can prove that if $g\in A$ is a step function, $\mathcal D_+(g) \geq \mathcal D_+(h)$ and $\mathcal D_-(g) \geq \mathcal D_-(h)$, which implies the optimality of $h$. The proof is based on the following observation. At each $t_{n+1}$, either
\begin{itemize}
\item[1)] for any $\epsilon>0$ one can find $t\in[t_{n+1}, t_{n+1}+\epsilon]$ such that $u(t)<h_{n}$, or
\item[2)] for any $\epsilon>0$ one can find $t\in[t_{n+1}, t_{n+1}+\epsilon]$ such that $l(t)>h_{n}$.
\end{itemize}
Otherwise, there exists $\epsilon>0$ such that $h_n \in A_{t_n,t_{n+1}+\epsilon}$, contradicting the definition of $t_n$, which requires $A_{t_n, t_{n+1}+\epsilon} = \emptyset$.
From this observation, we can prove that on each interval $(t_n, t_{n+1}]$, any feasible path must jump at least once in the same direction as that of the jump of $h$.
To see this, first suppose that 1) is the case at $t_{n+1}$, and $g\in A$ is a step function.
Note that due to its continuity, $l(\cdot)$ should have achieved its supremum at $t_{sup}\in [t_n,t_{n+1}]$, i.e., $l(t_{sup}) = h_n$, and hence, $g(t_{sup})\geq h_n$.
On the other hand, due to the right continuity of $g$ and 1), $g$ has to be strictly less than $h_{n}$ at $t_{n+1}$, i.e., $g(t_{n+1})< h_n$.
Therefore, $g$ must have a downward jump on $(t_{sup},t_{n+1}]\subseteq (t_n,t_{n+1}]$.
Note that the direction of the jump of $h$ in the interval $(t_n,t_{n+1}]$ (more specifically at $t_{n+1}$) also has to be downward.
Since $g$ is an arbitrary feasible path, this means that whenever $h$ jumps downward on $(t_n,t_{n+1})$, any feasible path in $A$ should also jump downward. Hence, any feasible path must have either equal or a greater number of downward jumps as $h$'s on $[0,1]$. Case 2) leads to a similar conclusion about the number of upward jumps of feasible paths. The number of upward jumps of $h$ is optimal, proving that $h$ is indeed the optimal path.

\subsection{Multiple Optima} \label{subsec:multiple-asymptotics-example}
This section illustrates how to handle a case where we require Theorem~\ref{thm:two-sided-multiple-asymptotics}, and consider an illustrative example
where a rare event can be caused by two different configurations of big jumps.
Suppose that the regularly varying indices $-\alpha$ and $-\beta$ for positive and negative parts of the L\'evy measure $\nu$ of $X$ are equal, and consider the set $A \triangleq \{\xi\in \mathbb D: |\xi(t)| \geq t - 1/2\}$.
Then, $\argmin_{\substack{(j,k)\in \mathbb Z_+^2\\ \mathbb D_{j,k} \cap A \neq \emptyset}}\mathcal I(j,k) = \{(1,0), (0,1)\}$, and $\mathbb D_{\ll 1,0} = \mathbb D_{\ll 0,1} = \mathbb D_{0,0}$.
Since $|\xi(1)| \geq 1/2$ for any $\xi\in A$, $d(A, \mathbb D_{0, 0}) = 1/2>0$.
Theorem~\ref{thm:two-sided-multiple-asymptotics} therefore applies, and for each $\epsilon>0$, there exists $N$ such that
\begin{align*}
\P(\bar X_n \in A)
&
\geq
\frac{\big(C_{l,m}(A^\circ\cap \mathbb D_{1,0})-\epsilon\big)L_+(n) + \big(C_{l,m}( A^\circ\cap \mathbb D_{0,1})-\epsilon\big)L_-(n)}{n^{\alpha-1}},
\\
\P(\bar X_n \in A)
&\leq
\frac{\big(C_{l,m}(A^-\cap \mathbb D_{1,0})+\epsilon\big)L_+(n) + \big(C_{l,m}(A^-\cap \mathbb D_{0,1})+\epsilon\big)L_-(n)}{n^{\alpha-1}},
\end{align*}
for all $n\geq N$.
Note that $A$ is closed, since if there is $\xi\in \mathbb D$ and $s\in [0,1]$ such that $|\xi(s)| < s-1/2$, then $B(\xi, \frac{s-1/2-\xi(s)}{2})\subseteq A^c$. Therefore,  $A^-\cap \mathbb D_{1,0} = A\cap \mathbb D_{1,0} = \{\xi=x1_{[u,1]}: x \geq 1/2, 0 < u \leq 1/2\}$, and hence,
$C_{1,0}(A^- \cap \mathbb D_{1,0}) = \P(U_1 \in (0,1/2])\nu_\alpha[1/2,\infty) = (1/2)^{1-\alpha}$.
Noting that $ A^\circ \cap \mathbb D_{1,0} \supseteq(A\cap \mathbb D_{1,0})^\circ = \{\xi = x1_{[u,1]}: x > 1/2, 0<u<1/2\}$, we deduce $C_{1,0}( A^\circ\cap \mathbb D_{1,0})\geq \P(U_1 \in (0,1/2))\nu_\alpha (1/2,\infty) = (1/2)^{1-\alpha}$.
Therefore, $C_{1,0}(A^\circ \cap \mathbb D_{1,0}) = C_{1,0}(A^- \cap \mathbb D_{1,0}) = (1/2)^{1-\alpha}$.
Similarly, we can check that
$C_{0,1}( A^\circ \cap \mathbb D_{0,1}) = C_{0,1}(A^- \cap \mathbb D_{0,1}) = (1/2)^{1-\beta}\, (=(1/2)^{1-\alpha})$. Therefore, for $n\geq N$,
\begin{align*}
&
((1/2)^{1-\alpha} - \epsilon)(L_+(n)+L_-(n))n^{1-\alpha} 
\\
&
\hspace{97pt}
\leq \P(\bar X_n \in A) 
\\
&
\hspace{97pt}
\leq ((1/2)^{1-\alpha} + \epsilon)(L_+(n)+L_-(n))n^{1-\alpha}.
\end{align*}
This is equivalent to
\begin{align*}
\left(\frac12 \right)^{1-\alpha}
&
\leq
\liminf_{n\to\infty}\frac{\P(\bar X_n \in A)}{(L_+(n)+L_-(n))n^{1-\alpha}}\\
&\leq
\limsup_{n\to\infty}\frac{\P(\bar X_n \in A)}{(L_+(n)+L_-(n))n^{1-\alpha}}
\leq
\left(\frac12 \right)^{1-\alpha}.
\end{align*}
Hence,
$$
\lim_{n\to\infty}\frac{\P(\bar X_n \in A)}{(L_+(n)+L_-(n))n^{1-\alpha}}
=
\left(\frac12 \right)^{1-\alpha}.
$$

\appendix

\section{Inequalities}

\begin{lemma}[Generalized Kolmogorov inequality; \cite{Shneer2009}]\label{result:gen_kol}
Let $S_n = X_1+\cdots+X_n$ be a random walk with mean zero increments, i.e., $\E X_i = 0$.
Then,
$$\P(\max_{k\leq n} S_k \geq x) \leq C\frac{nV(x)}{x^2},$$
where $V(x) = \E (X_1^2; |X_1|\leq x)$, for all $x>0$.
\end{lemma}

\begin{lemma}[Etemadi's inequality]\label{result:etimedi}
Let $X_1, ..., X_n$ be independent real-valued random variables defined on some common probability space, and let $\alpha \geq 0$. Let $S_k$ denote the partial sum $S_{k} = X_{1} + \cdots + X_{k}$.
Then
$$\mathbb{P} \Bigl( \max_{1 \leq k \leq n} | S_{k} | \geq 3 x \Bigr) \leq 3 \max_{1 \leq k \leq n} \mathbb{P} \bigl( | S_{k} | \geq x \bigr).$$
\end{lemma}

\begin{lemma}[Prokhorov's inequality; \cite{Prohorov}]\label{result:prokhorov}
Suppose that $\xi_i$, $i=1,\ldots,n$ are independent, zero-mean random variables such that there exists a constant $c$ for which $|\xi_i| \leq c$ for $i=1,\ldots,n$, and $\sum_{i=1}^n \var \xi_i < \infty$. Then
$$\P\left(\sum_{i=1}^n \xi_i > x\right) \leq \exp \left\{-\frac{x}{2c} \arcsinh\frac{xc}{2\sum_{i=1}^n \var \xi_i} \right\}, $$
which, in turn, implies
$$\P\left(\sum_{i=1}^n \xi_i > x\right) \leq \left(\frac{cx}{\sum_{i=1}^n \var \xi_i}\right)^{-\frac{x}{2c}}. $$
\end{lemma}

We extend the Etemadi's inequality to L\'evy processes in the following lemma.
\begin{lemma}\label{lem:cont_etemadi}
Let $Z$ be a L\'evy process. Then,
$$\mathbb{P} \Bigl( \sup_{t \in [0,n]} | Z(t) | \geq  x \Bigr) \leq 3 \sup_{t\in[0,n]} \P \bigl( | Z(t) | \geq x/3 \bigr).$$
\end{lemma}
\begin{proof}
Since $Z$ (and hence $|Z|$ also) is in $\mathbb D$, $\sup_{0\leq k\leq 2^m}|Z(\frac{nk}{2^m})|$ converges to $\sup_{t\in[0,n]}|Z(t)|$ almost surely as $m\to\infty$.
To see this, note that one can choose $t_i$'s such that $|Z(t_i)| \geq \sup_{t\in[0,n]}|Z(t)|-i^{-1}$. Since $\{t_i\}$'s are in a compact set $[0,n]$, there is a subsequence, say, $t'_i$, such that $t'_i \to t_0$ for some $t_0 \in [0,n]$. The supremum has to be achieved at either $t_0^-$ or $t_0$. Either way, with large enough $m$, $\sup_{0\leq k\leq 2^m}|Z(\frac{nk}{2^m})|$ becomes arbitrarily close to the supremum. Now, by bounded convergence,
\begin{align}
&
\P\left\{\sup_{t\in[0,n]} |Z(t)| > x\right\}
\nonumber
\\
&= \lim_{m\to\infty} \P\left\{\sup_{0\leq k \leq 2^m}\left|Z(\frac{nk}{2^m})\right|>x\right\}\nonumber\\
&= \lim_{m\to\infty} \P\left\{\sup_{0\leq k \leq 2^m}\left|\sum_{i=0}^k\left(Z(\frac{ni}{2^m})-Z(\frac{n(i-1)}{2^m})\right)\right|>x\right\}\nonumber\\
&\leq \lim_{m\to\infty} 3\sup_{0\leq k \leq 2^m}\P\left\{\left|\sum_{i=0}^k\left(Z(\frac{ni}{2^m})-Z(\frac{n(i-1)}{2^m})\right)\right|>x/3\right\}\nonumber\\
&= \lim_{m\to\infty} 3\sup_{0\leq k \leq 2^m}\P\left\{\left|Z(\frac{nk}{2^m})\right|>x/3\right\}\nonumber\\
&\leq 3\sup_{t\in[0,n]}\P\left\{\left|Z(t)\right|>x/3\right\},\nonumber
\end{align}
where $Z(t) \triangleq 0$ for $t<0$.
\end{proof}

\revrem{

\newpage
\section{List of Notations}
\begin{itemize}
\item $(\mathbb S,d)$: complete separable metric space
\item $F_\delta \triangleq \{x \in \mathbb D: d(x,F) \leq \delta\}$

\item $G^{-\delta} \triangleq ((G^c)_\delta)^c$

\item $A^\circ$: interior of $A$ \qquad\qquad
\\
$\bar A$: closure of $A$ \qquad\qquad \\
$\partial A = \bar A\setminus A^\circ$: boundary of $A$

\item $\nu$: regularly varying L\'evy measure with index $-\alpha$ and $-\beta$\\
i.e., $\nu[n,\infty) = n^{-\alpha}L_+(n)$ and $\nu(-\infty,-n] = n^{-\beta}L_-(n)$\\
$L_+(n) = n^\alpha\nu[n,\infty)$ \\
$L_-(n) = n^\beta\nu(-\infty,-n]$

\item $X$: L\'evy process with L\'evy measure $\nu$\\
$X_n(t) = X(nt)$\\
$\bar X_n(t) = \frac1n X_n(t) - ta - \mu_1^+\nu_1^+t$\qquad or \qquad$\bar X_n(t) = \frac1n X_n(t) - ta - (\mu_1^+\nu_1^+-\mu_1^-\nu_1^-)t$

\item
$\I_{<j,k} = \{(l,m) \in \mathbb Z_+^2 \setminus \{(j,k)\}: (\alpha-1)l + (\beta-1)m \leq (\alpha-1)j + (\beta-1)k\}$
\\
$\I_{= j,k} = \{(l,m)\in \mathbb Z_+^2: (\alpha-1)l+(\beta-1)m = (\alpha-1)j+(\beta-1)k\}$
\\
$\I_{\ll j,k} = \{(l,m)\in \mathbb Z_+^2: (\alpha -1)l + (\beta-1)m < (\alpha-1)j+(\beta-1)k\}$
\\
$\I_{<(j_1,\ldots,j_d)}
= \big\{(l_1,\ldots,l_d)\in \mathbb Z_+^d\setminus\{(j_1,\ldots,j_d)\}: (\alpha_1-1)l_1+\cdots+(\alpha_d-1)l_d \leq (\alpha_1-1)j_1+\cdots+(\alpha_d-1)j_d\big\}
$
\\
$J_{>(j_1,\ldots,j_d)}
=
\big\{(l_1,\ldots,l_d)\in \mathbb Z_+^2: (\alpha_1-1)l_1+\cdots+(\alpha_d-1)l_d > (\alpha_1-1)j_1+\cdots+(\alpha_d-1)j_d\big\} \cup \{(j_1,\ldots, j_d)\}
$

\item
$\mathbb R_+$: set of non-negative real numbers\\
$\mathbb Z_+$: set of non-negative integers
\item
$\mathbb R_+^{\infty\downarrow} = \{x \in \mathbb R_+^\infty: x_1 \geq x_2 \geq \ldots\}$\\
$\mathbb R_+^{j\downarrow} = \{x \in \mathbb R_+^j: x_1 \geq x_2 \geq \ldots \geq x_j\}$\\
$\mathbb H_{j} = \{x \in \mathbb R_+^{\infty\downarrow}: x_j > 0, x_{j+1} = 0\}$\\
$\mathbb H_{\leqslant j} = \{x\in \mathbb R_+^{\infty \downarrow}: x_{j+1} = 0\}$
\item
$\mathbb D = \mathbb D([0,1],\R)$: real-valued RCLL functions on $[0,1]$\\
$\mathbb D_s^\uparrow$: subspace of $\mathbb D$ consisting of non-decreasing step functions vanishing at 0 \\
$\mathbb D_{j}$: subspace of $\mathbb D$ consisting of non-decreasing step functions vanishing at 0 with $j$ jumps\\
$\mathbb D_{j,k}$: subspace of $\mathbb D$ consisting of step functions vanishing at 0 with $j$ upward jumps and $k$ downward jumps\\
$\mathbb D_{\leqslant j} = \bigcup_{0\leqslant l\leqslant j} \mathbb D_{l}$ \\
$\mathbb D_{< j} = \bigcup_{0\leqslant l< j} \mathbb D_{l}$ \\
$\mathbb D_{<j,k} = \bigcup_{(l,m)\in \I_{<j,k}} \mathbb D_{l,m}$\\
$\mathbb D_{<(j,k)} = \bigcup_{(l,m)\in \I_{<j,k}} \mathbb D_l\times \mathbb D_m$
\\
$\mathbb D_{= j,k} = \bigcup_{(l,m)\in \I_{= j,k}} \mathbb D_{l,m}$\\
$\mathbb D_{\ll j,k} = \bigcup_{(l,m)\in \I_{\ll j,k}} \mathbb D_{l,m}$

\item
$\mathbb C_{(l_1,\ldots,l_d)}
= \bigcup_{i=1}^d  (\mathbb D^{i-1} \times \mathbb D_{<l_i} \times \mathbb D^{d-i})
$

\item $d$: Skorokhod metric on $\mathbb D([0,1],\R)$

\item
$\hat S_j = \{(x,u)\in \R_+^{j\downarrow}\times [0,1]^j: 0, 1, u_1,\ldots, u_j \text{ are all distinct}\}$\\
$S_j = \{(x,u)\in \R_+^{\infty\downarrow}\times [0,1]^\infty: 0, 1, u_1,\ldots, u_j \text{ are all distinct}\}$\\
$\hat T_j: \hat S_j \rightarrow \mathbb D_j$ defined by $\hat T_j(x,u) = \sum_{i=1}^j x_i 1_{[u_i,1]}$ \\
$T_m: S_m\to \mathbb D$ defined by $T_m(x,u) = \sum_{i=1}^m x_i 1_{[u_i,1]}$

\item $\nu_\alpha(x,\infty) = x^{-\alpha}$\\
$\nu_\alpha^j$: restriction (to $\mathbb R_+^{j\downarrow}$) of $j$-fold product measure of $\nu_\alpha$

\item $U_i$, $V_i$: i.i.d.\ uniform random variables on $[0,1]$
\item $C_j(\cdot) = \E\left[\nu_\alpha^j\{y\in (0,\infty)^j: \sum_{i=1}^j y_i1_{[U_i,1]}\in \cdot\}\right]$
\\
$C_{j,k}(\cdot) = \E \Big[\nu_\alpha ^j\times\nu_\beta^k \{(x,y)\in (0,\infty)^j\times(0,\infty)^k:\sum_{i=1}^j x_i 1_{[U_i,1]} - \sum_{i=1}^k y_i1_{[V_i,1]}\in \cdot\}\Big]$


\item $\nu_1^+ = \nu[1, \infty)$
\\
$\mu_1^+ =  \frac{1}{\nu_1^+}\int_{[1,\infty)}x\nu(dx)$
\\
$\nu_1^- = \nu(- \infty, -1]$
\\
$\mu_1^- =  \frac{1}{\nu_1^-}\int_{(-\infty,-1]}x\nu(dx)$

\item 
$\mathcal D_+(\xi)$: number of upward jumps of $\xi\in \mathbb D$\\
$\mathcal D_-(\xi)$: number of downward jumps of $\xi\in \mathbb D$

\item $\mathcal J(A)= \inf_{\xi\in \mathbb D_s^\uparrow \cap A}\mathcal D_+(\xi)$
\item $\mathcal I(j,k) = (\alpha-1)j + (\beta-1)k$
\item $(\mathcal J(A), \mathcal K(A)) = \argmin_{\substack{(j,k)\in \mathbb Z_+^2\\ \mathbb D_{j,k} \cap A \neq \emptyset}}\mathcal I(j,k)$

\item
$I(\xi)\triangleq
\left\{\begin{array}{ll}
	(\alpha-1)\mathcal D_+(\xi) + (\beta-1)\mathcal D_-(\xi) & \text{if $\xi$ is a step function with $\xi(0) = 0$}
	\\
	\infty & \text{otherwise}
\end{array}
\right.
$

\item
$\delta_{(x,y)}$:  Dirac measure concentrated on $(x,y)$

\item
$
Q_n(x) =  n\nu[x,\infty)
$\\
$
Q_n^{\gets}(y) = \inf\{s>0: n\nu[s,\infty)< y\}
$\\
$
\tilde N_n  = N_n\big([0,1]\times [1,\infty)\big)
$\\
$
N_n = \textstyle{\sum_{l=1}^\infty} \delta_{(U_l, Q_n^{\gets}(\Gamma_l))}
$

\item
$J_n(s)
= \sum_{l=1}^{\tilde N_n} Q_n^\gets (\Gamma_l)1_{[U_l,1]}(s)
\stackrel{\mathcal D}{=} \int_{x>1} xN([0,ns]\times dx)$
\\
$\Gamma_l = E_1 + E_2 + ... + E_l$
\\
$E_i$'s are i.i.d.\ standard exponential random variables\\
$U_l$'s are i.i.d.\ uniform variables in $[0,1]$
\item
$
\bar J_n = \textstyle{\frac1n \sum_{l=1}^{\tilde N_n}} (Q_n^\gets (\Gamma_l)-\mu_1^+)1_{[U_l,1]}
$\\
$
\hat J_n^{\leqslant j} = \textstyle{\frac1n \sum_{l=1}^{j}} Q_n^\gets (\Gamma_l)1_{[U_l,1]}
$\\
$
\check J_n^{\leqslant j} = \textstyle{\frac1n\sum_{l=1}^{j}} -\mu_1^+1_{[U_l,1]},
$\\
$\bar J_n^{> j} = \textstyle{\frac1n \sum_{l=j+1}^{\tilde N_n}} (Q_n^\gets (\Gamma_l)-\mu_1^+)1_{[U_l,1]},
$\\
$
\bar R_n^+ = \textstyle{\frac1n \I(\tilde N_n < j)\sum_{l=\tilde N_n+1}^j} (Q_n^\gets(\Gamma_l) - \mu_1^+)1_{[U_l,1]},
$

\end{itemize}

}

\bibliographystyle{apalike}      
\bibliography{RheeBlanchetZwart16_ref170224}

\begin{thebibliography}{}

\bibitem[Asmussen and Pihlsg{\aa}rd, 2005]{AsmussenPihlsgaard}
Asmussen, S. and Pihlsg{\aa}rd, M. (2005).
\newblock Performance analysis with truncated heavy-tailed distributions.
\newblock {\em Methodol. Comput. Appl. Probab.}, 7(4):439--457.

\bibitem[Barles, 1985]{Barles}
Barles, G. (1985).
\newblock Deterministic impulse control problems.
\newblock {\em SIAM J. Control Optim.}, 23(3):419--432.

\bibitem[Billingsley, 2013]{billingsley2013convergence}
Billingsley, P. (2013).
\newblock {\em Convergence of probability measures}.
\newblock John Wiley \& Sons.

\bibitem[Blanchet and Shi, 2012]{blanchetrisk}
Blanchet, J. and Shi, Y. (2012).
\newblock Measuring systemic risks in insurance - reinsurance networks.
\newblock {\em Preprint}.

\bibitem[Borovkov and Borovkov, 2008]{BorovkovBorovkov}
Borovkov, A.~A. and Borovkov, K.~A. (2008).
\newblock {\em Asymptotic analysis of random walks: Heavy-tailed
  distributions}.
\newblock Number 118. Cambridge University Press.

\bibitem[Buraczewski et~al., 2013]{BuraDamekMikosch13}
Buraczewski, D., Damek, E., Mikosch, T., and Zienkiewicz, J. (2013).
\newblock Large deviations for solutions to stochastic recurrence equations
  under {K}esten's condition.
\newblock {\em Ann. Probab.}, 41(4):2755--2790.

\bibitem[Chen et~al., 2017]{chen2017efficient}
Chen, B., Blanchet, J., Rhee, C.-H., and Zwart, B. (2017).
\newblock Efficient rare-event simulation for multiple jump events in regularly
  varying random walks and compound poisson processes.
\newblock {\em arXiv preprint arXiv:1706.03981}.

\bibitem[Cont and Tankov, 2004]{cont2004financial}
Cont, R. and Tankov, P. (2004).
\newblock {\em Financial modelling with jump processes}.
\newblock Chapman \& Hall.

\bibitem[de~Haan and Lin, 2001]{DeHaanLin01}
de~Haan, L. and Lin, T. (2001).
\newblock On convergence toward an extreme value distribution in {$C[0,1]$}.
\newblock {\em Ann. Probab.}, 29(1):467--483.

\bibitem[Dembo and Zeitouni, 2009]{dembozeitouni}
Dembo, A. and Zeitouni, O. (2009).
\newblock {\em Large deviations techniques and applications}, volume~38.
\newblock Springer Science \& Business Media.

\bibitem[Denisov et~al., 2008]{DiekerDenisovShneer}
Denisov, D., Dieker, A., and Shneer, V. (2008).
\newblock Large deviations for random walks under subexponentiality: the
  big-jump domain.
\newblock {\em The Annals of Probability}, 36(5):1946--1991.

\bibitem[Durrett, 1980]{durrett1980conditioned}
Durrett, R. (1980).
\newblock Conditioned limit theorems for random walks with negative drift.
\newblock {\em Zeitschrift f{\"u}r Wahrscheinlichkeitstheorie und Verwandte
  Gebiete}, 52(3):277--287.

\bibitem[Embrechts et~al., 1979]{EmbrechtsVeraverbeke}
Embrechts, P., Goldie, C.~M., and Veraverbeke, N. (1979).
\newblock Subexponentiality and infinite divisibility.
\newblock {\em Z. Wahrsch. Verw. Gebiete}, 49(3):335--347.

\bibitem[Embrechts et~al., 1997]{EmbrechtsKluppelbergMikosch97}
Embrechts, P., Kl{\"u}ppelberg, C., and Mikosch, T. (1997).
\newblock {\em Modelling extremal events}, volume~33 of {\em Applications of
  Mathematics (New York)}.
\newblock Springer-Verlag, Berlin.
\newblock For insurance and finance.

\bibitem[Foss et~al., 2007]{FossModulated}
Foss, S., Konstantopoulos, T., and Zachary, S. (2007).
\newblock Discrete and continuous time modulated random walks with heavy-tailed
  increments.
\newblock {\em Journal of Theoretical Probability}, 20(3):581--612.

\bibitem[Foss and Korshunov, 2012]{FK12}
Foss, S. and Korshunov, D. (2012).
\newblock On large delays in multi-server queues with heavy tails.
\newblock {\em Mathematics of Operations Research}, 37(2):201--218.

\bibitem[Foss et~al., 2011]{FKZ}
Foss, S., Korshunov, D., and Zachary, S. (2011).
\newblock {\em An introduction to heavy-tailed and subexponential
  distributions}.
\newblock Springer.

\bibitem[Gantert, 1998]{Gantert98}
Gantert, N. (1998).
\newblock Functional erd{\H{o}}s-renyi laws for semiexponential random
  variables.
\newblock {\em The Annals of Probability}, 26(3):1356--1369.

\bibitem[Gantert, 2000]{Gantert00}
Gantert, N. (2000).
\newblock The maximum of a branching random walk with semiexponential
  increments.
\newblock {\em Ann. Probab.}, 28(3):1219--1229.

\bibitem[Gantert et~al., 2014]{Gantert14}
Gantert, N., Ramanan, K., and Rembart, F. (2014).
\newblock Large deviations for weighted sums of stretched exponential random
  variables.
\newblock {\em Electron. Commun. Probab.}, 19:no. 41, 14.

\bibitem[Hult and Lindskog, 2005]{HultLindskog05}
Hult, H. and Lindskog, F. (2005).
\newblock Extremal behavior of regularly varying stochastic processes.
\newblock {\em Stochastic Process. Appl.}, 115(2):249--274.

\bibitem[Hult and Lindskog, 2006]{HultLindskog06}
Hult, H. and Lindskog, F. (2006).
\newblock Regular variation for measures on metric spaces.
\newblock {\em Publ. Inst. Math. (Beograd) (N.S.)}, 80(94):121--140.

\bibitem[Hult and Lindskog, 2007]{HultLindskog07}
Hult, H. and Lindskog, F. (2007).
\newblock Extremal behavior of stochastic integrals driven by regularly varying
  {L}\'evy processes.
\newblock {\em Ann. Probab.}, 35(1):309--339.

\bibitem[Hult et~al., 2005]{HLMS}
Hult, H., Lindskog, F., Mikosch, T., and Samorodnitsky, G. (2005).
\newblock Functional large deviations for multivariate regularly varying random
  walks.
\newblock {\em The Annals of Applied Probability}, 15(4):2651--2680.

\bibitem[Konstantinides and Mikosch, 2005]{KonstantinidesMikosch05}
Konstantinides, D.~G. and Mikosch, T. (2005).
\newblock Large deviations and ruin probabilities for solutions to stochastic
  recurrence equations with heavy-tailed innovations.
\newblock {\em Ann. Probab.}, 33(5):1992--2035.

\bibitem[Kyprianou, 2014]{kyprianou2014fluctuations}
Kyprianou, A.~E. (2014).
\newblock {\em Fluctuations of L{\'e}vy processes with applications:
  Introductory Lectures}.
\newblock Springer Science \& Business Media.

\bibitem[Lindskog et~al., 2014]{LRR}
Lindskog, F., Resnick, S.~I., and Roy, J. (2014).
\newblock Regularly varying measures on metric spaces: Hidden regular variation
  and hidden jumps.
\newblock {\em Probability Surveys}, 11:270--314.

\bibitem[Mikosch and Samorodnitsky, 2000]{MikoschSamorodnitsky00}
Mikosch, T. and Samorodnitsky, G. (2000).
\newblock Ruin probability with claims modeled by a stationary ergodic stable
  process.
\newblock {\em Ann. Probab.}, 28(4):1814--1851.

\bibitem[Mikosch and Wintenberger, 2013]{MikoschWintenberger13}
Mikosch, T. and Wintenberger, O. (2013).
\newblock Precise large deviations for dependent regularly varying sequences.
\newblock {\em Probab. Theory Related Fields}, 156(3-4):851--887.

\bibitem[Mikosch and Wintenberger, 2016]{mikosch2016large}
Mikosch, T. and Wintenberger, O. (2016).
\newblock A large deviations approach to limit theory for heavy-tailed time
  series.
\newblock {\em Probability Theory and Related Fields}, 166(1-2):233--269.

\bibitem[Nagaev, 1969]{Nagaev69}
Nagaev, A.~V. (1969).
\newblock Limit theorems that take into account large deviations when
  {C}ram\'er's condition is violated.
\newblock {\em Izv. Akad. Nauk UzSSR Ser. Fiz.-Mat. Nauk}, 13(6):17--22.

\bibitem[Nagaev, 1977]{Nagaev77}
Nagaev, A.~V. (1977).
\newblock A property of sums of independent random variables.
\newblock {\em Teor. Verojatnost. i Primenen.}, 22(2):335--346.

\bibitem[Prokhorov, 1959]{Prohorov}
Prokhorov, Y.~V. (1959).
\newblock An extremal problem in probability theory.
\newblock {\em Theor. Probability Appl.}, 4:201--203.

\bibitem[Pyke, 1965]{pyke1965spacings}
Pyke, R. (1965).
\newblock Spacings.
\newblock {\em Journal of the Royal Statistical Society. Series B
  (Methodological)}, pages 395--449.

\bibitem[Samorodnitsky, 2004]{Samorodnitsky04}
Samorodnitsky, G. (2004).
\newblock Extreme value theory, ergodic theory and the boundary between short
  memory and long memory for stationary stable processes.
\newblock {\em Ann. Probab.}, 32(2):1438--1468.

\bibitem[Shneer and Wachtel, 2009]{Shneer2009}
Shneer, S. and Wachtel, V. (2009).
\newblock Heavy-traffic analysis of the maximum of an asymptotically stable
  random walk.
\newblock {\em arXiv preprint arXiv:0902.2185}.

\bibitem[Whitt, 1980]{whitt1980some}
Whitt, W. (1980).
\newblock Some useful functions for functional limit theorems.
\newblock {\em Mathematics of operations research}, 5(1):67--85.

\bibitem[Zwart et~al., 2004]{ZBM}
Zwart, B., Borst, S., and Mandjes, M. (2004).
\newblock Exact asymptotics for fluid queues fed by multiple heavy-tailed
  on-off flows.
\newblock {\em Annals of Applied Probability}, pages 903--957.

\end{thebibliography}

\end{document}